\documentclass[12pt]{elsarticle}
\makeatletter
\def\ps@pprintTitle{%
   \let\@oddhead\@empty
   \let\@evenhead\@empty
   \let\@oddfoot\@empty
   \let\@evenfoot\@oddfoot
}
\makeatother

\usepackage{fullpage}
\usepackage[colorlinks]{hyperref}
\usepackage[utf8]{inputenc}
\usepackage{setspace}
\usepackage{tocloft}

\biboptions{square,sort,comma,numbers}
\usepackage{graphicx, titlesec}
\usepackage{amscd,amsmath,amsthm,amssymb}
\usepackage{verbatim}
\usepackage[all]{xy}
\usepackage{color}
\usepackage{hyperref}
\usepackage{tikz, float} \usetikzlibrary {positioning}
\usetikzlibrary{calc}

\usepackage{makeidx}         
\usepackage{graphicx}        
\usepackage{multicol}        
\usepackage[bottom]{footmisc}

\usepackage{array}

\usepackage{newtxmath}      

\makeindex

\newcommand{\BLambda}{{B_\ell}}
\newcommand{\id}{\text{id}}
\newcommand{\ini}{\text{in}}

\titleformat{\subsection}[runin]
{\normalfont\bfseries}{\thesubsection}{.5em}{}
\def\NZQ{\mathbb}               

\def\RR{{\NZQ R}}

\def\frk{\mathfrak}               

\def\Phi{{\frk n}}
\def\Phi{{\frk N}}

\def\wb{{\mathbf w}}

\def\A{{\mathcal A}}

\def\opn#1#2{\def#1{\operatorname{#2}}} 
\opn\chara{char} \opn\length{\ell} \opn\pd{pd} \opn\rk{rk}
\opn\projdim{proj\,dim} \opn\injdim{inj\,dim} \opn\rank{rank}
\opn\depth{depth} \opn\grade{grade} \opn\height{height}
\opn\embdim{emb\,dim} \opn\codim{codim}

\opn\Tr{Tr} \opn\bigrank{big\,rank}
\opn\superheight{superheight}\opn\lcm{lcm}
\opn\trdeg{tr\,deg}
\opn\reg{reg} \opn\lreg{lreg} \opn\ini{in} \opn\lpd{lpd}
\opn\size{size} \opn\sdepth{sdepth}
\opn\link{link}\opn\fdepth{fdepth}\opn\lex{lex}
\opn\LM{LM}
\opn\LC{LC}
\opn\NF{NF}
\opn\Merge{Merge}
\opn\sgn{sgn}
\opn\type{type}
\opn\div{div} \opn\Div{Div} \opn\cl{cl} \opn\Pic{Pic}
\opn\Prin{Prin}
\opn\op{op}
\opn\indeg{indeg} \opn\outdeg{outdeg}
\opn\red{red}
\opn\Spec{Spec} \opn\Supp{Supp} \opn\supp{supp} \opn\Sing{Sing}
\opn\Ass{Ass} \opn\Min{Min}\opn\Mon{Mon} \opn\val{val}
\opn\Ann{Ann} \opn\Rad{Rad} \opn\Soc{Soc}
 \opn\Ker{Ker} \opn\Coker{Coker} \opn\Am{Am}
\opn\Hom{Hom} \opn\Tor{Tor} \opn\Ext{Ext} \opn\End{End}
\opn\Aut{Aut} \opn\id{id}

\opn\nat{nat}
\opn\pff{pf}
\opn\Pf{Pf} \opn\GL{GL} \opn\SL{SL} \opn\mod{mod} \opn\ord{ord}
\opn\Gin{Gin} \opn\Hilb{Hilb}\opn\sort{sort}
\opn\Image{Image}
\opn\vol{Vol}
\opn\aff{aff} \opn\con{conv} \opn\relint{relint} \opn\st{st}
\opn\lk{lk} \opn\cn{cn} \opn\core{core} \opn\vol{vol}
\opn\link{link} \opn\star{star}\opn\lex{lex}\opn\set{set}
\opn\dist{dist}
\opn\gr{gr}

\def\pot#1#2{#1[\kern-0.28ex[#2]\kern-0.28ex]}

\opn\dirlim{\underrightarrow{\lim}}
\opn\inivlim{\underleftarrow{\lim}}
\let\to=\rightarrow

\def\Implies{\ifmmode\Longrightarrow \else
        \unskip${}\Longrightarrow{}$\ignorespaces\fi}
\def\implies{\ifmmode\Rightarrow \else
        \unskip${}\Rightarrow{}$\ignorespaces\fi}
\def\iff{\ifmmode\Longleftrightarrow \else
        \unskip${}\Longleftrightarrow{}$\ignorespaces\fi}

\let\:=\colon

\newtheorem{theorem}{Theorem}[section]
\newtheorem{lemma}[theorem]{Lemma}
\newtheorem{corollary}[theorem]{Corollary}
\newtheorem{proposition}[theorem]{Proposition}

\newtheorem{conjecture}[theorem]{Conjecture}

\theoremstyle{remark}
\newtheorem{remark}[theorem]{Remark}

\theoremstyle{definition}
\newtheorem{example}[theorem]{Example}

\newtheorem{definition}[theorem]{Definition}

\DeclareMathOperator{\Gr}{Gr}
\DeclareMathOperator{\Flag}{Fl}

\let\kappa=\varkappa

\def\qed{\ifhmode\textqed\fi
      \ifmmode\ifinner\quad\qedsymbol\else\dispqed\fi\fi}
\def\textqed{\unskip\nobreak\penalty50
       \hskip2em\hbox{}\nobreak\hfil\qedsymbol
       \parfillskip=0pt \finalhyphendemerits=0}
\def\dispqed{\rlap{\qquad\qedsymbol}}

\opn\dis{dis}
\def\pnt{{\raise0.5mm\hbox{\large\bf.}}}

\opn\Lex{Lex}
\opn\syz{{\rm syz}}
\opn\spoly{{\rm spoly}}
\opn\LM{{\rm LM}}
\opn\lm{{\rm lm}}
\opn\lcm{{\rm lcm}} \opn\A{\mathcal A}

\numberwithin{equation}{section}

\newcommand{\inwb}{{\rm in}_{{\bf w}_\ell}}
\newcommand{\ul}[1]{\underline{#1}}

\DeclareMathOperator{\init}{in}

\usepackage{multirow}

\begin{document}

\begin{frontmatter}

\title{Standard monomial theory and toric degenerations of\\
Richardson varieties inside Grassmannians and flag varieties}
\author{
Narasimha Chary Bonala, Oliver Clarke and Fatemeh Mohammadi
}
\begin{abstract}
We study toric degenerations of opposite Schubert and Richardson varieties inside degenerations of Grassmannians and flag varieties. These degenerations are parametrized by matching fields in the sense of Sturmfels and Zelevinsky \cite{sturmfels1993maximal}. 
We construct so-called restricted matching field ideals whose generating sets are understood combinatorially through tableaux. We determine when these ideals are toric and coincide with Gr\"obner degenerations of Richardson varieties using the well established standard monomial theory for Grassmannians and flag varieties.
\end{abstract}
\end{frontmatter}
{
  \hypersetup{linkcolor=black}
 \setcounter{tocdepth}{1}
\setlength\cftbeforesecskip{0.1pt}
{\small\tableofcontents}}

\section{Introduction}

In this note we offer a new family of toric degenerations of Richardson varieties 
arising from Gr\"obner monomial degenerations of the Pl\"ucker ideal. This includes the well-studied diagonal and antidiagonal monomial degenerations as particular examples. We use combinatorial techniques and study permutations, semi-standard Young tableaux and standard monomial bases for Pl\"ucker ideals in order to construct these toric degenerations.

\subsection{Background. }  \hspace{.05mm}
The geometry of flag varieties heavily depends on the study of its Schubert varieties. For example, they provide an excellent way of understanding the multiplicative structure of the cohomology ring of the flag variety. In this context, it is essential to understand how Schubert varieties intersect in a general position. 
A Richardson variety in a flag variety is the intersection of a Schubert variety and an opposite Schubert variety.  In \cite{deodhar1985some} and  \cite{richardson1992intersections}, the fundamental properties of these varieties are studied, including their irreducibility.  
On the other hand, degeneration techniques play a significant role in understanding a given algebraic variety in terms of well-studied algebraic varieties like toric varieties. It is desirable
to extend the powerful machinery of toric varieties to a larger class of varieties by studying degenerations of a general variety to a toric variety. A toric degeneration of a variety $X$ is a flat family $f: \mathcal X \to  \mathbb A^{\!{1}}$, where the special fiber (the fiber over zero) is a toric variety and all other fibers are isomorphic to $X$. In a toric degeneration, some of the algebraic invariants of $X$ will be the same for all the fibers. Hence, we can do the computations on the toric fiber.

The study of toric degenerations of flag varieties was started in \cite{gonciulea1996degenerations} by Gonciulea and Lakshmibai using
standard monomial theory. In \cite{KOGAN}, Kogan and Miller obtained toric degenerations of flag varieties using geometric methods. Moreover, in \cite{caldero2002toric} Caldero constructed such toric degenerations using tools from representation theory. The article \cite{fang2017toric} by Fang, Fourier and Littelmann
contains more details on recent developments and provides an excellent overview of toric degenerations of flag varieties.We note that for Flag~4 and 5, toric degenerations obtained from Gr\"obner degeneration have been explicitly computed, see \cite{bossinger2017computing}.
In \cite{kim2015richardson}, the author studied the Gr{\"o}bner degenerations of Richardson varieties inside the flag variety, where the special fiber is the toric variety of the Gelfand-Tsetlin polytope; this is a generalization of the results of \cite{KOGAN}. In \cite{morier2008geometric}, Morier-Genoud obtained semi-toric degenerations of Richardson varieties by methods of Caldero in \cite{caldero2002toric}, namely the Berenstein-Littelmann-Zelevinsky string parametrization of canonical basis \cite{littelmann1998cones,berenstein1999tensor}.


\subsection{Our contributions.}
{
We study toric degenerations 
of Grassmannians and flag varieties arising from combinatorial objects called block diagonal matching fields introduced in \cite{KristinFatemeh} and generalized in \cite{OllieFatemeh, OllieFatemeh2}. Associated to each matching field is a weight vector that produces a one-parameter family, by Gr\"obner degeneration, where the special fiber is the initial ideal of the Pl\"ucker ideal. For Grassmannians and flag varieties, these initial ideals are called matching field ideals, whose generators can be understood combinatorially by tableaux. We use matching field ideals to study toric degenerations of various subvarieties of Grassmannians and flag varieties including Schubert, opposite Schubert and Richardson varieties. The defining ideals of these subvarieties are obtained from Pl\"ucker ideals by setting some particular variables to zero. Similarly, we define \textit{restricted matching field ideals} associated to Schubert and Richardson varieties that are variants of matching field ideals obtained by setting the same collection of variables to zero.
We aim to classify the toric, i.e. binomial and prime, restricted matching field ideals, which in this case is equivalent to showing that these ideals are monomial-free. We use techniques from standard monomial theory to construct monomial bases for restricted matching field ideals and initial ideals of Richardson varieties. We show that if a restricted matching field ideal is monomial-free then it is toric and coincides with the initial ideal of a Richardson, Schubert or opposite Schubert variety inside the Grassmannian or flag variety.} 
Our results generalizes the previous results on Schubert varieties from \cite{OllieFatemeh,OllieFatemeh3} and include, as a special case, the Gelfand-Tsetlin degeneration, also known as diagonal or antidiagonal Gr\"obner degeneration.

\smallskip
\noindent A similar approach has been taken in \cite{kim2015richardson}, where the author describes the semi-toric degenerations of Richardson varieties in flag varieties, where each Richardson variety is degenerated to a union of toric varieties.  We notice that the corresponding ideals of many such degenerations are either zero or contain monomials, hence their corresponding varieties are not toric.
Hence, we aim to characterize nonzero monomial-free ideals. In particular, we explicitly describe degenerations of Schubert, opposite Schubert and Richardson varieties inside Grassmannians and flag varieties, and provide a complete characterization for permutations leading to zero or monomial-free ideals.
For example, if the size of the index set in  \cite[Corollary V.26]{kim2015richardson} is one, then we
obtain a toric degeneration corresponding to a unique face of the Gelfand-Tsetlin polytope.
If the size of the index set is greater than one, our calculations show that some of these semi-toric degenerations are in fact toric. These cases arise when either the corresponding Schubert or opposite Schubert variety does not degenerate to a toric variety.
We also remark that the degenerations in \cite{kim2015richardson} correspond to the antidiagonal Gr\"obner degenerations inside flag varieties, however we consider a one-parametric degenerations of Richardson varieties inside both Grassmannians and flag varieties which contain the diagonal case (which is isomorphic to the antidiagonal case) as a special case.


\subsection{Outline of the paper.}
In Section~\ref{sec:prim} we fix our notation, and we recall the definitions of the main objects under study such as Grassmannians, flag, Schubert, opposite Schubert and Richardson varieties.  Sections~\ref{sec:gr} and ~\ref{sec:flag} contain our main results 
characterizing non-zero monomial-free \textit{restricted matching field ideals}, see Definitions~\ref{def:matching_field_ideal_grassmannian} and \ref{def:matching_field_ideal_flag}. In Figure~\ref{figure:Grassmannian_Flowchart} 
we provide a pictorial summary of some of our main results in which we show our inductive process to obtain zero and monomial-free ideals. In Table~\ref{table:flag_calculation} we summarise our computational results, namely the number of zero, toric and non-toric ideals  arising from our constructions. 
In Section~\ref{sec:kim} we perform calculations for $\Flag_3$ and $\Flag_4$ and compare our results to those in \cite{kim2015richardson}. 
In \S\ref{sec:standard_monomial} we study monomial bases of Richardson varieties and prove that if the restricted matching field ideal is monomial-free then it coincides with the initial ideal of the Richardson variety, see Theorems~\ref{thm:toric_degen_std_monomial_ell_0_gr}, \ref{thm:toric_degen} and Corollary~\ref{cor:toric_degen_std_monomial_gr}.

\medskip
\noindent{\bf Acknowledgement.} NC is supported by the SFB/TRR 191 ``Symplectic structures in Geometry, Algebra and Dynamics".
He gratefully acknowledges support from the Max Planck Institute for Mathematics in Bonn, and the EPSRC Fellowship EP/R023379/1 who supported his multiple visits to Bristol. 
OC was supported by EPSRC Doctoral Training Partnership (DTP) award EP/N509619/1.
FM was partially supported by EPSRC Early Career Fellowship EP/R023379/1 and BOF Starting Grant from Ghent University. This project began during the ``Workshop for Young Researchers"
in Cologne. We would like to thank the organizers of the meeting, Lara Bossinger and Sara Lamboglia, and in particular Stephane Launois for supporting the authors visit via the EPSRC grant EP/R009279/1.


\section{Preliminaries}\label{sec:prim}

Throughout we fix a field $\mathbb{K}$ with char$(\mathbb{K})=0$. We are mainly interested in the case when 
$\mathbb{K}=\mathbb{C}$, {the field of complex numbers}. We let $[n]$ be the set $\{1, \dots, n \}$ and by $S_n$ we denote the symmetric group on $[n]$. A permutation $w \in S_n$ is written $w = (w_1, \dots, w_n)$ where $w_i = w(i)$ for each $1 \le i \le n$. 
We fix 
$w_0:=(n, n-1, \ldots, 2, 1)$ for the longest product of adjacent transpositions in $S_n$. The permutations of $S_n$ act naturally on the left of subsets of $[n]$. So, for each $I= \{i_1, \ldots, i_k\} \subset [n]$, we have $w_0I =\{n+1-i_1,\ldots, n+1-i_k\}$ which is obtained by applying the permutation $w_0$ element-wise to $I$.
We now recall the definitions of Grasmmannian and flag varieties along with their Schubert, opposite Schubert and Richardson varieties.

\subsection{Flag varieties.}\label{sec:flag_def}  
A full flag is a sequence of vector subspaces of $\mathbb{K}^n$: $$\{0\}= V_0\subset V_1\subset\cdots\subset V_{n-1}\subset V_n=\mathbb{K}^n$$ where ${\rm dim}_{\mathbb{K}}(V_i) = i$. The set of all full flags is called the flag variety denoted by $\Flag_n$, which is naturally embedded  in a product of Grassmannians using the Pl\"ucker variables.
Each point in the flag variety can be represented by an $n\times n$ matrix $X=(x_{i,j})$ whose first $k$ rows span $V_k$. Each $V_k$ corresponds to a point in the Grassmannian $\Gr(k,n)$. The ideal of $\Flag_n$, denoted by $I_n$ is the kernel of the polynomial map
\[
\varphi_n:\  \mathbb{K}[P_J:\ \varnothing\neq J\subsetneq \{1,\ldots,n\}]\rightarrow \mathbb{K}[x_{i,j}:\ 1\leq i\leq n-1,\ 1\leq j\leq n]
\]
sending each variable
$P_J$ to the determinant of the submatrix of $X$ with row indices $1,\ldots,|J|$ and column indices in $J$. We call the variables $P_J$ of the ring  Pl\"ucker variables and their images $\varphi_n(P_J)$ Pl\"ucker forms. 
For each $\alpha=(\alpha_J)_{J}$ in $\mathbb{Z}_{\geq 0}^{{n\choose k}}$ we fix the notation ${\bf P}^{{\bf \alpha}}$ denoting the monomial $\prod_{J}P_J^{\alpha_J}$.
Similarly, we define the {\em Pl\"ucker ideal} of $\Gr(k,n)$, denoted by $G_{k,n}$ as the kernel of the map $\varphi_n$ restricted to the ring with variables $P_J$ with $|J|=k$.

\subsection{Schubert varieties.} \label{sec:schubert_def}
Let SL$(n,\mathbb K)$ be the group of complex $n\times n$ matrices with determinant $1$,  and let $B$ be its subgroup consisting of upper triangular matrices.  There is a natural transitive action of SL$(n, \mathbb K)$ on the flag variety $\Flag_n$ which identifies $\Flag_n$ with the set of left cosets SL$(n, \mathbb K)/B$, since $B$ is the stabiliser of the standard flag
$0\subset \langle e_1\rangle \subset\cdots \subset \langle e_1, \ldots, e_n\rangle=\mathbb K^n $. Given a permutation $w\in S_n$, we denote by $\sigma_w$ the $n\times n$ permutation matrix with 1's in positions $(w(i),i)$ for all $i$.
By the Bruhat decomposition, we can write the aforementioned set of cosets as 
${\rm SL}(n,\mathbb K)/B= \coprod_{w\in S_n}B\sigma_wB/B.$
Given a permutation $w$, its Schubert variety is $$X_w=\overline{B\sigma_wB/B} \subseteq \Flag_n$$ which is the Zariski closure of the corresponding cell $B\sigma_wB/B$ in the Bruhat decomposition. 
The ideal of $X_w$ is obtained from $I_n$ by setting $P_I$ to zero for each $I \in S_w$
where
$$
S_w=\{I : I\subset[n] \ \text{with} \ I=(i_1,\ldots,i_{|I|})\not\leq (w_{\ell_1},w_{\ell_2},\ldots,w_{\ell_{|I|}})\}
$$
and $w_{\ell_1}<w_{\ell_2}<\cdots<w_{\ell_{|I|}}$ is obtained by ordering the first $\lvert I \rvert$ entries of $w$. Note that $\{a_1 < \dots < a_m \} \le \{b_1 < \dots < b_m \}$ means that $a_i \le b_i$ for each $1 \le i \le m$. See Example~\ref{def:sets}.

\smallskip

Similarly, we can study the Schubert varieties inside $\Gr(k,n)$. The permutations giving rise to distinct Schubert varieties in $\Gr(k,n)$ are of the form $w=(w_1,\ldots,w_n)$ 
where  $w_{1}<\cdots<w_k$, $w_{k+1}<\cdots<w_n$. Therefore, it is enough to record the permutations of $S_n$ as $w=(w_1,\ldots,w_k)$ which will be called a {\em Grassmannian permutation}. Suppose that $w$ is a Grassmannian permutation, we define $S_{w,k}=S_w\cap\{I:\ |I|=k\}$. The elements in $S_{w,k}$ correspond to the variables which vanish in the ideal of the Schubert variety of the Grassmannian, see Definition~\ref{def:S}.


\subsection{Opposite Schubert varieties.}
Similar to the Schubert case, $\Flag_n$ can be decomposed as $\Flag_n=\coprod_{v\in S_n}B^-\sigma_vB/B$ where $B^-$ is subgroup of ${\rm SL}(n, \mathbb K)$ of lower triangular matrices.
For $v\in S_n$, we define its opposite Schubert variety as $$X^v=\overline{B^-\sigma_vB/B}\subset \Flag_n.$$
Let $w_0$ be the permutation $(n, n-1, \ldots, 2, 1)$. As $B^-=\sigma_{w_0}B\sigma_{w_0}$, we can observe the opposite Schubert variety $X^v$ is a translate $w_0X_{w_0v}$ of the Schubert variety $X_{w_0v}$ in $\Flag_n$.

\subsection{Richardson varieties.}
Let $v, w\in S_n$. We denote the corresponding Schubert and opposite Schubert varieties by $X_w$ and $X^v$ respectively. Then the Richardson variety $X_w^v$ is defined as $X_w\cap X^v$.  We recall that $X_w^v\neq \emptyset$ if and only if $v\leq w$ (see for example \cite{brion2003geometric}).
We set 
$$
T_w^v=\{I : I\subset[n] \ \text{with} \ (v_{k_1},v_{k_2},\ldots,v_{k_{|I|}})\leq I=(i_1,\ldots,i_{|I|})\leq (w_{\ell_1},w_{\ell_2},\ldots,w_{\ell_{|I|}})\},
$$
where $w_{\ell_1}<w_{\ell_2}<\cdots<w_{\ell_{|I|}}$ and $v_{k_1}<v_{k_2}<\cdots<v_{k_{|I|}}$ are obtained by ordering the first $\lvert I \rvert$ entries of $v$ and $w$. 
The associated ideal of the Richardson variety $X_w^v$ is
{
$$I(X_w^v) = (I_n + \langle P_I: I\in S_w^v\rangle) \cap \mathbb K[P_I : I \in T_w^v] = (I(X_w)+I(X^v)) \cap \mathbb K[P_I : I \in T_w^v],$$
}
where $I_n$ is the Pl\"ucker ideal of $\Flag_n$ and $S_w^v=\{I:\ \emptyset\subset I\subset [n]\}\backslash T_{w}^v $. Similarly, the associated ideal of the Richardson variety $X_w^v$ inside Grassmannian $\Gr(k,n)$ is
{
$$(G_{k,n}+\langle P_I: I\in S_{w,k}^v\rangle) \cap \mathbb K[P_I : I \in T_{w,k}^v],$$ where $T_{w,k}^v = \{I \subseteq [n] : |I| = k \} \backslash S_{w,k}^v$.
}
We now give an example of variables and refer the reader to \cite[\S 3.4.]{kreiman2002richardson} for more details.

\begin{example}\label{def:sets}
Let $n = 4$. Consider the permutations $v = (2314)$ and $w = (4231)$. The subsets of $[n]$ in $T^v_w$ of size one are given by those entries that lie between $v_1 = 2$ and $w_1 = 4$, which are $2, 3$ and $4$. The subsets of size two are those that lie between $\{v_1, v_2 \} = 23$ and $ \{w_1, w_2 \} = 24$ which are $23$ and $24$. The subsets of size three are those which lie between $\{v_1, v_2, v_3 \} = 123$ and $\{ w_1, w_2, w_3\} = 234$ which are all possible three-subsets. So the sets $T^v_w$ and $S^v_w$ are,
\[
{T}^{(2314)}_{(4231)} = \{2,3,4, 23,24,123,124,134,234 \} \text{ and } S^{(2314)}_{(4231)} = \{1, 12, 13, 14, 34\}.
\]
Note that $S^{(2314)}_{(4231)} = S^{(2314)}\cup S_{(4231)}$ where 
\[
S^{(2314)} = \{1, 12, 13, 14\} \text{ and } S_{(4231)} = \{34\}.
\]
\end{example}

\subsection{Matching fields.}\hspace{0.3mm}\label{sec:prelim_matching_field}
{
A matching field is a combinatorial object that encodes a weight vector for the ring $\mathbb K[P_J]$. These weight vectors are induced from a weight on the ring $\mathbb K[x_{i,j}]$, see Definition~\ref{example:block}. The aim of this section is to define a particular family of matching fields called block diagonal matching fields. We will show that these matching fields give rise to toric degenerations of Richardson varieties inside  Grassmannians and flag varieties.
}

{
A matching field also admits the data of a monomial map $\mathbb K[P_J] \rightarrow \mathbb K[x_{i,j}]$ which takes each $P_J$ to a term of the Pl\"ucker form $\varphi_n (P_J) = \det(X_J) \in \mathbb K[x_{i,j}]$. The ideal of the matching field is the kernel of this monomial map, which is a toric (prime binomial) ideal. Whenever a block diagonal matching field gives rise to a toric initial ideal, the corresponding toric ideal equals to the matching field ideal.
}

\begin{definition}\label{example:block}
Given integers $k,n$ and 
$0\leq \ell\leq n$, we fix the $k \times n$ matrix $M_\ell$ with entries:
\[
M_\ell(i,j)=\begin{cases} 
      (i-1)(n-j+1) & \textrm{if } i\neq 2, \\
      \ell-j+1 & \textrm{if } i=2, 1\leq j\leq \ell,\\
      n-j+\ell+1 & \textrm{if } i=2, \ell< j\leq n.\\
   \end{cases}
\]
Let $X=(x_{i,j})$ be a $k \times n$ matrix of indeterminates. For each $k$-subset $J$ of $[n]$, the initial term of the Pl\"ucker form $\varphi_n(P_J) \in \mathbb{K}[x_{ij}]$ denoted by $\ini_{M_\ell}(P_J)$ is the sum of all terms in $\varphi_n(P_J)$ of the lowest weight, where the weight of a monomial $\bf m$ is the sum of entries in $M_\ell$ corresponding to the variables in $\bf m$. We denote $M_\ell(\textbf{m})$ for the weight of $\textbf{m}$. We prove below that $\ini_{M_\ell}(P_J)$ is a monomial for each subset $J \subseteq [n]$.
The weight of each variable $P_J \in \mathbb K[P_J]$ is defined as the weight of each term of $\textrm{in}_{M_\ell}(P_J)$ with respect to $M_\ell$, and it is called {\em the weight induced by $M_\ell$}. We denote $\wb_\ell$ for the weight vector induced by $M_\ell$.
\end{definition}

\begin{lemma}\label{lem:wt_add_const}
{
Let $M = (m_{i,j})$ and $M' = (m'_{i,j})$ be $k \times n$ weight matrices. Suppose there exists $p \in \{ 1, \dots, n\}$ such that $m_{i,j} = m'_{i,j}$ for all $i \in [k]$ and $j \in [n] \backslash p$. Suppose that there exists $c \in \RR$ such that $m'_{i,p} = m_{i,p} + c$ for all $i \in \{1, \dots, k \}$. Then the initial terms of the Pl\"ucker forms are the same with respect to $M$ and $M'$.}
\end{lemma}

\begin{proof}
{
Let $J$ be a $k$-subset of $[n]$. If $J$ does not contain $p$ then the submatrices of $M$ and $M'$ with columns indexed by $J$ coincide, hence the initial terms of the Pl\"ucker form $\varphi_n(P_J)$ with respect to $M$ and $M'$ are the same. On the other hand, if $J$ contains $p$ then consider each monomial $\textbf{x}$ in the Pl\"ucker form $\varphi_n(P_J)$. The monomial is squarefree and contains a unique variable of the form $x_{i,p}$ for some $i \in \{1, \dots, k \}$. Therefore $M'(\textbf{x}) = M(\textbf{x}) + c$. Therefore the initial term of $\varphi_n(P_J)$ is the same with respect to $M$ and $M'$.}
\end{proof}

{By the same method, one can prove an analogous result for weight matrices which differ by a constant in a particular row.}

\begin{proposition}\label{prop:unique}
{
Let $M = M_0$ be the $k \times n$ weight matrix and $J = \{j_1 < \dots < j_k \} \subset [n]$ be a $k$-subset. Then the initial term $\textrm{in}_M(P_J)$ of the Pl\"ucker form is the leading diagonal term, in particular it is a monomial.}
\end{proposition}

\begin{proof}
{
We show that the leading diagonal term of the Pl\"ucker form $\varphi_n(P_J)$ is the initial term $\textrm{in}_M(P_J)$. We proceed by induction on $k$. If $k = 1$ then the result holds trivially. So assume $k > 1$. We have
\[
\varphi_n(P_J) = \sum_{\sigma \in S_k} x_{1,j_{\sigma(1)}} \cdots x_{k,j_{\sigma(k)}}.
\]
For each $\sigma \in S_k$ such that $x_{1,j_{\sigma(1)}} \cdots x_{k,j_{\sigma(k)}}$ has minimum weight with respect to $M$, consider the value $\sigma(k) \in [k]$. 
Suppose $\sigma(k) = p$ for some $p \in [k]$. Then, by induction, we have that the leading term of the $\varphi_n(P_{J \backslash j_p})$ is the leading diagonal term. So $\sigma(1) = 1, \dots, \sigma(p-1) = {p-1}, \sigma(p) = {p+1}, \dots, \sigma(k-1) = k$ and $\sigma(k) = p$, therefore the weight of the monomial is
\begin{align*}
M(x_{1,j_{\sigma(1)}} \cdots x_{k,j_{\sigma(k)}}) &= 
\sum_{i = 1}^{p-1} (i-1)(n - j_i +1) +
\sum_{i = p+1}^k (i-2)(n - j_i + 1)
+ (k - 1)(n - j_p + 1) \\
&= \sum_{i = 1}^{k} (i-1)(n - j_i +1) - \sum_{i = p+1}^k (n-j_i + 1) + (k-p)(n - j_p +1) \\
&= M(x_{1,j_1} \cdots x_{k,j_k}) + \sum_{i = p+1}^k j_i - (k - p)j_p
\end{align*}
Note that $j_p < j_{p+1} < \dots < j_k$ and so $\sum_{i = p+1}^k j_i - (k - p)j_p > 0$. Therefore, if $\sigma(k) < k$ then the weight $M(x_{1,j_{\sigma(1)}} \cdots x_{k,j_{\sigma(k)}})$ is not minimum. Hence $\sigma(k) = k$ and we are done by induction.}
\end{proof}

\begin{proposition}\label{prop:unique_ell}
{
Let $\ell \in \{1, \dots, n-1 \}$, $M = M_\ell$, the $k \times n$ weight matrix, and $J = \{j_1 < \dots < j_k \} \subset [n]$ be a $k$-subset. Suppose $k \ge 2$, then the initial term of the Pl\"ucker form $\varphi_n(P_J)$ is given by
\[
\textrm{in}_M(P_J) = \left\{
\begin{tabular}{@{}ll}
     $x_{1, j_1}x_{2, j_2}x_{3, j_3}\dots x_{k, j_k}$ & if $j_1 > \ell$ or $j_2 \le \ell$, \\
     $x_{1, j_2}x_{2, j_1}x_{3, j_3} \dots x_{k, j_k}$ & otherwise. 
\end{tabular}
\right.
\]
In particular the leading term is a monomial.
}
\end{proposition}

\begin{proof}
{
Suppose that $j_1 > \ell$. By definition, the weight matrices $M_\ell$ and $M_0$ differ only in the second row. The entries of the second row are
\begin{align*}
(M_0):      & \quad [n \ n-1 \ \dots \ 1], \\
(M_\ell) :  & \quad [\ell \ \ell-1 \ \dots \ 1 \ n \ n-1 \ \dots \ \ell+1].
\end{align*}
Consider the submatrices of $M_0$ and $M_\ell$ consisting of the columns indexed by $J$. Since $j_1 > \ell$ the second row entries differ by exactly $\ell$ in each respective position. And so by the row-version of Lemma~\ref{lem:wt_add_const}, the leading term of the Pl\"ucker form $\varphi_n(P_J)$ is the same with respect to $M_0$ and $M_\ell$. By Proposition~\ref{prop:unique}, the initial term is $\textrm{in}_M(P_J) = x_{1, j_1}x_{2, j_2}x_{3, j_3}\dots x_{k, j_k}$.
}

{
Suppose that $j_1 \le \ell$. 
We will prove the result by induction on $k$. For the case $k = 2$ we have
\[
M_\ell = 
\begin{bmatrix}
0    & 0      & \dots & 0 & 0 & 0   & \dots & 0      \\
\ell & \ell-1 & \dots & 1 & n & n-1 & \dots & \ell+1 
\end{bmatrix}.
\]
If $j_1 > \ell$ or $j_2 \le \ell$ then the leading term of the Pl\"ucker form $\varphi_n(P_J)$ is the leading diagonal term, i.e. $\textrm{in}_M(P_J) = x_{1, j_1} x_{2, j_2}$. Otherwise we have $j_1 \le \ell$ and $j_2 > \ell$, and so the leading term of the Pl\"ucker form is the antidiagonal term, i.e. $\textrm{in}_M(P_J) = x_{1, j_2} x_{2, j_1}$.
}

{
Suppose $k > 2$. For each $\sigma \in S_k$ such that $x_{1, j_{\sigma(1)}} \dots x_{k, j_{\sigma(k)}}$ has minimum weight with respect to $M_\ell$, consider the value $p = \sigma(k) \in [k]$. Let $J' = J \backslash j_p = \{j_1' < j_2' < \dots < j_{k-1}' \}$. There are two cases for $J'$, either $j_2' \le \ell$ or $j_2' > \ell$.
}

\smallskip

{
\textbf{Case 1.} Assume $j_2' \le \ell$. By induction we have $\textrm{in}_M(P_{J'}) = x_{1, j_1'} x_{2, j_2'} \dots x_{k-1, j_{k-1}'}$. And so we have $\sigma(1) = 1, \dots, \sigma(p-1) = p-1, \sigma(p) = p+1, \dots, \sigma(k-1) = k, \sigma(k) = p$. Suppose by contradiction that $p \le k-1$, then we have
\[
    M_\ell(x_{1, j_{\sigma(1)}}  \dots x_{k, j_{\sigma(k)}}) - M_\ell(x_{1, j_1}  \dots x_{k, j_k}) 
    = \sum_{i = p}^k \left(M_\ell(x_{i, j_{\sigma(i)}}) - M_\ell(x_{i, j_i})\right)
    = \sum_{i = p}^k (i - 1)(j_i - j_{\sigma(i)})
\]
\[
    = \left(\sum_{i = p}^{k-1} (i-1)(j_i - j_{i+1})\right) + (k-1)(j_k - j_p)
    = \left(\sum_{i = p+1}^{k} j_i\right) + (p-k)j_p 
    = \sum_{i = p+1}^{k} (j_i - j_p) > 0.
\]
But by assumption $x_{1, j_{\sigma(1)}} \dots x_{k, j_{\sigma(k)}}$ has minimum weight, a contradiction. And so we have $p = k$ hence $\textrm{in}_M(P_J) = x_{1, j_1} x_{2, j_2} \dots x_{k, j_k}$.
}

\smallskip

{
\textbf{Case 2.} Assume $j_2' > \ell$. Either we have $j_1' \le \ell$ or $j_1' > \ell$. In this case assume further that $j_1' \le \ell$, we will show that $j_1' > \ell$ is impossible in Case 3. By induction we have $\textrm{in}_M(P_{J'}) = x_{1, j_2'} x_{2, j_1'} x_{3, j_3'} \dots x_{k-1, j_{k-1}'}$. Assume by contradiction that $k \neq p$. We proceed by taking cases on $p$, either $p = 1$, $p = 2$ or $3 \le p \le k - 1$.
}

{
\textbf{Case 2.1} Assume $p = 1$. So we have $\sigma(1) = 3, \sigma(2) = 2, \sigma(3) = 4, \dots, \sigma(k-1) = k, \sigma(k) = 1$. Since $j_p < j_1' \le \ell$, we have
\[
    M_\ell(x_{1, j_{\sigma(1)}}  \dots x_{k, j_{\sigma(k)}}) - M_\ell(x_{1, j_1} x_{2, j_2}  \dots x_{k, j_k}) 
    = \sum_{i = 1}^k \left(M_\ell(x_{i, j_{\sigma(i)}}) - M_\ell(x_{i, j_i})\right)
\]
\[
    = \sum_{i = 1}^k (i - 1)(j_i - j_{\sigma(i)})
    = \left(\sum_{i = 3}^{k-1} (i-1)(j_i - j_{i+1})\right) + (k-1)(j_k - j_1)
\]
\[
    = \left(\sum_{i = 4}^{k} j_i\right) + 2j_3 - (k-1)j_1
    = \left(\sum_{i = 4}^{k} (j_i - j_1)\right) + 2(j_3 - j_1) 
    > 0.
\]
But by assumption $x_{1, j_{\sigma(1)}} \dots x_{k, j_{\sigma(k)}}$ has minimum weight, a contradiction.
}

{
\textbf{Case 2.2} Assume $p = 2$. So we have $\sigma(1) = 3, \sigma(2) = 1, \sigma(3) = 4, \dots, \sigma(k-1) = k, \sigma(k) = 2$. Since $j_p < j_1' \le \ell$, we have
\[
    M_\ell(x_{1, j_{\sigma(1)}}  \dots x_{k, j_{\sigma(k)}}) - M_\ell(x_{1, j_1} x_{2, j_2}  \dots x_{k, j_k}) 
    = \sum_{i = 1}^k \left(M_\ell(x_{i, j_{\sigma(i)}}) - M_\ell(x_{i, j_i})\right)
\]
\[
    = \sum_{i = 1}^k (i - 1)(j_i - j_{\sigma(i)})
    = (2 - 1)(j_2 - j_1) + \left(\sum_{i = 3}^{k-1} (i-1)(j_i - j_{i+1})\right) + (k-1)(j_k - j_2)
\]
\[
    = \left(\sum_{i = 4}^{k} j_i\right) +2j_3 - j_1 - (k-2)j_2
    = \left(\sum_{i = 4}^{k} (j_i - j_2)\right) + (j_3 - j_2) + (j_3 - j_1)> 0.
\]
But by assumption $x_{1, j_{\sigma(1)}} \dots x_{k, j_{\sigma(k)}}$ has minimum weight, a contradiction.
}

{
\textbf{Case 2.3} Assume $3 \le p \le k-1$. And so we have $\sigma(1) = 2, \sigma(2) = 1, \sigma(3) = 3, \dots, \sigma(p-1) = p-1, \sigma(p) = p+1, \dots, \sigma(k-1) = k, \sigma(k) = p$. Therefore
\[
    M_\ell(x_{1, j_{\sigma(1)}}  \dots x_{k, j_{\sigma(k)}}) - M_\ell(x_{1, j_2} x_{2, j_1} x_{3, j_3} \dots x_{k, j_k}) 
    = \sum_{i = 3}^k \left(M_\ell(x_{i, j_{\sigma(i)}}) - M_\ell(x_{i, j_i})\right)
\]
\[
    = \sum_{i = p}^k (i - 1)(j_i - j_{\sigma(i)})
    = \left(\sum_{i = p}^{k-1} (i-1)(j_i - j_{i+1})\right) + (k-1)(j_k - j_p)
\]
\[
    = \left(\sum_{i = p+1}^{k} j_i\right) + (p-k)j_p
    = \sum_{i = p+1}^{k} (j_i - j_p) > 0.
\]
But by assumption $x_{1, j_{\sigma(1)}} \dots x_{k, j_{\sigma(k)}}$ has minimum weight, a contradiction.
}

{
\textbf{Case 3.} Assume $j_1', j_2' > \ell$. By induction, we have $\textrm{in}_M(P_{J'}) = x_{1, j_1'} x_{2, j_2'} \dots x_{k-1, j_{k-1}'}$. Since $j_1 \le \ell$ we must have $j_1' = j_2, \dots j_{k-1}' = j_k $ and so $\sigma(1) = 2, \sigma(2) = 3, \dots, \sigma(k-1) = k, \sigma(k) = 1$.
\[
    M_\ell(x_{1, j_{\sigma(1)}}  \dots x_{k, j_{\sigma(k)}}) 
    - M_\ell(x_{1, j_2} x_{2, j_1} x_{3, j_3} \dots x_{k, j_k}) 
    = \left(M_\ell(x_{2, j_{\sigma(2)}}) - M_\ell(x_{2, j_1})\right) + \sum_{i = 3}^k \left(M_\ell(x_{i, j_{\sigma(i)}}) - M_\ell(x_{i, j_i})\right)
\]
\[
    = \left((n - j_3 + \ell + 1) - (\ell - j_1 + 1)\right) + \sum_{i = 3}^k (i - 1)(j_i - j_{\sigma(i)})
    = (j_1 - j_3) + n + \left(\sum_{i = 3}^{k-1} (i-1)(j_i - j_{i+1})\right) + (k-1)(j_k - j_1)
\]
\[
    = \left(\sum_{i = 3}^{k} j_i\right) + j_3 - (k-2)j_1 + n
    = \left(\sum_{i = 3}^{k} (j_i - j_1)\right)  + j_3 + n > 0.
\]
But by assumption $x_{1, j_{\sigma(1)}} \dots x_{k, j_{\sigma(k)}}$ has minimum weight, a contradiction. And so if $j_1' > \ell$ then we must have $j_2' \le \ell$.
}
\end{proof}


\begin{definition}\label{def:matching_field_perm}
Given integers $k$, $n$ and $0\leq\ell\leq n$, $M_\ell$ leads to a permutation for each subset $I = \{i_1, \dots, i_k\} \subset [n]$. More precisely, we think of $I$ as being identified with the Pl\"ucker form $\varphi_n(P_I)$ and
we consider the set to be ordered by $I = \{i_1 < \dots < i_{k}\}$. Since $\ini_{M_\ell}(P_I)$ is a unique term in the corresponding minor of $X=(x_{i,j})$, we have $\ini_{M_\ell}(P_I)=x_{1,i_{\sigma(1)}}\cdots x_{k,i_{\sigma(k)}}$ for some $\sigma \in S_k$, which we call the permutation associated to $M_\ell$. We represent the variable $\ini_{M_\ell}(P_I)$ as a $k \times 1$ tableau where the entry of $(j, 1)$ is $i_{\sigma(j)}$ for each $j \in [k]$. And so we can think of $M_\ell$ as inducing a new ordering on the elements of $I$ which can be read from the tableau.
\end{definition}
\begin{remark}\label{def:block}
{By Propositions~\ref{prop:unique} and \ref{prop:unique_ell} the initial term $\ini_{M_\ell}(P_J)$ is a monomial for each Pl\"ucker form $\varphi_n(P_J)$ where $J = \{j_1 < \dots < j_k \} \subset [n]$. Furthermore, these propositions give a precise description of the initial terms and the induced weight on the Pl\"ucker variable $P_J$ as follows.
\[
{\bf w}_\ell(P_J) = \left\{
\begin{tabular}{@{}ll}
    $0$ & if $k = 1$, \\
    $(n+\ell + 1 - j_2) + \sum_{i = 3}^k (i - 1)(n+1-j_i)$ & if $k \ge 1$ and $|J \cap \{1, \ldots, \ell \}| = 0 $, \\
    $(\ell + 1 - j_1) + \sum_{i = 3}^k (i - 1)(n+1-j_i)$ & if $k \ge 1$ and $|J \cap \{1, \ldots, \ell \}| = 1$, \\
    $(\ell+1 - j_2) + \sum_{i = 3}^k (i - 1)(n+1-j_i)$ & if $|J \cap \{1, \ldots, \ell \}| \ge 2$.
\end{tabular}\right.
\]
}
{Propositions~\ref{prop:unique} and \ref{prop:unique_ell} show that the permutation given by $M_\ell$ and associated to $J$, which defines the matching field, is given by}:
\[
 \BLambda(J)= \left\{
     \begin{array}{@{}ll}
      id  & \text{if $k=1$ or $\lvert J \cap \{1,\ldots,\ell\}\rvert \neq 1$},\\
      (12)  & \text{otherwise}, \\
     \end{array}
   \right.
\]
where $(12)$ is the transposition interchanging $1$ and $2$. The functions $B_{\ell}$ are called $2$-block diagonal matching fields in \cite{KristinFatemeh}. Note that $\ell=0$ or $n$ gives rise to the choice of the diagonal terms in each submatrix as in Example~\ref{ex:diag}. Such matching fields are called {\em diagonal}. See, e.g. \cite[Example 1.3]{sturmfels1993maximal}.  Given a block diagonal matching field $\BLambda$ we define $B_{\ell,1} = \{1, \dots, \ell \}$ and $B_{\ell,2} = \{\ell + 1, \dots, n \}$.
\end{remark}

\begin{example}\label{ex:diag}
{Let $k=3$, $n=5$ and $\ell=0$, so the matching field $B_\ell$ is the diagonal matching field, with $B_{\ell,1}=\emptyset$ and $B_{\ell,2}=\{1,2,3,4,5\}$. We have}
\[
M_0=
\begin{bmatrix}
     0  & 0 & 0  & 0  & 0 \\
     5  & 4 & 3  & 2  & 1 \\
     10 & 8 & 6  & 4  & 2 \\
\end{bmatrix} \textrm{ a weight matrix for }
X =
\begin{bmatrix}
     x_{11}  & x_{12} & x_{13}  & x_{14}  & x_{15} \\
     x_{21}  & x_{22} & x_{23}  & x_{24}  & x_{25} \\
     x_{31}  & x_{32} & x_{33}  & x_{34}  & x_{35} \\
\end{bmatrix}\ .
\]
The corresponding weight vector on $P_{123}, P_{124},\ldots,P_{345}$ is
${\bf w}_0=(10, 8, 6, 7, 5, 4, 7, 5, 4, 4).$
Thus, for each $I=\{i < j < k\} \subseteq [5]$ we have that $\text{in}_{M_0} (P_I) = x_{1i}x_{2j}x_{3k}$. Therefore, the corresponding tableaux for $P_I$ are:
\[\begin{tabular}{|c|}
    \hline 1 \\ \hline 2 \\ \hline 3  \\ \hline
\end{tabular}\ ,\ 
\begin{tabular}{|c|}
    \hline
    1    \\
    \hline
    2  \\
    \hline
    4  \\
    \hline
\end{tabular}
\ ,\ \begin{tabular}{|c|}
    \hline
    1    \\
    \hline
    2  \\
    \hline
    5  \\
    \hline
\end{tabular}
\ ,\ \begin{tabular}{|c|}
    \hline
    1    \\
    \hline
    3  \\
    \hline
    4  \\
    \hline
\end{tabular}
\ ,\ \begin{tabular}{|c|}
    \hline
    1    \\
    \hline
    3  \\
    \hline
    5  \\
    \hline
\end{tabular}
\ ,\ \begin{tabular}{|c|}
    \hline
    1    \\
    \hline
    4  \\
    \hline
    5  \\
    \hline
\end{tabular}
\ ,\ \begin{tabular}{|c|}
    \hline
    2    \\
    \hline
    3  \\
    \hline
    4  \\
    \hline
\end{tabular}
\ ,\ \begin{tabular}{|c|}
    \hline
    2    \\
    \hline
    3  \\
    \hline
    5  \\
    \hline
\end{tabular}
\ ,\ \begin{tabular}{|c|}
    \hline
    2    \\
    \hline
    4  \\
    \hline
    5  \\
    \hline
\end{tabular}
\ ,\ \begin{tabular}{|c|}
    \hline
    3    \\
    \hline
    4  \\
    \hline
    5  \\
    \hline
\end{tabular}\ .
\]
\end{example}
{\noindent Note that each initial term $\text{in}_{M_0}(P_I)$ is the leading diagonal term of the Pl\"ucker form $\varphi_n(P_I)$. Let us consider a block diagonal matching field which is not diagonal.}

\begin{example}\label{ex:ell_3_mf}
{Let $k=3$, $n=5$ and $\ell=3$. Then $B_{\ell,1}=\{1,2,3\}$, $B_{\ell,2}=\{4,5\}$ and
\[
M_2=\begin{bmatrix}
     0  & 0 & 0  & 0  & 0 \\
     3  & 2 & 1  & 5  & 4 \\
     10 & 8 & 6  & 4  & 2 \\
\end{bmatrix}.
\]
Comparing this matrix with $M_0$, the weight matrix for the diagonal case, we see that the only differences are in the second row. The entries of the second row are obtained by permuting the entries in the second row of $M_0$. The corresponding weight vector on the Pl\"ucker variables $P_{123}, P_{124},\ldots,P_{345}$ is
\[{\bf w}_2=(8, 6, 4, 5, 3, 5, 5, 3, 4, 3).\]
For each $I=\{i,j,k\}$ we have the corresponding tableaux for $P_I$ which are
\[
\begin{tabular}{|c|}
    \hline
    1    \\
    \hline
    2  \\
    \hline
    3  \\
    \hline
\end{tabular}\ ,\ \begin{tabular}{|c|}
    \hline
    1    \\
    \hline
    2  \\
    \hline
    4  \\
    \hline
\end{tabular}
\ ,\ \begin{tabular}{|c|}
    \hline
    1    \\
    \hline
    2  \\
    \hline
    5  \\
    \hline
\end{tabular}
\ ,\ \begin{tabular}{|c|}
    \hline
    1    \\
    \hline
    3  \\
    \hline
    4  \\
    \hline
\end{tabular}
\ ,\ \begin{tabular}{|c|}
    \hline
    1    \\
    \hline
    3  \\
    \hline
    5  \\
    \hline
\end{tabular}
\ ,\ \begin{tabular}{|c|}
    \hline
    4    \\
    \hline
    1  \\
    \hline
    5  \\
    \hline
\end{tabular}
\ ,\ \begin{tabular}{|c|}
    \hline
    2    \\
    \hline
    3  \\
    \hline
    4  \\
    \hline
\end{tabular}
\ ,\ \begin{tabular}{|c|}
    \hline
    2    \\
    \hline
    3  \\
    \hline
    5  \\
    \hline
\end{tabular}
\ ,\ \begin{tabular}{|c|}
    \hline
    4    \\
    \hline
    2  \\
    \hline
    5  \\
    \hline
\end{tabular}
\ ,\ \begin{tabular}{|c|}
    \hline
    4    \\
    \hline
    3  \\
    \hline
    5  \\
    \hline
\end{tabular}\ .
\]
}
\end{example}

\section{Degenerations inside Grassmannians}\label{sec:gr}

\subsection{Gr\"obner degenerations induced by matching fields.}\hspace{0.3mm}\label{sec:degen_grassmannian_intro}
{
In this section we introduce the main tool for studying Gr\"obner degenerations. These objects are called \textit{restricted matching field ideals}, which are obtained from the toric initial ideals of the Grassmannian by setting some variables to zero. We begin by studying restricted matching field ideals and later, in Section~\ref{sec:standard_monomial}, we show that these ideals coincide with initial ideals of Richardson varieties inside the Grassmannian.
}

\medskip

\noindent{\bf Notation.} 
{Throughout this section we fix $k,n,\ell$ and the weight vector ${\bf w}_\ell$, see Definition~\ref{def:block}.
We denote the initial ideal of $G_{k,n}$ with respect to $\wb_\ell$ by $G_{k,n,\ell} := {\rm in}_{{\bf w}_\ell}(G_{k,n})$. This is defined as the ideal generated by polynomials $\inwb(f)$ for all polynomials $f \in G_{k,n}$}, where 
\[
\inwb(f)=\sum_{\alpha_j\cdot \wb_\ell=d}{c_{{\bf \alpha}_j}\bf P}^{{\bf \alpha}_j}\quad\text{for}\quad f=\sum_{i=1}^t c_{{\bf \alpha}_i}{\bf P}^{{\bf \alpha}_i}\quad\text{and}\quad d=\min\{\alpha_i\cdot \wb_\ell:\ i=1,\ldots,t\}.
\] 
The ideal $G_{k,n,\ell}$ is introduced in \cite{OllieFatemeh}. In the following theorem we summarize some of its important properties. We refer to Section 4 of \cite{OllieFatemeh} for detailed proofs. 
\begin{theorem}[Theorem 4.1, Theorem 4.3 and Corollary 4.7 in \cite{OllieFatemeh}]\label{thm:JAlebra}
\begin{itemize}
    \item[ ]
    \item[{\rm (i)}] The ideal $G_{k,n,\ell}$ is toric and it is equal to the kernel of the monomial map \begin{eqnarray}\label{eqn:monomialmap}
\phi_{\ell} \colon\  \mathbb{K}[P_I]  \rightarrow \mathbb{K}[x_{ij}]  
\quad\text{with}\quad
 P_{I}   \mapsto {\rm sgn}(B_\ell(I)) \ini_{M_\ell}(P_I).
\end{eqnarray}
Here $M_\ell$ is the matrix in Definition~\ref{example:block} and ${\rm sgn}(-)$ denotes the standard sign function for permutations.
    \item[{\rm (ii)}] The ideal $G_{k,n}$ has a quadratic Gr\"obner basis with respect to $\wb_\ell$. In particular, there exist quadratic polynomials $g_1,\ldots,g_s$ in $G_{k,n}$ such that
    \[
    G_{k,n,\ell}=\langle \init_{\wb_\ell}(g_1), \ldots,\init_{\wb_\ell}(g_s) \rangle,
    \]
    where $s$ is the size of a quadratic minimal generating set of $G_{k,n}$.
\end{itemize}
\end{theorem}
The kernel of the monomial map (\ref{eqn:monomialmap}) is the \textit{matching field ideal} of the block diagonal matching field $B_\ell$ and, in this case, coincides with the initial ideal of the Grassmannian $G_{k,n,\ell}$. We recall that the ideal of the Richardson variety is $(G_{k,n} + \langle P_I : I \in S_{w,k}^v \rangle) \cap \mathbb K[P_I : I \in T_{w,k}^v]$. Replacing $G_{k,n}$ with a matching field ideal leads us to the construction of the restricted matching field ideal. Note that the ideal $G_{k,n,\ell}|_w^v$ is also quadratically generated since it is obtained from $G_{k,n,\ell}$ by setting some variables to zero. For an explicit construction of a quadratic generating set for $G_{k,n,\ell}|_w^v$ see Lemma~\ref{lem:elim_ideal_gen_set}.

\begin{definition}[Restricted matching field ideal]\label{def:S}\label{def:matching_field_ideal_grassmannian}

Given a collection $S$ of $k$-subsets of $[n]$, denote by $T$ the collection of $k$-subsets of $[n]$ that are not in $S$. We define $G_{k,n,\ell}|_{S} = (G_{k,n,\ell} + \langle P_I : I \in S \rangle) \cap \mathbb K[P_I : I \in T]$.

It is useful to think of $G_{k,n,\ell}|_S$ as the ideal obtained from $G_{k,n,\ell}$ by setting the variables $\{P_I : I \in S\}$ to be zero. And so we say that the variable $P_I$ \textit{vanishes} in $G_{k,n,\ell}|_S$ if $I \in S$. Similarly we say that a polynomial or monomial $f \in \mathbb K[P_I : I \subseteq [n], \ |I| = k]$ vanishes in $G_{k,n,\ell}|_S$ if $f \in \langle P_I : I \in S \rangle$. The ideal $G_{k,n,\ell}|_S$ can be computed in $\mathtt{Macaulay2}$~\cite{M2} as an elimination ideal using the command
$\mathtt{eliminate}(G_{k,n,\ell} + \langle P_I : I \in S \rangle, \{P_I : I \in S\})$.

Let $w=(w_1,\ldots,w_k)$ and $v=(v_1,\ldots,v_k)$ be two {Grassmannian permutations} with $v\leq w$.
To simplify our notation through this note, we define the following ideals: 
\begin{equation}
\label{eq:gr}
    G_{k,n,\ell}|_w:=G_{k,n,\ell}|_{S_{w,k}},
\quad G_{k,n,\ell}|^v:=G_{k,n,\ell}|_{S_{k}^v}\quad {\rm and}\quad
G_{k,n,\ell}|_w^v:=G_{k,n,\ell}|_{S_{w,k}^v}.
\end{equation}
\end{definition}
Note that ideals in \eqref{eq:gr}
are all generated in degree two, since they are obtained by setting some variables to zero in the ideal $G_{k,n,\ell}$ which is generated in degree two by Theorem~\ref{thm:JAlebra}. The proof of this claim is given in Lemma~\ref{lem:elim_ideal_gen_set} in which we give an explicit generating set.

\subsection{Richardson varieties.}
Here, we follow our notation in \eqref{eq:gr}.
First we show that the set of permutations leading to zero Richardson varieties is independent of $\ell$. Then we use this to characterize all zero ideals of the form $G_{k,n,\ell}|_w^v$.

Let us begin with an example which illustrates our techniques for manipulating tableaux in the following proofs.
\begin{example}
Following Example~\ref{ex:ell_3_mf}, let $k = 3$, $n = 5$ and $\ell = 3$. Let $v = (1,3,5,2,4)$ and $w = (2,4,5,1,3)$ be permutations and so we have
\[
S_{w,k}^v = \{123, 124, 125, 134, 234, 345\}, \quad
T_{w,k}^v = \{135, 145, 235, 245 \}.
\]
Consider the binomial relation $P_{135}P_{245} - P_{145}P_{235}$ in $G_{k,n,0}$. Since $135, 245, 145, 235 \in T_{w, k}^v$ we have that $P_{135}P_{245} - P_{145}P_{235} \in G_{k,n,0}|_w^v$. In tableau form, the relation is
\[
P_{135}P_{245} - P_{145}P_{235} =
\begin{tabular}{|c|c|}
    \hline
    1 & 2  \\
    \hline
    3 & 4  \\
    \hline
    5 & 5 \\
    \hline
\end{tabular} -
\begin{tabular}{|c|c|}
    \hline
    1 & 2  \\
    \hline
    4 & 3  \\
    \hline
    5 & 5 \\
    \hline
\end{tabular} \ .
\]
However the tableaux above are not matching field tableaux for the matching field $B_3$, moreover the binomial $P_{135}P_{245} - P_{145}P_{235}$ does not lie in the matching field ideal $G_{k,n,\ell}$. This is most easily seen shown by observing that the matching field tableaux representing $P_{135}P_{235}$ and $P_{145}P_{235}$ are not row-wise equal. Explicitly
\[
P_{135}P_{245} = 
\begin{tabular}{|c|c|}
    \hline
    1 & 4  \\
    \hline
    3 & 2  \\
    \hline
    5 & 5 \\
    \hline
\end{tabular} \ , \quad 
P_{145}P_{235} = 
\begin{tabular}{|c|c|}
    \hline
    4 & 2  \\
    \hline
    1 & 3  \\
    \hline
    5 & 5 \\
    \hline
\end{tabular} \ .
\]
However, by permuting the entries in these tableaux we can obtain a binomial inside $G_{k,n,\ell}$ for which at least one of the terms does not vanish. In this case we have $P_{125}P_{345} - P_{135}P_{245} \in G_{k,n,\ell}$ and $P_{135}P_{245}$ does not vanish in $G_{k,n,\ell}|_w^v$. In tableau form we have
\[
P_{125}P_{345} - P_{135}P_{245} = 
\begin{tabular}{|c|c|}
    \hline
    1 & 4  \\
    \hline
    2 & 3  \\
    \hline
    5 & 5 \\
    \hline
\end{tabular} - 
\begin{tabular}{|c|c|}
    \hline
    1 & 4  \\
    \hline
    3 & 2  \\
    \hline
    5 & 5 \\
    \hline
\end{tabular} \ .
\]
So $G_{k,n,\ell}|_w^v$ is non-zero. In this case we have that $P_{125}P_{345}$ vanishes in $G_{k,n,\ell}|_w^v$ and so $G_{k,n,\ell}|_w^v$ contains a monomial while it can be shown that $G_{k,n,0}|_w^v$ does not contain any monomials.
\end{example}

\begin{proposition}
\label{prop:zero_rich_equiv_diag}
The ideal $G_{k,n,\ell}|_w^v$ is zero  if and only if $G_{k,n,0}|_w^v$ is zero.
\end{proposition}

\begin{proof}
We begin by showing that if $G_{k,n,0}|_w^v$ is nonzero then so is $G_{k,n,\ell}|_w^v$. Suppose that $P_I P_J - P_{I'} P_{J'}$ is a binomial in $G_{k,n,0}$
such that $P_I P_J$ does not vanish in $G_{k,n,0}|_w^v$. Now if we have that $P_I P_J - P_{I'} P_{J'}$ is a binomial in $G_{k,n,\ell}$ then we are done. So let us assume that $P_I P_J - P_{I'} P_{J'}$ is not a binomial in $G_{k,n,\ell}$. In particular, if $B_\ell(I) = B_\ell(J)$ then it easily follows that $B_\ell(I) = B_\ell(J) = B_\ell(I') = B_\ell(J')$ and so $P_IP_J - P_{I'}P_{J'} \in G_{k,n,\ell}$. So we must have that $B_\ell(I) \neq B_\ell(J)$.
So without loss of generality let us assume that $B_\ell(I) = id$ and $B_\ell(J) = (12)$. Since $P_I P_J - P_{I'} P_{J'}$ is not a binomial in $G_{k,n,\ell}$, when written in tableau form, $P_I P_J$ and $P_{I'} P_{J'}$ are not row-wise equal. Since $P_I P_J - P_{I'} P_{J'} \in G_{k,n,0}$, all but the first two rows of $P_IP_J$ and $P_{I'}P_{J'}$ are row-wise equal. Therefore the first two rows of $P_IP_J$ and $P_{I'}P_{J'}$ different (even after permutation of the columns):
\[
\begin{tabular}{|c|c|}
    \multicolumn{1}{c}{$I$} & \multicolumn{1}{c}{$J$} \\
    \hline
    $i_1$ & $j_1$  \\
    \hline
    $i_2$ & $j_2$  \\
    \hline
    $X$ & $Y$ \\
    \hline
\end{tabular}
\, - \,
\begin{tabular}{|c|c|}
    \multicolumn{1}{c}{$I'$} & \multicolumn{1}{c}{$J'$} \\
    \hline
    $i_1$ & $j_1$  \\
    \hline
    $j_2$ & $i_2$  \\
    \hline
    $X'$ & $Y'$ \\
    \hline
\end{tabular} \, .
\]
Here $X,Y,X',Y'$ represent the remaining parts of the tableaux. Since the first two rows are different, we deduce that the entries $i_1, i_2, j_1, j_2$ are distinct. There are two cases, either $i_1, i_2 \in B_{\ell,1}$ or $i_1, i_2 \in B_{\ell,2}$.

\textbf{Case 1.} Assume $i_1, i_2 \in B_{\ell,1}$. Let $\alpha < \beta < \gamma$ be the values $i_1, i_2, j_1$ written in ascending order. We define the sets $\tilde I = \{\alpha, \beta \} \cup Y$, $\tilde J = \{\gamma, j_2\} \cup Y$, $\tilde I' = \{\alpha, \gamma \} \cup Y$ and $\tilde J' = \{\beta, j_2 \} \cup Y$. Consider the binomial $P_{\tilde I}P_{\tilde J} - P_{\tilde I'} P_{\tilde J'}$. In tableau form this binomial is
\[
\begin{tabular}{|c|c|}
    \multicolumn{1}{c}{$\tilde I$} & \multicolumn{1}{c}{$\tilde J$} \\
    \hline
    $\alpha$ & $j_2$  \\
    \hline
    $\beta$ & $\gamma$  \\
    \hline
    $Y$ & $Y$ \\
    \hline
\end{tabular}
\, - \,
\begin{tabular}{|c|c|}
    \multicolumn{1}{c}{$\tilde I'$} & \multicolumn{1}{c}{$\tilde J'$} \\
    \hline
    $\alpha$ & $j_2$  \\
    \hline
    $\gamma$ & $\beta$  \\
    \hline
    $Y$ & $Y$ \\
    \hline
\end{tabular} \, .
\]
Note that $\tilde I,\tilde J,\tilde I'$ and $\tilde J'$ are well-defined as $i_1,i_2\not \in Y$. Since the two tableau are row-wise equal we have that
$P_{\tilde I} P_{\tilde J} - P_{\tilde I'} P_{\tilde J'}$ is a binomial in $G_{k,n,\ell}$.
Since $i_1 < i_2$ and $j_1 < i_2$ we have that $i_2 = \gamma$ so $P_{\tilde I'} P_{\tilde J'}$ does not vanish in $G_{k,n,\ell}|_w^v$. So we have shown that  $G_{k,n,\ell}|_w^v$ is nonzero. 

\textbf{Case 2.} Assume $i_1, i_2 \in B_{\ell,2}$. 
{
Let $\alpha < \beta < \gamma$ be the values $i_1, i_2, j_2$ written in ascending order. Let $Z = X$ if $i_2 \ge j_2$, otherwise let $Z = Y$. We define the sets $\tilde I = \{\alpha, \gamma \} \cup Z$, $\tilde J = \{\beta, j_1\} \cup Z$, $\tilde I' = \{\beta, \gamma \} \cup Z$ and $\tilde J' = \{\alpha, j_1\} \cup Z$. The binomial $P_{\tilde I}P_{\tilde J} - P_{\tilde I'} P_{\tilde J'}$ is given by the tableaux
\[
\begin{tabular}{|c|c|}
    \multicolumn{1}{c}{$\tilde I$} & \multicolumn{1}{c}{$\tilde J$} \\
    \hline
    $\alpha$ & $\beta$  \\
    \hline
    $\gamma$ & $j_1$  \\
    \hline
    $Z$ & $Z$ \\
    \hline
\end{tabular}
\, - \,
\begin{tabular}{|c|c|}
    \multicolumn{1}{c}{$\tilde I'$} & \multicolumn{1}{c}{$\tilde J'$} \\
    \hline
    $\beta$ & $\alpha$  \\
    \hline
    $\gamma$ & $j_1$  \\
    \hline
    $Z$ & $Z$ \\
    \hline
\end{tabular} \, .
\]
Since the tableaux are row-wise equal we have $P_{\tilde I} P_{\tilde J} - P_{\tilde I'} P_{\tilde J'}$ is a binomial in $G_{k,n,\ell}$.} 
If $i_1 = \alpha$ then $P_{\tilde I} P_{\tilde J}$ does not vanish in $G_{k,n,\ell}|_w^v$. Otherwise if $i_1 = \beta$ then $P_{\tilde I'} P_{\tilde J'}$ does not vanish in $G_{k,n,\ell}|_w^v$. Note that we cannot have $i_1 = \gamma$ because $i_1 < i_2$. So we have shown that  $G_{k,n,\ell}|_w^v$ is nonzero. 

For the converse we show that if $G_{k,n,\ell}|_w^v$ is nonzero then $G_{k,n,0}|_w^v$ is nonzero. Let $P_I P_J - P_{I'} P_{J'}$ be a binomial in $G_{k,n,\ell}$ such that $P_I P_J$ does not vanish in $G_{k,n,\ell}|_w^v$. If $P_I P_J - P_{I'} P_{J'}$ is a binomial in $G_{k,n,0}$ then we are done. So we assume that $P_I P_J - P_{I'} P_{J'}$ is not a binomial in $G_{k,n,0}$. Therefore we have that $B_\ell(I) \neq B_\ell(J)$ and we may assume without loss of generality that $B_\ell(I) = id$ and $B_\ell(J) = (12)$. We also deduce that, written in tableau form, the first two rows of $P_I P_J$ and $P_{I'} P_{J'}$ are different (even after permutation of the columns)
\[
\begin{tabular}{|c|c|}
    \multicolumn{1}{c}{$I$} & \multicolumn{1}{c}{$J$} \\
    \hline
    $i_1$ & $j_2$  \\
    \hline
    $i_2$ & $j_1$  \\
    \hline
    $X$ & $Y$ \\
    \hline
\end{tabular}
\, - \,
\begin{tabular}{|c|c|}
    \multicolumn{1}{c}{$I'$} & \multicolumn{1}{c}{$J'$} \\
    \hline
    $i_1$ & $j_2$  \\
    \hline
    $j_1$ & $i_2$  \\
    \hline
    $X'$ & $Y'$ \\
    \hline
\end{tabular} \, .
\]
Where $X,Y,X',Y'$ represent the remaining parts of the tableaux. Since the first two rows are different, we deduce that the entries $i_1, i_2, j_1, j_2$ are distinct. There are two cases, either $i_1, i_2 \in B_{\ell,1}$ or $i_1, i_2 \in B_{\ell,2}$.

\textbf{Case 1.} Assume $i_1, i_2 \in B_{\ell,1}$. 
Let $\alpha < \beta < \gamma$ be the values $i_1, i_2, j_1$ written in ascending order. Let $Z = X$ if $i_2 \ge j_1$, otherwise let $Z = Y$. We define the sets $\tilde I = \{\alpha, \gamma \} \cup Z$, $\tilde J = \{\beta, j_2 \} \cup Z$, $\tilde I' = \{\beta, \gamma \} \cup Z$ and $\tilde J' = \{\alpha, j_2\} \cup Z$. The binomial $P_{\tilde I}P_{\tilde J} - P_{\tilde I'} P_{\tilde J'}$ is given by the tableaux
\[
\begin{tabular}{|c|c|}
    \multicolumn{1}{c}{$\tilde I$} & \multicolumn{1}{c}{$\tilde J$} \\
    \hline
    $\alpha$ & $\beta$  \\
    \hline
    $\gamma$ & $j_2$  \\
    \hline
    $Z$ & $Z$ \\
    \hline
\end{tabular}
\, - \,
\begin{tabular}{|c|c|}
    \multicolumn{1}{c}{$\tilde I'$} & \multicolumn{1}{c}{$\tilde J'$} \\
    \hline
    $\beta$ & $\alpha$  \\
    \hline
    $\gamma$ & $j_2$  \\
    \hline
    $Z$ & $Z$ \\
    \hline
\end{tabular} \, .
\]
Since the tableaux are row-wise equal, we have $P_{\tilde I} P_{\tilde J} - P_{\tilde I'} P_{\tilde J'}$ is a binomial in $G_{k,n,0}$.
Since $i_1 < i_2$ and $i_1 < j_1$ we have that $i_1 = \alpha$ so $P_{\tilde I'} P_{\tilde J'}$ does not vanish in $G_{k,n,0}|_w^v$. So we have shown that  $G_{k,n,0}|_w^v$ is nonzero. 

\textbf{Case 2.} Assume $i_1, i_2 \in B_{\ell,2}$. 
Let $\alpha < \beta < \gamma$ be the values $i_1, i_2, j_2$ written in ascending order. Define the sets $\tilde I = \{\alpha, \beta \} \cup X$, $\tilde J = \{j_1, \gamma \} \cup X$, $\tilde I' = \{\alpha, \gamma \} \cup X$ and $\tilde J' = \{j_1, \beta\} \cup X$. The binomial $P_{\tilde I}P_{\tilde J} - P_{\tilde I'} P_{\tilde J'}$ is given by the tableaux
\[
\begin{tabular}{|c|c|}
    \multicolumn{1}{c}{$\tilde I$} & \multicolumn{1}{c}{$\tilde J$} \\
    \hline
    $\alpha$ & $j_1$  \\
    \hline
    $\beta$ & $\gamma$  \\
    \hline
    $X$ & $X$ \\
    \hline
\end{tabular}
\, - \,
\begin{tabular}{|c|c|}
    \multicolumn{1}{c}{$\tilde I'$} & \multicolumn{1}{c}{$\tilde J'$} \\
    \hline
    $\alpha$ & $j_1$  \\
    \hline
    $\gamma$ & $\beta$  \\
    \hline
    $X$ & $X$ \\
    \hline
\end{tabular} \, .
\]
Since the tableaux are row-wise equal we have that $P_{\tilde I}P_{\tilde J} - P_{\tilde I'} P_{\tilde J'}$ is a binomial in $G_{k,n,0}$.
Note that $B_{\ell}(J') = id$ so $j_2 < i_2$. Since $i_1 < i_2$ we have that $i_2 = \gamma$ so $P_{\tilde I} P_{\tilde J}$ does not vanish in $G_{k,n,0}|_w^v$. So we have shown that $G_{k,n,0}|_w^v$ is nonzero. 
\end{proof}

We now prove our main result in this section which give a complete characterization of nonzero monomial-free ideals of Richardson varieties inside Grassmannians.
\begin{theorem}\label{thm:Rich}
\hspace{-3mm}\begin{itemize}
    \item[{\rm (i)}] $G_{k,n,\ell}|_w^v$ is monomial-free
    if and only if $G_{k,n,\ell}|_w$ and $G_{k,n,\ell}|^v$ are both monomial-free.
    \item[{\rm (ii)}] $G_{k,n,0}|_w^v$ is zero if and only if at least one of the following holds:
 \begin{itemize}
  \item There exists some $i$ such that $v_i=w_i$ and $G_{k-1, n,0}|_{w'}^{v'}$ 
  is zero where we define the permutations $v',w'$ as
  $v' = (v_1, \dots, v_{i-1}, v_{i+1}, \dots, v_k)$ and $w' = (w_1, \dots, w_{i-1}, w_{i+1}, \dots, w_k)$.
 \item There exists some $i$ such that $v=(i, i+1, \ldots, i+k-1)$ and $w=(i+1, \ldots, i+k)$.
 \end{itemize}
\end{itemize}

\end{theorem}
\begin{proof}
{\bf (i).} By definition we have 
\begin{align*}
    G_{k,n,\ell}|_w^v &= (G_{k,n,\ell} + \langle P_I : I \in S_{w,k}^v \rangle) \cap \mathbb K[P_I : I \notin S_{w,k}^v] \\
    &= (G_{k,n,\ell} + \langle P_I : I \in S_{w,k} \rangle + \langle P_I : I \in S_{k}^v \rangle) \cap \mathbb K[P_I : I \notin S_{w,k}^v] \\
    &= \left((G_{k,n,\ell} + \langle P_I : I \in S_{w,k} \rangle) \cap \mathbb K[P_I : I \notin S_{w,k}^v]\right) \\
    &\quad + \left((G_{k,n,\ell} + \langle P_I : I \in S_{k}^v \rangle) \cap \mathbb K[P_I : I \notin S_{w,k}^v]\right)\\
    &= G_{k,n,\ell}|_w + G_{k,n,\ell}|^v \subseteq \mathbb K[P_I : I \notin S_{w,k}^v].
\end{align*}
In the above we consider $G_{k,n,\ell}|_w$ and $G_{k,n,\ell}|^v$ as ideals of the ring $\mathbb K[P_I : I \notin S_{w,k}^v]$ by inclusion of their generators.
On the one hand if $G_{k,n,\ell}|^v$ or $G_{k,n,\ell}|_w$ contain a monomial then the same monomial appears in $G_{k,n,\ell}|_w^v$.
On the other hand suppose that $G_{k,n,\ell}|_w^v$ contains a monomial $P_I P_J$ then $v \le I, J \le w$. Also there exists $I', J' $ such that $P_I P_J - P_{I'} P_{J'}$ is a binomial in $G_{k,n,\ell}$ and either $I', J' \not\ge v$ or $I', J' \not\le w$. If $I', J' \not\ge v$ then $P_I P_J$ is a monomial in $G_{k,n,\ell}|^v$. If $I', J' \not\le w$ then $P_I P_J$ is a monomial in $G_{k,n,\ell}|_w$.

\medskip

{\bf (ii).} First suppose that there exists some $i$ such that $v_i = w_i$ and $G_{k-1,n,0}|_{w'}^{v'} = 0$. Suppose by contradiction that $P_I P_J$ does not vanish in $G_{k,n,0}|_w^v$, where $P_I P_J - P_{I'} P_{J'}$ is a binomial in $G_{k,n,0}$. By removing the $i^{\rm th}$ entry of $I, J, I'$ and $J'$ we obtain sets $\tilde{I}, \tilde{J}, \tilde{I'}$ and $\tilde{J'}$. Note that $P_{\tilde{I}}P_{\tilde{J}} - P_{\tilde{I'}} P_{\tilde{J'}}$ is a binomial in $G_{k-1,n,0}$. Also by assumption $P_{\tilde{I}} P_{\tilde{J}}$ does not vanish in $G_{k-1,n,0}|_{w'}^{v'}$ which is a contradiction. So $G_{k,n,0}|_w^v = 0$.

Suppose that,
$v = (v_1, v_1 + 1, \dots, v_1 + k - 1)$ and $w = (v_1 + 1, v_1 + 2, \dots, v_1 + k)$. Suppose by contradiction that there is a binomial $P_I P_J - P_{I'} P_{J'}$ in $G_{k,n,0}$ such that $P_I P_J$ does not vanish in $G_{k,n,0}|_w^v$. Then we have that $i_1, j_1 \in \{v_1, v_1 + 1 \}, \dots, i_k, j_k \in \{v_1 + k - 1, v_1 + k \}$. Since $I \neq J$ we may assume without loss of generality that there exists $s$ such that $i_s < j_s$. Therefore we have $i_s = v_1 + s - 1$ and $j_s = v_1 + s$. This determines the following values of $I$: $i_1 = v_1, i_2 = v_1 + 1, \dots, i_s = v_s + s - 1$ and also determines the following values for $J$: $j_s = v_1 + s, j_{s+1} = v_1 + s + 1, \dots, j_k = v_1 + k$. We now focus on $I'$ and $J'$. Without loss of generality we may assume that $i'_s = i_s = v_s + s - 1$ and $j'_s = j_s = v_1 + s$. Similarly these values determine the same values of $I'$ and $J'$. Therefore $I'$ and $I$ are identical on the first $s + 1$ values and so $J$ and $J'$ are also identical on these values. Also $J$ and $J'$ are identical on the last $n-s$ values because $j'_s = v_1 + s$. Hence $J$ and $J'$ are identical and the binomial $P_I P_J - P_{I'} P_{J'}$ is trivial.

For the converse take a pair of Grasmmannian permutations $v, w$ such that $G_{k,n,0}|_w^v = 0$. First let us assume there exists $t$ such that $v_t = w_t$. We show that $G_{k-1,n,0}|_{w'}^{v'} = 0$ by contradiction. Let $P_I P_J - P_{I'} P_{J'}$ be a binomial in $G_{k-1,n,0}$ such that $P_I P_J$ does not vanish in $G_{k-1,n,0}|_{w'}^{v'}$. We have that $v_{t-1} < v_t < v_{t+1}$ and $w_{t-1} < w_t = v_t < w_{t+1}$. Since $P_I P_J$ does not vanish we have $i_{t-1}, j_{t-1} \le w_{t-1} < w_t = v_t < v_{t+1} < i_{t}, j_{t}$. Therefore we may add $v_t$ to each $I, J, I'$ and $ J'$ to obtain $\tilde{I}, \tilde{J}, \tilde{I'}$ and $\tilde{J'}$. By construction $P_{\tilde{I}} P_{\tilde{J}} - P_{\tilde{I'}} P_{\tilde{J'}}$ a binomial for $G_{k,n,0}|_w^v$ for which $P_{\tilde{I}} P_{\tilde{J}}$ does not vanish, a contradiction.

Now let us assume that for each $t$, we have $v_t < w_t$. Now we will show that $v$ and $w$ are of the desired form by contradiction. That is, we assume that there exists $s$ such that either $ v_{s+1} - v_s  \ge 2$ or $w_{s+1} - w_s \ge 2$. We treat these as two separate cases.

\textbf{Case 1.} Let $v_{s+1} - v_s \ge 2$ for some $s$. Then let
\[
I = \{v_1, \dots, v_k \}\, , \quad 
J = \{v_1 + 1, \dots, v_k + 1 \},
\]
\[
I' = \{v_1, \dots, v_s, v_{s+1} + 1, \dots, v_k + 1 \} \text{ and } 
J' = \{v_1 + 1, \dots, v_s + 1, v_{s+1}, \dots, v_k \}.
\]
Note that $v_s + 1 < v_{s+1}$ because $v_{s+1} - v_s \ge 2$. Since $v_t < w_t$ for each $t$ therefore none of $P_I, P_J, P_{I'}$ and $P_{J'}$ vanish in $G_{k,n,0}|_w^v$. By construction we have $P_I P_J - P_{I'} P_{J'}$ is a binomial in $G_{k,n,0}$ and so $G_{k,n,0}|_w^v$ is nonzero, a contradiction.

\textbf{Case 2.} Let $w_{s+1} - w_s  \ge 2$ for some $s$. Then let
\[
I = \{w_1, \dots, w_k \}\, , \quad 
J = \{w_1 - 1, \dots, w_k - 1 \},
\]
\[
I' = \{w_1, \dots, w_s, w_{s+1} - 1, \dots, w_k - 1 \} \text{ and } 
J' = \{w_1 - 1, \dots, w_s - 1, w_{s+1}, \dots, v_k \}.
\]
Note that $w_s < w_{s+1} - 1$ because $w_{s+1} - w_s \ge 2$. Since $v_t < w_t$ for each $t$ therefore none of $P_I, P_J, P_{I'}$ and $P_{J'}$ vanish in $G_{k,n,0}|_w^v$. By construction we have $P_I P_J - P_{I'} P_{J'}$ is a binomial in $G_{k,n,0}$ and so $G_{k,n,0}|_w^v$ is nonzero, a contradiction.
\end{proof}

\subsection{Schubert and opposite Schubert varieties.}
We first recall the characterization of nonzero monomial-free Schubert ideals from \cite{OllieFatemeh}.
\begin{lemma}[Theorem 5.7 in \cite{OllieFatemeh}] \label{lem:G_knlw_classification}
\hspace{2mm}
\begin{itemize}
    \item[{\rm (i)}] The ideal 
$G_{k,n,\ell}|_w$ contains a monomial
if and only if:
\begin{itemize}
\item $\ell\neq 0$
\item $w_1\in\{2,\ldots,n-k\}\backslash\{\ell\}$, and
\item $\{w_2,\ldots,w_k\}\subseteq \{\ell+1,\ldots,n\}\backslash\{w_1+1\}$
\end{itemize}
\item[{\rm (ii)}] The ideal 
$G_{k,n,\ell}|_w$ is zero if and only if $w$ belongs to the following set:
\[
Z_{k,n} = \{(1,2, \dots, k-1, i) : k \le i \le n \} \cup \{ (1, \dots, \hat{i}, \dots, k, k+1) : 1 \le i \le k-1 \}.
\]
\end{itemize}
\end{lemma}

\begin{remark}
{
As a special example of pairs of permutations which satisfy the conditions of Theorem~\ref{thm:Rich}(ii), we have that if $w \in Z_{k,n}$ or $w_0 v \in Z_{k,n}$ then $G_{k,n,\ell}|_w^v$ is zero.
}\end{remark}

The proofs of many of the statements here
rely on the following key but straightforward observation.
\begin{lemma}[Key Lemma] \label{lem:key_lemma}
Fix $n > k$. Let $I, J \in \binom{[n]}{k}$. Then $I \le J$ if and only if $w_0 I \ge w_0 J$.
\end{lemma}

Before stating our main result, we first characterize which ideals are zero and which ideals contain monomials, which arise from the diagonal case when $\ell=0$. We then extend this result to arbitrary $\ell$ in Theorem~\ref{main:zero Gr}.

\begin{lemma} \label{lem:gr_zero_opposite}
The ideal $G_{k,n,0}|^v$ is zero if and only if $w_0v \in Z_{k,n}$. Moreover, if $\ell=0$ or $\ell>n-k+1$, then the ideal $G_{k,n,\ell}|^v$ is either zero or it does not contain any monomial.
\end{lemma}

\begin{proof}
{
For the first statement, $G_{k,n,0}|^v$ is nonzero if and only if there exists a binomial $P_I P_J - P_{I'} P_{J'}$ in $G_{k,n,0}$ such that $I, J \ge v$. By Lemma~\ref{lem:key_lemma} $I, J \ge v $ if and only if $w_0I, w_0J \le w_0v$. Observe that $P_{w_0I} P_{w_0J} - P_{w_0I'} P_{w_0J'}$ is also a binomial in $G_{k,n,0}$ and all binomials can be written in this form since $w_0^2 = id$. Therefore $G_{k,n,0}|^v$ is non-zero if and only if $G_{k,n,0}|_{w_0v}$ is non-zero. 
}

{
The second statement is a consequence of Lemma~\ref{lem:G_knlw_classification} and the bijection between binomials described above. If $G_{k,n,\ell}|^v$ contains a monomial $P_I P_J$ then there exists a binomial $P_I P_J - P_{I'} P_{J'}$ in $G_{k,n,0}$ such that $I, J \ge v$ and $I', J' < v$. By Lemma~\ref{lem:key_lemma}, $w_0I, w_0J \le w_0v$ and $w_0I', w_0J' > w_0v$. Since $P_{w_0I} P_{w_0J} - P_{w_0I'} P_{w_0J'}$ is also a binomial in $G_{k,n,0}$, therefore $P_{w_0I} P_{w_0J}$ is a monomial in $G_{k,n,\ell}|_{w_0v}$, which contradicts Lemma~\ref{lem:G_knlw_classification}.
}
\end{proof}

In the above proof of Lemma~\ref{lem:gr_zero_opposite}, if $B_\ell$ is the non-diagonal block diagonal matching field and $P_I P_J - P_{I'} P_{J'}$ is a binomial in its corresponding initial ideal $G_{k,n,\ell}$ then it does not follow that $P_{w_0 I} P_{w_0 J} - P_{w_0 I'} P_{w_0 J'}$ is also a binomial. We illustrate this point with the following example.

\begin{example}
Consider $\Gr(3,6)$ with block diagonal matching field $B_1$ where $B_{1,1} = \{1 \}$ and $B_{1,2} = \{2, 3, 4, 5, 6 \}$. We describe the binomial through the following tableaux:
\[
\begin{tabular}{|c|c|}
    \multicolumn{1}{c}{$I$} & \multicolumn{1}{c}{$J$} \\
    \hline
    2 & 4 \\
    \hline
    1 & 5 \\
    \hline
    5 & 6 \\
    \hline
\end{tabular} \, - \,
\begin{tabular}{|c|c|}
    \multicolumn{1}{c}{$I'$} & \multicolumn{1}{c}{$J'$} \\
    \hline
    4 & 2 \\
    \hline
    1 & 5 \\
    \hline
    5 & 6 \\
    \hline
\end{tabular} \, .
\]
We apply $w_0$ to each set above to obtain the sets,
\[
w_0 I = \{2, 5, 6 \} \, , \,
w_0 J = \{1, 2, 3 \} \, , \,
w_0 I' = \{2, 3, 6 \} \, , \,
w_0 J' = \{1, 2, 5 \} \, .
\]
We show that these sets do not give a binomial with respect to the matching field $B_1$. Consider the position of $3$ which appears only in $w_0 J$ and $w_0 I'$. The true order of $w_0 J$ is $(2,1,3)$ and the true order of $w_0 I'$ is $(2,3,6)$. We see that the value $3$ does not lie in the same position, hence the sets do not give rise to a binomial. 

{We note that it is also enough to show that ${\bf w}_{\ell}(w_0I)+{\bf w}_{\ell}(w_0J)\neq {\bf w}_{\ell}(w_0I')+{\bf w}_{\ell}(w_0J')$, however the above method considers only the true order of sets $w_0J$ and $w_0I'$. In general, we will use the true orders of sets and their tableaux in our proofs because there may be many different weight vectors that induce the same matching field.} 
\end{example}

We are now able to prove our main result about opposite Schubert varieties in the Grassmannian which characterizes all permutations $v$ for which $G_{k,n,\ell}|^v$ contains a monomial or is zero.

\begin{theorem}
\label{main:zero Gr}
\hspace{1mm}
\begin{itemize}
    \item[{\rm (i)}] The ideal $G_{k,n,\ell}|^v$ contains a monomial
    if and only if:
\begin{itemize}
    \item $1\leq\ell\leq n-k+1$
    \item $v_1\in \{1, \dots, \ell \}$
    \item $v_2\in \{v_1+2, \ldots, n\}\setminus \{\ell+1\}$ 
\end{itemize}
\item[{\rm (ii)}] The ideal $G_{k,n,\ell}|^v$ is zero if and only if $w_0 v \in Z_{k,n}$.
    
\end{itemize}
\end{theorem}

\begin{proof}
{\bf (i)}
Suppose that $\ell\geq 1$ and $G_{k,n,\ell}|^v$ is nonzero.
First assume that $\ell \le n-k+1$ and take $v = (v_1, \dots, v_k)$ such that $v_1 \in B_{\ell,1}$ and $v_2 \in \{v_1 + 2, \dots, n \} \backslash \{\ell + 1 \}$. We will show that $G_{k,n,\ell}|^v$ contains a monomial by taking cases on $v_2$. Note that $v_2 \neq \ell + 1$.

\textbf{Case 1.} Let $v_2 \le \ell$. Consider the following sets which we write in the true order according to the matching field. Let
\[
I = \{\ell + 1, v_2 - 1, n - k + 3, n-k+4, \dots, n - 1, n\}, \quad
J = \{v_1, v_2, n - k + 3, n-k+4, \dots, n - 1, n\},
\]
\[
I' = \{v_1, v_2 - 1, n - k + 3, n-k+4, \dots, n - 1, n\} \text{ and }
J' = \{\ell+1, v_2, n - k + 3, n-k+4, \dots, n - 1, n\}.
\]
By construction we have that
\[
P_IP_J - P_{I'}P_{J'}
\]
is a binomial in $G_{k,n,\ell}$. We have that $I' \not\ge v$ hence $P_{I'}$ vanishes in $G_{k,n,\ell}|^v$. However $I \ge v$ and $J \ge v$ hence $P_IP_J$ appears as a monomial in $G_{k,n,\ell}|^v$.

\textbf{Case 2.} Let $v_2 \ge \ell + 2$. We now prove that $v_2 + k - 1 \le n$. Suppose by contradiction that $v_2 + k - 1 > n$ then it follows that $v_3 = v_2 + 1$, $v_4 = v_3 + 1$, \dots, $v_k = n$. Now we have that $w_0v = (1,2, \dots, k-1, n - v_1 + 1) \in Z_{k,n}$. By Lemma~\ref{lem:gr_zero_opposite}, $G_{k,n,\ell}|^v$ is zero, a contradiction. Therefore $v_2 + k - 1 \le n$. It follows that there exists $j \in \{2, \dots, k\}$ such that $v_j + 1 \le n$ and $v_j + 1 \not\in\{v_{j+1}, v_{j+2}, \dots, v_k\}$.

Consider the following sets which we write in the true order according to the matching field. Let
\[ 
I = \{v_2, v_1, v_3, \dots, v_k \}, \quad
J = \{\ell + 1, v_2 + 1, v_3 + 1, \dots, v_{j-1} + 1, v_{j} + 1, v_{j+1}, v_{j+2}, \dots, v_k\},
\]
\[ 
I' = \{\ell+1, v_1, v_3, \dots, v_k \} \text{ and }
J' = \{v_2, v_2 + 1, v_3 + 1, \dots, v_{j-1} + 1, v_{j} + 1, v_{j+1}, v_{j+2}, \dots, v_k \}.
\]
By construction we have that
\[
P_IP_J - P_{I'}P_{J'}
\]
is a binomial in $G_{k,n,\ell}$. Since $v_2 \ge \ell + 2$, we have that $I' < v$ hence $P_{I'}$ vanishes in $G_{k,n,\ell}|^v$. However $I \ge v$ and $J \ge v$ hence $P_IP_J$ appears as a monomial in $G_{k,n,\ell}|^v$.

For the converse we assume that $G_{k,n,\ell}|^v$ is contains a monomial. If $\ell > n - k +1$ or $\ell=0$ then by Lemma~\ref{lem:gr_zero_opposite} we have that $G_{k,n,\ell}|^v$ is monomial-free, a contradiction. So we may assume that $\ell \le n-k+1$. Suppose by contradiction that $v_1 \not\in B_{\ell,1}$ then $v_1 \ge \ell+1$. Suppose $P_I P_J$ is a monomial appearing in $G_{k,n,\ell}|^v$. In particular $P_I$ and $P_J$ do not vanish so we have that $I, J \ge v$. We deduce that $I \cap B_{\ell,1} = \emptyset$ and $J \cap B_{\ell,1}= \emptyset$. Suppose that the monomial $P_IP_J$ is obtained from the binomial $P_IP_J - P_{I'}P_{J'}$ in $\inwb(G_{k,n})$. Then we have that $I' \cap B_{\ell,1} = \emptyset$ and $J' \cap B_{\ell,1} = \emptyset$. Therefore the true ordering on all indices $I, I', J, J'$ is the diagonal order. It follows that the same monomial must appear in the ideal $G_{k,n,0}|^v$. However by Lemma~\ref{lem:gr_zero_opposite}, $G_{k,n,0}|^v$ is monomial-free, a contradiction. So we may assume that $v_1 \in B_{\ell,1}$.

It remains to show that if $G_{k,n,\ell}|^v$ contains a monomial then $v_2 \in \{v_1 + 2, \dots, n \} \backslash \{\ell + 1 \}$. By the above argument, we may assume that $1\le \ell \le n - k +1$ and $v_1 \in B_{\ell,1}$.

Assume by contradiction that $v_2 \not \in \{v_1 + 2, \dots, n \} \backslash \{\ell + 1\}$. Then there are two cases, either $v_2 = v_1 + 1$ or $v_2 = \ell + 1$.

\textbf{Case 1.} Let $v_2 = v_1+1$.  Let $P_I P_J$ be a monomial in $G_{k,n,\ell}|^v$ arising from a binomial $P_I P_J - P_{I'} P_{J'}$ in $G_{k,n,\ell}$ and write
\[
I = \{i_1 < \dots < i_k \} \quad \text{and} \quad
J = \{j_1 < \dots < j_k \} \, .
\]
By assumption we have $I, J \ge v$ so in particular $i_1, j_1 \ge v_1$ and $i_2, j_2 \ge v_2$. It is easy to see that $B_{\ell}(I) \neq B_{\ell}(J)$ otherwise it follows that $P_{I'}P_{J'}$ does not vanish in $G_{k,n,\ell}|^v$.
So without loss of generality, assume that $B_{\ell}(I) = id$ and $B_{\ell}(J) = (12)$. So, in tableau form, the binomial $P_I P_J - P_{I'} P_{J'}$ is given by
\[
\begin{tabular}{|c|c|}
    \multicolumn{1}{c}{$I$} & \multicolumn{1}{c}{$J$} \\
    \hline
    $i_1$ & $j_2$  \\
    \hline
    $i_2$ & $j_1$  \\
    \hline
    $\vdots$ & $\vdots$ \\
    \hline
\end{tabular}
\, - \,
\begin{tabular}{|c|c|}
    \multicolumn{1}{c}{$I'$} & \multicolumn{1}{c}{$J'$} \\
    \hline
    $i_1$ & $j_2$  \\
    \hline
    $j_1$ & $i_2$  \\
    \hline
    $\vdots$ & $\vdots$ \\
    \hline
\end{tabular} \, .
\]
Note, we must have the first two rows of these two tableaux are different, otherwise $P_{I'}P_{J'}$ does not vanish in $G_{k,n,\ell}|^v$. 

By assumption we have $i_2, j_2 \ge v_2$ hence $P_{J'}$ does not vanish in $G_{k,n,\ell}|^v$. Hence $P_{I'}$ must vanish. We take cases on $B_{\ell}(I')$.

\textbf{Case 1.a.} Let $B_{\ell}(I') = id$. Then we have $i_1 < j_1$. Since $P_{I'}$ vanishes, we must have $j_1 < v_2 = v_1 + 1$. Therefore $i_1 < v_1$, a contradiction. 

\textbf{Case 1.b.} Let $B_{\ell}(I') = (12)$. Then we have $j_1 < i_1$. 
Since $P_{I'}$ vanishes we must 
have $i_1 < v_2 = v_1 + 1$, and so $j_1 < v_1$ which is a contradiction.

\textbf{Case 2.} Let $v_2 = \ell +1$. Let $P_I P_J$ be a monomial in $G_{k,n,\ell}|^v$ arising from a binomial $P_I P_J - P_{I'} P_{J'}$ in $\inwb(G_{k,n})$ and write
\[
I = \{i_1 < \dots < i_k \} \quad \text{and} \quad
J = \{j_1 < \dots < j_k \} \, .
\]
By assumption we have $I, J \ge v$ so in particular $i_1, j_1 \ge v_1$ and $i_2, j_2 \ge v_2$. It is easy to see that $B_{\ell}(I) \neq B_{\ell}(J)$ otherwise it follows that $P_{I'}P_{J'}$ does not vanish in $G_{k,n,\ell}|^v$. So without loss of generality, assume that $B_{\ell}(I) = id$ and $B_{\ell}(J) = (12)$. So, in tableau form, the binomial $P_I P_J - P_{I'} P_{J'}$ is given by
\[
\begin{tabular}{|c|c|}
    \multicolumn{1}{c}{$I$} & \multicolumn{1}{c}{$J$} \\
    \hline
    $i_1$ & $j_2$  \\
    \hline
    $i_2$ & $j_1$  \\
    \hline
    $\vdots$ & $\vdots$ \\
    \hline
\end{tabular}
\, - \,
\begin{tabular}{|c|c|}
    \multicolumn{1}{c}{$I'$} & \multicolumn{1}{c}{$J'$} \\
    \hline
    $i_1$ & $j_2$  \\
    \hline
    $j_1$ & $i_2$  \\
    \hline
    $\vdots$ & $\vdots$ \\
    \hline
\end{tabular} \, .
\]
Note, we must have the first two rows of these two tableaux are different, otherwise $P_{I'}P_{J'}$ does not vanish in $G_{k,n,\ell}|^v$. By assumption we have $i_2, j_2 \ge v_2 = \ell + 1$ hence $P_{J'}$ does not vanish in $G_{k,n,\ell}|^v$. Hence $P_{I'}$ must vanish. Since $j_1 \in B_{\ell,1}$, we must have $i_1 < v_2$. Since $B_{\ell}(I) = id$ and $i_2 \ge v_2 = \ell + 1 \in B_{\ell,2}$, we must have $i_1 \in B_{\ell,2}$. So $i_1 \ge \ell + 1$, a contradiction.
And so we have shown that $v_2 \in \{v_1 + 2, \dots, n \} \backslash \{\ell + 1 \}$. Therefore $v$ satisfies all desired conditions.


{\bf (ii)}
Since $G_{k,n,\ell}|^v = G_{k,n,\ell}|_{w_0}^v$ the statement follows from Proposition~\ref{prop:zero_rich_equiv_diag} and Lemma~\ref{lem:gr_zero_opposite}.
\end{proof}

\begin{figure}
    \centering
    \includegraphics[scale=0.6]{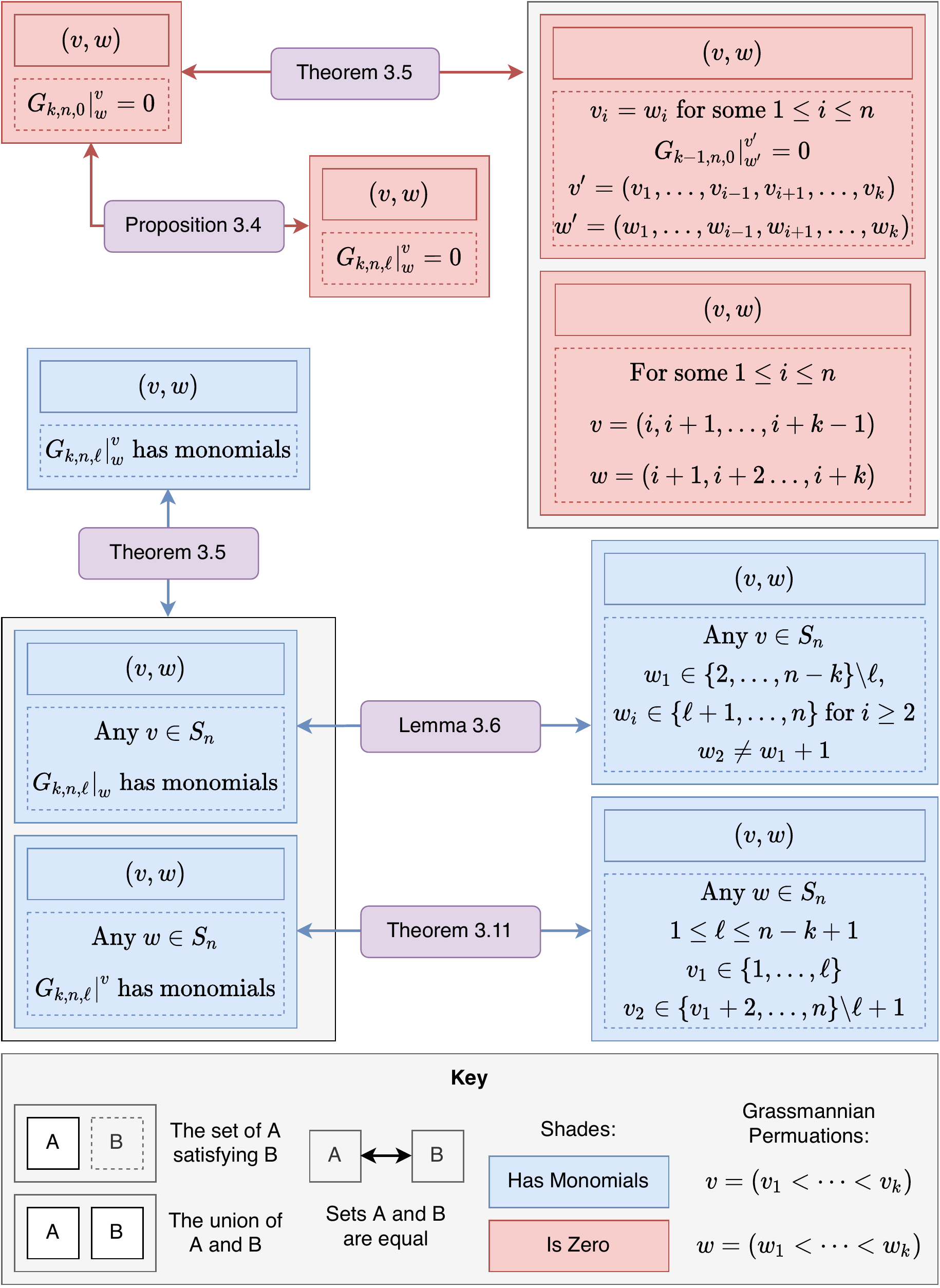}
    \caption{{Depiction of the main results in Section~\ref{sec:gr}. In particular we see the characterization of all pairs of permutations $(v,w)$ for which $G_{k,n,\ell}|_w^v$ is zero, shown in red, and $G_{k,n,\ell}|_w^v$ contains monomials, shown in blue. }}
    \label{figure:Grassmannian_Flowchart}
\end{figure}

\begin{remark}
For general Grassmannians, there are other combinatorial constructions leading to toric degenerations such as
plabic graphs, Newton–Okounkov bodies, 
and cluster algebras \cite{rietsch2017newton,bossinger2020families}. We remark that all degenerations arising in this work can be realized as Gr\"obner degenerations, nevertheless, this is not true in general; See e.g.
\cite{kateri2015family} for an example of a toric degneration which cannot be identified as a Gr\"obner degeneration.
\end{remark}
\begin{remark}
We summarise the results of this section in Figure~\ref{figure:Grassmannian_Flowchart}, which depicts the process by which one obtains pairs of permutations $(v,w)$ whose corresponding ideal $G_{k,n,\ell}|_w^v$ is either zero or contains monomials.
\end{remark}

\newpage\section{Degenerations inside flag varieties}\label{sec:flag}
\subsection{Gr\"obner degenerations.}\label{subsec:flag}\hspace{0.3mm}
We first fix our notation throughout this section. Even though similar results hold for arbitrary $\ell$, we only include results for the diagonal case as the general case requires enormous case studies and various technical arguments.

\smallskip

Fix $n$. We recall the definition of the matrix $M_0$ from Section~\ref{sec:prelim_matching_field} taking $k = n-1$, that is $M_0(i,j)=(i-1)(n-j+1)$ for each $1 \le i \le n-1$ and $1 \le j \le n$. We denote ${\bf w}$ for the \emph{weight vector} induced by $M_0$ on the Pl\"ucker variables and $F_n:={\rm in}_{\bf w}(I_n)$ for its initial ideal. 

\begin{definition}\label{def:matching_field_ideal_flag}
Given a collection $S$ of subsets of $[n]$ we denote $F_n|_{S}$ for the ideal obtained from $F_n$ by setting the variables $\{P_I : I \in S\}$ to be zero.
We say that the variable $P_I$ vanishes in $F_n|_S$ if $I \in S$. 
In particular, to simplify our notation through this note, we define the following ideals for every pair of permutations $v$ and $w$ in $S_n$:
{
\[
F_n|_w:=F_n|_{S_{w}},
\quad F_n|^v:=F_n|_{S^v}\quad {\rm and}\quad
F_n|_w^v:=F_n|_{S_{w}^v}.
\]
}
\end{definition}

\begin{theorem}[Theorem 3.3 and Corollary 4.13 in \cite{OllieFatemeh2}]\label{thm:Pure}    
The ideal $F_n$ is toric and it is equal to the ideal of the diagonal matching field, which is the kernel of the monomial map:
\begin{eqnarray}\label{eqn:monomialmapflags}
\phi_{0} \colon\  \mathbb{K}[P_J]  \rightarrow \mathbb{K}[x_{ij}]  
\quad\text{with}\quad
 P_{J}   \mapsto \ini_{M_0}(P_J).
\end{eqnarray}
Moreover, $F_n$ is generated by quadratic binomials.
\end{theorem}
Similarly to the Grassmannian case, we note that since $F_n$ is quadratically generated, we have that $F_n|_w^v$ is also quadratically generated. In Lemma~\ref{lem:elim_ideal_gen_set}, we give an explicit description of a canonical quadratic generating set.

\begin{definition}\label{def:notation}
We write $Z^{op}_n$ and $T^{op}_n$ for the set of permutations $v$ in $S_n$ for which $F_n|^v$ is zero and monomial-free respectively. We denote by $N^{op}_n$ the collection of permutations $v \in S_n$ for which the ideal $F_n|^v$ contains at least one monomial. Note that such ideals are not toric. Similarly, we denote $Z_n$ and $T_n$ for the set of permutations $w$ in $S_n$ for which the ideal $F_n|_w$ is zero and monomial-free, respectively. For Richardson varieties, we define analogous sets:
\begin{itemize}
    \item $
T^R_{n} = \{(v, w) \in S_n \times S_n : F_n|_w^v \text{ has no monomial} \} \text{ and } Z^R_{n} = \{ (v, w) \in S_n \times S_n : F_n|_w^v \text{ is zero}\}.
$
\end{itemize}

\end{definition}

\subsection{Ascending, descending and compatible permutations.}

\smallskip

In order to classify the permutations which lie in $T^R_n$ and $Z^R_n$ it is important to define the following combinatorial properties and sets of permutations. 

\begin{definition}\label{def:W_v} Let $w = (w_1, \dots, w_n)$ be a permutation in $S_n$ with $w_t=n$.
\begin{itemize}
\item We say $w$ has the \emph{ascending property} if $w_1 < w_2 < \dots < w_t = n$. We denote $S_n^<$ for the set of permutations with ascending property.

\item We say $w$ has the \emph{descending property} if $n = w_t > w_{t+1} > \dots > w_n$. We denote $S_n^>$ for the set of permutations in $S_n$ with descending property.

\item Given $k \le n$, we denote by $(w_1, \dots, w_k)\!\!\uparrow$ the ordered list whose elements are $\{w_1, \dots, w_k \}$ 
taken in increasing order.

\item 
We define $
W_v:=\{vs_{i_1}\cdots s_{i_p}: \{i_1, \dots, i_p\} \subseteq I_v  ~\mbox{and}~ |i_a-i_b|>1~\mbox{for all}~ a, b \in \{i_1,\ldots, i_p\}\}, 
$ where $I_v:=\{i\in [n-1]: v_{i+1}=v_i+1\}$ and $s_i \in S_n$ is the transposition which swaps $i$ and $i+1$.
\end{itemize}
{
Suppose that $w = (w_1, \dots, w_{n+1}) \in S_{n+1}$ is a permutation where $w_t = n+1$ for some $t \in [n+1]$. We denote $\ul{w} = (w_1, \dots, w_{t-1}, w_{t+1}, \dots, w_n) \in S_n$. 
}
\end{definition}

We use the following simple but important observation in the proof of Theorem~\ref{thm:toric_op_schu_diag}. We include it here as it follows easily from the definition and highlights an important relationship between $S^v$ and $S^{\underline v}$ where $v$ is a permutation with the ascending property.

\begin{lemma}\label{lem:asc_prop}
{
Suppose $v \in S_{n+1}$ has the ascending property and let $I \subset [n+1]$ be any subset with $n+1 \in I$. Then $I \ge v$ if and only if $I \backslash \{ n+1 \} \ge \underline{v}$.
}
\end{lemma}

\begin{proof}
{
Since $v$ has the ascending property we write $v_1 < v_2 < \dots < v_r = n+1$ for some $r$. If $|I| \le r$ then we have $i_1 \ge v_1, \dots, i_{|I| - 1} \ge v_{|I| - 1}$. Hence $I \backslash \{ n+1 \} \ge \underline{v}$. If $|I| > r$ then ordering the first $|I|$ elements of $v$ and removing $n+1$ is the same as ordering the first $|I| - 1 $ elements of $\underline{v}$. It follows that $I \backslash \{n+1 \} \ge \underline{v}$. The converse follows from a similar argument.
}
\end{proof}

\begin{definition}\label{def:compatible}
Let $v,w\in S_n$ with $n = v_t = w_{t'}$ and $n-1 = v_s = w_{s'}$. We say that $v, w$ are \emph{compatible} if $v \le w$ and in the case that $t\neq t'$ the following conditions hold:
\[
s'\le t, \quad t'\le s, \quad n=w_{t'}>w_{t'+1}>\cdots>w_t,\quad\text{and}\quad n=v_{t}>v_{t-1}>\cdots>v_{t'}.
\]
Given a subset $Z\subseteq S_n\times S_n$ we define its \emph{extension} as follows:
 \begin{align*}
 \widehat{Z} := & \{(v,w) \in S_{n+1} \times S_{n+1} : (\underline{v}, \underline{w}) \in Z\text{ and } v_t = w_t = n \text{ for some } t\} \\
    \cup & \{(v,w) \in S_{n+1} \times S_{n+1} : (\underline{v}, \underline{w}) \in Z  \text{ and }  v, w  \text{ are compatible}\}.
 \end{align*}\end{definition}

Compatibility of a pair of permutations is important for characterizing $T^R_{n+1} \cup Z^R_{n+1}$ in terms of $T^R_n \cup Z^R_n$. 

\subsection{Richardson varieties.}
\begin{theorem}
\label{thm:toric_fl_rich_diag}
With the notation above, we have:
\begin{itemize}
   \item[{\rm (i)}] $Z_n^R=\{(v,w):\ w\in W_v\}$, 
   \item[{\rm (ii)}] $T^R_{n+1} \cup Z^R_{n+1} = \widehat{T^R_{n}} \cup  \widehat{Z^R_{n}}.$
\end{itemize}
\end{theorem}

\begin{proof}
To prove {\bf(i)} suppose that $w \in W_v$ and assume by contradiction that $F_n|_w^v$ is nonzero. So there exists a binomial $P_I P_J - P_{I'} P_{J'}$ in $F_n$ such that $P_I P_J$ is nonzero in $F_n|_w^v$.

For each $1 \le m \le n$, let us consider for which sets $I$ we have $P_I$ is zero and $|I| = m$. If $m \in I_v$ then $\{v_1, \dots, v_{m - 1} \} = \{w_1, \dots, w_{m - 1} \}$ and $w_{m} \in \{v_{m}, v_{m+1}  \}$. So it follows that if $|I| = m$ and $P_I$ is nonzero then $I$ is either $\{v_1, \dots, v_{m} \}$ or $\{v_1, \dots, v_{m-1}, v_{m+1} \}$. On the other hand if $m \not\in I_v$ then it follows that $\{v_1, \dots, v_m \} = \{w_1, \dots, w_m \}$. So if $P_I$ is nonzero and $|I| = m$ then $I = \{v_1, \dots, v_m \}$.

Suppose $|I| = |J| = m$. By assumption the binomial is not trivial and so $I \neq J$. By the above argument, there are at most two different sets of size $m$ whose corresponding variables do not vanish. So we must have $I = \{v_1, \dots, v_m \}$ and $J = \{w_1, \dots, w_m \}$ and $\{v_1, \dots, v_{m-1} \} = \{w_1, \dots, w_{m-1} \}$. Since $P_I P_J - P_{I'} P_{J'}$ is a binomial, it easily follows that either $I = I'$ and $J = J'$ or $I = J'$ and $J = I'$. In each case we have that $P_I P_J - P_{I'} P_{J'}$ is zero, a contradiction.

If $m = |I| \neq |J|$ then without loss of generality assume that $|I| < |J|$. Suppose that $m \in I_v$ then since $P_I$ is nonzero, we have that $I = \{v_1, \dots, v_m \}$ or $I = \{v_1, \dots, v_{m-1}, v_{m+1} \}$. Since $P_J$ is nonzero and $m+1 \not\in I$, so we have that $J \supseteq  \{v_1, \dots, v_{m+1} \} \supset I$. And so the binomial $P_I P_J - P_{I'} P_{J'}$ is trivial, a contradiction. On the other hand, if $m \not\in I_v$ then by a similar argument we have that $I = \{v_1, \dots, v_m \}$. And so $J \supset \{v_1, \dots, v_{|J| - 1} \} \supseteq \{v_1, \dots, v_m \} = I$. Hence the binomial $P_I P_J - P_{I'} P_{J'}$ is trivial, a contradiction.
 
For the converse, assume that $w\notin W_v$. By induction on $n$, we prove that $F_n|_w^v$ is nonzero. Note that for $n = 2$ the result holds trivially since all ideals $F_2|_w^v$ are zero. For the induction step we assume that $F_{n-1}|_{\underline w}^{\underline v}$ is zero if and only if $\underline w \in W_{\underline v}$. 
Let $w_{t'}=n=v_t$. Note that since $v\leq w$, we have $t'\leq t$. In the next part of the proof we perform the following slight abuse of notation. Given a permutation $\alpha$ and a permutation $\beta \in W_{\alpha}$, we write $I_{\alpha}$ for the subset $\{i_1, \dots, i_r \}$ of $[n-1]$ such that $\alpha = \beta s_{i_1} \dots s_{i_r}$.

{\bf Case 1.} Let $\underline w \in W_{\underline v}$. 

{\textbf{Case 1.1.}} Let $t=t'$. Since $w\notin W_v$ and $\underline w \in W_{\underline v}$ it follows that $t-1 \in I_{\underline v}$ so $v=(v_1,\dots, v_{t-1}, n, v_{t-1}+1, \dots, v_n)$ and $w=(w_1, \ldots, w_{t-2}, v_{t-1}+1, n, v_{t+1}, w_{t+2}\ldots w_n)$. We also have that $\{v_1, \dots, v_{t-2} \} = \{w_1, \dots, w_{t-2} \}$. Now we define
\[
I=\{v_1, \dots v_{t-2}, v_{t+1}, n \}, \quad
J= (v_1, \dots v_{t-2}, v_{t-1}\},\]
\[
I'= \{v_1, \dots v_{t-2}, v_{t-1}, n \} \text{ and }
J'= \{v_1, \dots v_{t-2}, v_{t+1} \} .
\]
Then $P_IP_J-P_{I'}P_{J'}$ is a binomial in $F_n$. Since $v\leq I, J\leq w$, $P_IP_J$ is does not vanish in $F_n|_w^v$ therefore $F_n|_w^v$ is a nonzero ideal.

\smallskip

\textbf{Case 1.2.} Let $t'<t$.

\textbf{Case 1.2.1.} Assume that $t-1 \in I_{\underline v}$ then we have
\[
v=(v_1,\dots, v_{t-1}, n, v_{t-1}+1, \dots, v_n) \textrm{ and } 
w=(w_1, \dots, w_{t-1}, v_{t-1}+1, v_{t-1}, w_{t+2}\dots w_n).
\]
Note that in this case we have $\{v_1, \dots, v_{t-2} \} = \{w_1, \dots, w_{t-1} \}\backslash \{ w_{t'}\}$. So in this case we define
$$
I=\{v_1, \dots, v_{t-2}, v_{t-1}, n \}, \quad
J=\{v_1, \dots, v_{t-2}, v_{t+1})\},
$$ 
$$
I'=\{v_1, \dots, v_{t-2}, v_{t+1}, n \} \text{ and }
J'=\{v_1, \dots, v_{t-2}, v_{t-1}\} .
$$
We have that $P_IP_J-P_{I'}P_{J'}$ is a binomial in $F_n$. As $v\leq I, J\leq w$, $P_IP_J$ does not vanish in $F_n|_w^v$ and so $F_n|_w^v$ is a nonzero ideal.

\smallskip

\textbf{Case 1.2.2.} Assume that $t'-1 \in I_{\underline v}$ then we have
\[
v=(v_1,\dots, v_{t'-1}, v_{t'-1}+1, \dots, v_n) \textrm{ and } 
w=(w_1, \dots, w_{t'-2}, v_{t'-1}+1,n, v_{t'-1}, w_{t'+2}\dots w_n).
\]
Note that we have $\{v_1, \dots, v_{t'-2} \} = \{w_1, \dots, w_{t'-2} \}$. So in this case we define
$$
I= \{v_1, \dots, v_{t'-2}, v_{t'-1}, n \}, \quad
J= \{v_1, \dots, v_{t'-2}, v_{t'-1}+1 \},
$$ 
$$
I'= \{v_1, \dots, v_{t'-2}, v_{t'-1}+1, n \}, \text{ and }
J'= \{v_1, \dots, v_{t'-2}, v_{t'-1}  \} .
$$
We have that $P_IP_J-P_{I'}P_{J'}$ is a binomial in $F_n$. As $v\leq I, J\leq w$, $P_IP_J$ does not vanish in $F_n|_w^v$ and hence $F_n|_w^v$ is a nonzero ideal.

\smallskip

\textbf{Case 1.2.3.} Assume that $t-1 \not\in I_{\underline v}$ and $t'-1 \not\in I_{\underline v}$. It follows that
\[
\{v_1, \dots, v_{t'-1} \} = \{w_1, \dots, w_{t'-1} \}, \quad
\{v_{t'}, \dots, v_{t-1} \} = \{w_{t'+1}, \dots, w_{t} \} 
\]
\[
\textrm{ and } \{v_{t+1}, \dots, v_{n} \} = \{w_{t+1}, \dots, w_{n} \}.
\]
By assumption we have that $w \not\in W_v$ so there exists $i \in \{t', t'+1, \dots, t-1 \}$ such that $v_i + 1 < n-1$. If $v_{t'} = n-1$ then we note that $v_{t'+1} \le n-2$ and we define
\[
I = \{v_1, \dots, v_{t'-1}, n-1, n \}, \quad  
J = \{v_1, \dots, v_{t'-1}, n-2 \},
\]
\[
I' = \{v_1, \dots, v_{t'-1}, n-2, n \}, \textrm{ and } 
J' = \{v_1, \dots, v_{t'-1}, n-1 \}.
\]
Otherwise if $v_{t'} < n-1$ we define
\[
{I = \{v_1, \dots, v_{t'-1}, v_{t'}, n \}, \quad  
J = \{v_1, \dots, v_{t'-1}, n-1 \},
}
\]
\[
I' = \{v_1, \dots, v_{t'-1}, n-1, n \}, \textrm{ and } 
J' = \{v_1, \dots, v_{t'-1}, v_{t'} \}.
\]
In both of the cases above we have that $P_I P_J - P_{I'} P_{J'}$ is a binomial in $F_n$ and also $P_I P_J$ is nonzero in $F_n|_w^v$. So we have shown that $F_n|_w^v$ is a nonzero ideal.

\smallskip

{\bf Case 2.} Assume that $\underline w \notin W_{\underline v}$. Then by induction we have $F_{n-1}|_{\underline w}^{\underline v}$ is nonzero ideal. So, there is a binomial $P_IP_J-P_{I'}P_{J'}$ in $F_{n-1}$  such that $P_IP_J$ does not vanish in $F_n|_{\underline w}^{\underline v}$. Without loss of generality we assume that $|I|=|I'|$ and $|J|=|J'|$. For each $L \in \{I, I', J, J' \}$ we now define $\tilde{L} \in \{ \tilde I, \tilde J, \tilde{I'}, \tilde{J'} \}$:
$$
\tilde L= \begin{cases} L & \text{if} ~ |L|< t',\\
   L\cup \{n\} & \text{if } ~|L|\geq t'.
\end{cases}
$$ 

Then $P_{\tilde{I}}P_{\tilde{J}}-P_{\tilde{I'}}P_{\tilde{J'}}$ is a binomial in $F_n$. Since $P_IP_J$ does not vanish in $F_n|_{\underline w}^{\underline v}$, we have $\underline v\leq I, J \leq \underline w$. Then $v\leq \tilde I, \tilde J\leq w$ and so $P_{\tilde{I}}P_{\tilde{J}}$ does not vanish in $F_n|_w^v$. Hence $F_n|_w^v$ is a nonzero ideal. This completes the proof of (i).

\medskip

{\bf (ii)} The proof of the second statement follows from a series of lemmas. Lemma~\ref{non toric} and Lemma~\ref{comp} imply that ${T_{n+1}^R}\cup Z_{n+1}^R \subseteq \widehat {T_{n}^R}\cup \widehat {Z_{n}^R}$.
 To prove the reverse, let $(v, w)\in \widehat {T_{n}^R}\cup \widehat {Z_{n}^R}.$ If $v_t=w_t=n+1$ for some $t$, by Lemma~\ref{nontoric2}, we see that $(v, w)\in {T_{n+1}^R}\cup Z_{n+1}^R$.
  Therefore assume that there is no $t$ such that $v_t=w_t=n+1$.
  If $(v, w)\not\in T_{n+1}^R \cup Z_{n+1}^R$ then by Lemma~\ref{comp}, the pair $(v, w)$ is not compatible, which is a contradiction. Hence, we conclude that $(v, w)\in {T_{n+1}^R}\cup Z_{n+1}^R$.
\end{proof}

\begin{remark}\hspace{1mm}
Note that Theorem~\ref{thm:toric_fl_rich_diag} provides a complete characterization for the permutations which lead to zero or monomial-free Schubert and opposite Schubert ideals, as we have: 
\begin{itemize}
    \item $Z_n=\{w\in S_n:\ (id,w)\in Z_n^R\}$ and $T_n=\{w\in S_n:\ (id,w)\in T_n^R\}$,
    \item $Z_n^{op}=\{w\in S_n:\ (v,w_0)\in Z_n^R\}$ and $T_n^{op}=\{v\in S_n:\ (v,w_0)\in T_n^R\}$.
\end{itemize}
\end{remark}

\begin{example}[Example of Theorem~\ref{thm:toric_fl_rich_diag}]
Consider the pair $((1,3,2),(2,3,1)) \in T^R_{3}$ which gives rise to a toric ideal. We find all pairs $(v,w) \in T^R_{4}$ such that $\underline{v} = (1,3,2)$ and $\underline{w} = (2,3,1)$. Firstly, if we have $4 = v_t = w_t$ for some $t$ then we have the following pairs:
\begin{itemize}
    \item $((1,3,2,\textbf{4}),(2,3,1,\textbf{4}))$,
    \item $((1,3,\textbf{4},2),(2,3,\textbf{4},1))$,
    \item $((1,\textbf{4},3,2),(2,\textbf{4},3,1))$,
    \item $((\textbf{4},1,3,2),(\textbf{4},2,3,1))$.
\end{itemize}
Secondly, for a compatible pair $(v,w)$ such that $v_t = 4$, $w_{t'} = 4$, assume that $t \neq t'$. Since $v \le w$ therefore $t > t'$. So we get the pair:
\begin{itemize}
    \item $((1,3,\textbf{4},2),(2,\textbf{4},3,1))$.
\end{itemize}

\end{example}
We now proceed to prove the lemmas used in the proof of Theorem~\ref{thm:toric_fl_rich_diag}(ii).

\begin{lemma}\label{lem:fl_toric_1}
Suppose $v \in T^{\text{op}}_{n}$, $w \in T_{n}$ and $v \le w$ then $F_n|_w^v$ does not contain a monomial.
\end{lemma}

\begin{proof} 
Suppose by contradiction that there is a monomial $P_IP_J$ in $F_n|_w^v$. Let $P_IP_J-P_{I'}P_{J'}$ be any binomial in $F_n$ such that $P_{I'}P_{J'}$ vanishes in $F_n|_w^v$. Then either $P_{I'}$ or $P_{J'}$ vanishes in $F_n|_w^v$. Also observe that since $w\in T_{n}$ and $P_IP_J$ does not vanish, we get that $P_{I'}$ and $P_{J'}$ do not vanish in $F_n|_w$. Hence we have $I'\le w$ and $J'\le w$. Similarly, as $v\in T_{n}^{op}$, we can see that $v\le I'$ and $v\le J'$. Thus we have $v\le I'\le w$ and $v\le J'\le w$. Hence $P_{I'}$ and $P_{J'}$ do not vanish in $F_n|_w^v$, which is a contradiction. Therefore, we conclude that $F_n|_w^v$ does not contain a monomial. 
\end{proof}

\begin{remark}
Note that many pairs $(v,w) \in T^R_{n}$ do not arise as in Lemma~\ref{lem:fl_toric_1}. That is, we may have $v \not\in T^{\text{op}}_{n}$ or $w \not\in T_{n}$.  
\end{remark}

\begin{lemma}\label{non toric}
If $(v, w)\in T_{n+1}^R \cup Z_{n+1}^R$ and $\underline v\leq  \underline w$, then $(\underline v, \underline w)\in T_{n}^R \cup Z_{n}^R.$ 
\end{lemma}
\begin{proof} Suppose that $(\underline v, \underline w)  \not\in T_{n}^R \cup Z_{n}^R$. Then there is monomial $P_IP_J$ in $F_n|_{\underline w}^{\underline v}$ which arises from the binomial $P_IP_J-P_{I'}P_{J'}$ in $F_n$. 
Now we construct a monomial in $F_{n+1}|_w^v.$  
Assume that  $|I|=|I'|\geq |J|=|J'|$.
Let $1\leq t'\leq t\leq n+1$ such that $v_t=w_{t'}=n+1$. We take cases on $t$ and $t'$. 
\medskip

{\bf Case 1.} Assume that $t'>|I|$. Then $P_IP_J-P_{I'}P_{J'}$ is a binomial in  $F_{n+1}$. It is clear that $P_{I}P_{J}$ does not vanish in $F_n|_w^v$ and $P_{I'}P_{J'}$ vanish in $F_n|_w^v$. Hence $P_IP_J$ is a monomial in $F_{n+1}|_w^v.$  

\medskip

{\bf Case 2.} If $|J|<t'\leq |I|$. Define 
\[
\tilde I=I\cup \{n+1\}, \ 
\tilde J=J, \ 
\tilde I'=I'\cup \{n+1\} \text{ and }
\tilde J'=J'.
\]
Then  $P_{\tilde I}P_{\tilde J}-P_{\tilde I'}P_{\tilde J'}$ is a binomial in  $F_{n+1}.$ Since $\underline v \leq I, J\leq \underline w$, then $v\leq \tilde I, \tilde J\leq w$. Since $P_{I'}P_{J'}$ vanishes in $F_n|_w^v$, then one of the following holds:  \begin{itemize}
    \item $\underline v\nleq I'$ or
$\underline v\nleq J'$ 
\item  $\underline w \ngeq I'$ or
$\underline w\ngeq J'$ 
\end{itemize}
Then by definition  $\tilde I'$ and $\tilde J'$ we can easily see one of the following holds: \begin{itemize}
    \item $v\nleq \tilde I'$ or
$ v\nleq \tilde J'$ 
\item  $ w \ngeq \tilde I'$ or
$w\ngeq \tilde  J'$ 
\end{itemize}
Hence $P_{\tilde I'}P_{\tilde J'}$ vanish in $F_n|_w^v$ and so $P_{\tilde I}P_{\tilde J}$ is a monomial in $F_{n+1}|_w^v.$  

\medskip

{\bf Case 3.} If $t'\leq |J|$. Define
\[
\tilde I=I\cup \{n+1\}, \  
\tilde J=J\cup \{n+1\}\, \ 
\tilde I'=I'\cup \{n+1\} \text{ and }
\tilde J'=J'\cup \{n+1\}.
\]
Then $P_{\tilde I}P_{\tilde J}-P_{\tilde I'}P_{\tilde J'}$ is a binomial in  $F_{n+1}.$ Since $\underline v \leq I, J\leq \underline w$, then $v\leq \tilde I, \tilde J\leq w$. By similar arguments as in Case 2, we can see that $P_{\tilde I}P_{\tilde J}$ is a monomial in $F_{n+1}|_w^v.$  
\end{proof}

\begin{lemma}\label{nontoric2}
Let $(v, w) \not\in  T^R_{n+1} \cup Z^R_{n+1}$ such that  $v_t=n+1=w_t$ for some $t$ and $\underline v\leq \underline w$,   then $(\underline v, \underline w) \not\in T^R_{n} \cup Z^R_n$.
\end{lemma}
\begin{proof}
Since $(v, w)\not\in  T^R_{n+1} \cup Z^R_{n+1}$, then there is a binomial $P_IP_J-P_{I'}P_{J'}$ in $F_{n+1}$ such that $P_IP_J$ does not vanish in $F_n|_w^v$ and $P_{I'}P_{J'}$ vanishes in $F_n|_w^v$. Assume that $|I|=|I'|\geq |J|=|J'|$.

\medskip

{\bf Case 1.} If $n+1\not\in I \cup J$, then $P_IP_J-P_{I'}P_{J'}$ is a binomial in $F_n$.  
We prove that $P_IP_J$ does not vanish in $F_n|_{\underline{w}}^{\underline{v}}$. Since $v\leq I$, we note that $n+1\notin \{v_1, \ldots, v_k\}$.
As $v_t=w_t$ for some $t$, $n+1\notin \{w_1, \ldots, w_k\}$. Then $\underline v\leq I$ and $I\leq \underline w$ and so $P_I$ does not vanish in $F_n|_{\ul w}^{\ul v}$. Similarly we can see that $P_J$ does not vanish in $F_n|_{\ul w}^{\ul v}$. Hence $P_IP_J$ does not vanish in $F_n|_{\underline{w}}^{\underline{v}}$. 

Now we prove that $P_{I'}P_{J'}$  vanishes in $F_n|_{\ul w}^{\ul v}$. Without loss of generality assume that $P_{I'}$ vanishes in $F_n|_w^v$. Then we have either $v\nleq I'$ or $I'\nleq w$. By above arguments, we have $n+1\notin \{v_1, \ldots, v_k, w_1, \ldots, w_k\}$. Then we can see that $\underline v\nleq I'$ or $I'\nleq \underline w$.

\medskip

{\bf Case 2.}
Assume that $n+1\in I \cup J$. Since $P_IP_J-P_{I'}P_{J'}$ is a binomial in $F_{n+1}$, if $|J|=1$, then $J\neq \{n+1\}$. For $L\in\{I, J, I', J'\}$, we define 
\[
\tilde L=
\begin{cases}
L\setminus \{n+1\} & \text{if} \  n+1 \in L, \\
L & \text{otherwise}.
\end{cases}
\]
Then it is easy to see that $P_{\tilde I}P_{\tilde J}-P_{\tilde I'}P_{\tilde J'}$ is a binomial in $F_n$. 
First note that for any $u\in S_{n+1}$, we have that $\underline u\leq \tilde L$ if and only if $u\leq L$. Hence if $n+1\in I\cap J$, then $P_{\tilde I}P_{\tilde J}$ is a monomial in $F_n|_{\underline w}^{\underline v}$. Therefore we assume that $n+1\in I, I'$ and $n+1\notin J, J'$.

Since $P_IP_J$ does not vanish in $F_n|_w^v$ and $w_t=n+1=v_t$ for some $t$, by similar arguments as above we can see that $P_{\tilde I}P_{\tilde J}$ does not vanish in $F_n|_{\underline w}^{\underline v}$. Now we prove that $P_{\tilde I'}P_{\tilde J'}$ vanishes in $F_n|_{\underline w}^{\underline v}$.
If $P_{I'}$ vanishes in $F_n|_w^v$, since $n+1\in I'$ then $P_{\tilde I'}$ vanishes in $F_n|_{\underline w}^{\underline v}$. Therefore assume that  $P_{I'}$ does not vanish in $F_n|_w^v$ and $P_{J'}$ vanishes in $F_n|_w^v$. Then we have either $v\nleq J'$ or $J'\nleq w$. 

If $w_t=n+1=v_t$ with $t>|J'|$, we are done. 
Assume that $t\leq |J'|$. If $J'\nleq w$, then $(j'_1<\cdots <j'_{|J'|})\nleq (w_1, \ldots, w_t=n+1, \ldots, w_{|J'|})\!\!\uparrow$. Since $n+1\notin J'$, we have $$(j'_1<\cdots <j'_{|J'|-1})\nleq (w_1, \ldots, w_{t-1}, w_{t+1}, \ldots, w_{|J'|})\!\!\uparrow .$$ Hence we have $\tilde J'=J' \nleq \underline w$. 

Finally assume that $v\nleq J'$. Since $t\leq |J|=|J'|$ and $n+1\notin J$, we observe that $v\nleq J$. Hence $P_J$ vanishes in $F_n|_w^v$, which is a contradiction. Hence if $t\leq |J|=|J'|$, then $v\leq J'$.  Therefore $P_{\tilde I'}P_{\tilde J'}$ is a monomial in $F_n|_{\underline w}^{\underline v}$.
\end{proof}

For the following lemma, recall that for any pair of permutations $(v,w) \in S_{n+1} \times S_{n+1}$, we denote $v_t = w_{t'} = n+1$ and $v_s = w_{s'} = n$.

\begin{lemma}\label{comp}
Let $(v, w)\in T_{n+1}^R\cup Z_{n+1}^R$ such that $t'<t$ and $\underline v\leq  \underline w$. Then  $v, w$ are compatible.
\end{lemma}

\begin{proof}
Since $(v, w)\in T_{n+1}^R \cup Z_{n+1}^R$, by Lemma \ref{non toric}, then $(\underline v, \underline w)\in T_{n}^R\cup Z_{n}^R$. We prove that if $(v, w) \in S_{n+1} \times S_{n+1}$ is not compatible and $v < w$ then $(v, w)\not\in T_{n+1}^R \cup Z^R_{n+1}$. 
Let $v_\ell=n-1=w_{\ell'}$.

\medskip

{\textbf{Case 1.} Assume that $t<s'$. Then we have $t'\leq t<s'\leq s$. Suppose that $s' = s$ and $\ell' > s$. Let
\[
I = \{w_1, \dots, w_{s'} \} = \{i_1 < i_2 < \dots < i_{s'-2} < n < n+1\}, \quad
J = \{i_1 < \dots < i_{s'-3} < n-1 \}.
\]
Note that $\ul v \le \ul w$ and $s = s' < \ell'$ therefore $i_{s'-2} < n-1$. And so we have $v \le I, J \le w$. Let
\[
I' = \{i_1 < \dots <i_{s'-3} < n-1 < n < n+1 \}, \quad
J' = \{i_1 < \dots < i_{s'-2}\}.
\]
By construction it is clear that $P_IP_J - P_{I'}P_{J'}$ is a binomial in $F_{n+1}$. Since $n-1 \notin \{w_1, \dots, w_{s'} \}$, it follows that $I' \not\le w$. And so $P_IP_J$ is a monomial in $F_{n+1}|_w^v$. So we may now assume that either $\ell'\leq s$ or $s' < s$. However if $s' < s$ then, by induction, we have $\ell'\leq s$. So we assume $\ell' \leq s$.
}

\medskip

{\bf Case 1.1.} Let $\ell'>s'$. Since $t'\leq t<s'$ and $t'\neq t$, we have $t'<s'-1$. Take 
\[
I=(w_1, \ldots, w_{s'})\!\!\uparrow, \quad 
J=\begin{cases} (w_1, \ldots, w_{t'-1}, w_{\ell'})\!\!\uparrow & \text {if}~ t'>1, \\ w_{\ell'} & \text {if}~ t'=1, \end{cases}
\]
\[
I'=\begin{cases}(w_1, \ldots, w_{t'-1}, w_{\ell'}, w_{t'}, w_{t'+2}, \ldots, w_{s'})\!\!\uparrow & \text {if}~ t'>1, \\  (w_{\ell'}, w_{t'}, w_{t'+2}, \ldots, w_{s'})\!\!\uparrow & \text {if}~ t'=1 \end{cases} \text{ and }
\]
\[
J'=\begin{cases}(w_1, \ldots, w_{t'-1}, w_{t'+1})\!\!\uparrow & \text {if}~ t'>1, \\ w_{t'+1} & \text {if}~ t'=1.  \end{cases}
\]
It is easy to see that $P_IP_J-P_{I'}P_{J'}$ is a binomial in $F_{n+1}$. It is clear from the construction, we have $v\leq I, J\leq w$. Then  $P_IP_J$ does not vanish in $F_{n+1}|_w^v$. Since $I'\nleq w$, then $P_{I'}$ vanishes in $F_{n+1}|_w^v$. Then $P_IP_J$ is a monomial in $F_{n+1}|_w^v$.
 
\medskip

\textbf{Case 1.2.}
Let $\ell'<s'$. Let $k=\max\{\ell', t\}$ and $v_{r}=max\{v_i: 1\leq i\leq k, i\neq t\}$. Now define 
\[
I=(v_1, \ldots, v_{k})\!\!\uparrow, \quad 
J=(v_1, \dots, v_{r-1}, v_{r+1}, \dots, v_{k-1}, n-1)\!\!\uparrow,
\]
\[
I'=(v_1, \ldots, v_{r-1}, v_{r+1}, \ldots, v_{k},  n-1,)\!\!\uparrow
\text{ and } 
J'= (v_1, \dots, v_{k-1})\!\!\uparrow
\]
Consider the tableau for $P_IP_J$ and $P_{I'}P_{J'}$. Note that all rows are the same except for the $(k-1)^{th}$ row, and in this row we interchange $n-1$ and $v_r$. Since $k<s'\leq s$, $n-1\notin I$ and $P_IP_J-P_{I'}P_{J'}$ is a binomial in $F_{n+1}$,
Clearly, $v\leq I, J\leq w$ and so $P_IP_J$ does not vanish in $F_{n+1}|_w^v$.
 Since $k<s'$ and $n-1\in I'$, we can see that $I'\nleq w$. So $P_{I'}$ vanishes in $F_{n+1}|_w^v$.  

\medskip

{\bf Case 2.} Assume that $t'>s$. Then we have $t\geq t'>s\geq s'$. Let $v_r=max\{v_i: 1\leq i\leq t', i\neq s\}$. We define 
\[
I=\begin{cases}(v_1, \ldots, v_{s-1}, n, v_{s+1}, \ldots, v_{t'})\!\!\uparrow & \text {if}~ s>1, \\    (n, v_{s+1}, \ldots, v_{t'})\!\!\uparrow & \text {if}~ s=1,  \end{cases} \quad J=(v_1, \ldots, v_{t'-1})\!\!\uparrow,
\]
\[
I'=(v_1, \ldots,v_{r-1}, v_{r+1}, \ldots, v_{t'}, n)\!\!\uparrow \text{ and }
J'=\begin{cases}(v_1, \ldots, v_{s-1}, v_{s+1}, \ldots, v_{t'-1}, v_r) \!\!\uparrow  & \text {if}~ s>1, \\ (v_{s+1}, \ldots, v_{t'-1}, v_r) \!\!\uparrow & \text {if}~ s=1. \end{cases}
\]
 Then $P_IP_J-P_{I'}P_{J'}$ is a binomial in $F_{n+1}$. Since $n-1, n \notin J'$ and $t'>s$, we have $v\nleq J'$ and $P_IP_J$ is a monomial in $F_{n+1}|_w^v$.

\medskip

{\bf Case 3.} Assume that there exists $t'\leq k < t$ such that $v_k>v_{k+1}$.
We take
\[
I=\begin{cases}(v_1, \ldots, v_{k-1}, v_{k+1}, n)\!\!\uparrow & \text {if}~ k>1, \\   (v_{k+1}, n)\!\!\uparrow & \text {if}~ k=1,  \end{cases} \quad J=(v_1, \ldots, v_{k})\!\!\uparrow,
\]
\[
I'=\begin{cases}(v_1, \ldots,v_{k-1}, v_{k}, n)\!\!\uparrow & \text {if}~ k>1, \\ (v_{k}, n)\!\!\uparrow & \text {if}~ k=1, \end{cases} \text{ and } J'=(v_1, \ldots, v_{k-1}, v_{k+1}) \!\!\uparrow.
\]
Then $P_IP_J-P_{I'}P_{J'}$ is a binomial in $F_{n+1}$. Since $v_{k+1}>v_k$, we have $v\leq I, J$ and $v\nleq J'$. Now we show that $I\leq w$. Since $\underline v \leq \underline w$ and $v_{k+1}>v_k$, we have 
$$(v_1, v_2, \ldots, v_{k-1}, v_{k+1})\!\!\uparrow \leq (w_1, \ldots, w_{t'-1}, w_{t'+1}, \ldots, w_k, w_{k+1})\!\!\uparrow .$$
Then $(v_1, v_2, \ldots, v_{k-1}, v_{k+1}, n)\!\!\uparrow \leq (w_1, \ldots, w_k, w_{k+1})\!\!\uparrow$.   Then we see that $P_IP_J$ is nonzero and $P_{J'}$ vanishes in $F_n|_w^v$. Hence $P_IP_J$ is a monomial in $F_{n+1}|_w^v$. 
 
\medskip

{\bf Case 4.} Assume that there exists $t'< k\leq t$ such that $w_k<w_{k+1}$.  In this case, we choose:
\[
I=(w_1, \ldots, w_{k})\!\!\uparrow, \quad
J=\begin{cases}(w_1, \ldots, w_{t'-1}, w_{t'+1}, \ldots, w_{k-1}, w_{k+1})\!\!\uparrow & \text{if}~ t'>1, \\
(w_2, \ldots, w_{k-1}, w_{k+1})\!\!\uparrow & \text{if}~ t'=1,
\end{cases} 
\]
\[
I'=(w_1, \ldots,w_{k-1}, w_{k+1})\!\!\uparrow \text{ and } J'=\begin{cases}
(w_1, \!\!\ldots, w_{t'-1}, w_{t'+1}, \ldots, w_{k-1}, w_{k})\!\!\uparrow & \text{if}~ t'> 1, \\
(w_2, \ldots, w_{k-1}, w_{k}) & \text{if}~ t'=1. \end{cases}
\]
 
It is easy to see that $P_IP_J-P_{I'}P_{J'}$ is a binomial in $F_{n+1}$.
Since $w_{k}<w_{k+1}$, we have $v\leq I\leq w$ and $J\leq w$. Now we show that $v\leq J$. Since $\underline v \leq \underline w$ and $w_{k+1}>w_k$, we have $(v_1, v_2, \ldots, v_{k-1})\!\!\uparrow \leq (w_1, \ldots, w_{t'-1}, w_{t'+1}, \ldots, w_{k-1}, w_{k+1})\!\!\uparrow$. Note that $P_IP_J$ is nonzero and $P_{I'}$ vanishes in $F_n|_w^v$. Then $P_IP_J$ is a monomial in $F_{n+1}|_w^v$. 
 
Hence we conclude that if $(v, w)$ is not compatible, then $(v, w) \not\in T_{n+1}^R \cup Z_{n+1}^R$. This completes the proof of the lemma. 
\end{proof}

\begin{remark}
For each $3 \le n \le 6$ and $0 \le \ell \le n-1$ we calculate the number of pairs of permutations $(v,w)$ such that $v < w$ and $F_{n,\ell}|_w^v$ is monomial-free, zero or contains a monomial. The results of these calculations are collected in Table~\ref{table:flag_calculation}. Here, $F_{n,\ell}|_w^v:={\rm in}_{{\bf w}_\ell}(I_{n})|_{P_I=0\atop  I\in S_{w}^v}$ is the ideal of the matching field $B_{\ell}$ whose corresponding weight vector is ${\bf w}_\ell$ and $I_n$ is the Pl\"ucker ideal of the flag variety. Also for the Grassmannian case we calculate, for a given pair $(n,k)$, the number of triples $(\ell, v, w)$ for which $G_{k,n,\ell}|_w^v$ is a toric ideal. These calculations are collected in Table~\ref{table:gr_calculations}.
\begin{table}[]
    \centering
    \resizebox{0.33\textwidth}{!}{
    \begin{tabular}{|c|c|c|c|c|}
        \hline
        $n$     & $\ell$    & Toric     & Zero      & Non-toric \\
        \hline
        3       & 0         & 4         & 4         & 5         \\
                & 1         & 4         & 4         & 5         \\
                & 2         & 1         & 6         & 6         \\
        \hline
        4       & 0         & 39        & 20        & 130       \\
                & 1         & 38        & 20        & 131       \\
                & 2         & 28        & 23        & 138       \\
                & 3         & 22        & 24        & 143       \\
        \hline
        5       & 0         & 343       & 114       & 3204      \\
                & 1         & 329       & 114       & 3218      \\
                & 2         & 269       & 125       & 3267      \\
                & 3         & 228       & 125       & 3308      \\
                & 4         & 255       & 133       & 3274      \\
        \hline
        6       & 0         & 3066      & 750       & 93871     \\
                & 1         & 2907      & 750       & 94030     \\
                & 2         & 2490      & 796       & 94401     \\
                & 3         & 2180      & 803       & 94704     \\
                & 4         & 2318      & 818       & 94551     \\
                & 5         & 2598      & 851       & 94238     \\
        \hline
    \end{tabular}
    }
    \caption{For each $3 \le n \le 6$ and $0 \le \ell \le n-1$, we give the number of pairs of permutations $(v,w)$ for which $F_{n,\ell}|_w^v$ is either a toric ideal, a zero ideal or non-toric ideal.}
    \label{table:flag_calculation}
\end{table}

\begin{table}
    \centering
    \resizebox{0.45\textwidth}{!}{
    \begin{tabular}{ccc}
      \begin{tabular}{|c|c|c|}
        \hline
        $G_{k,n,\ell}|_w^v$ & \multicolumn{2}{c|}{$k$}\\
        \hline 
        $n$ & 2 & 3 \\
        \hline
        4 & 10 & \\
        5 & 49 & 71 \\
        6 & 151 & 902 \\
        \hline
    \end{tabular}
    &&
    \begin{tabular}{|c|c|c|}
        \hline
        $G_{k,n,\ell}|^v$ & \multicolumn{2}{c|}{$k$}\\
        \hline 
        $n$ & 2 & 3 \\
        \hline
        4 & 6  &        \\
        5 & 17 & 23     \\
        6 & 34 & 74     \\
        \hline
    \end{tabular}
    \end{tabular}
    }
    \caption{For each $(k,n) \in \{(2,4), (2,5), (2,6), (3,5), (3,6)\}$ we calculate the number of triples $(\ell, v, w)$ such that $G_{k,n,\ell}|_w^v$ is toric. Similarly for the opposite Schubert varieties we calculate the number of pairs $(v,\ell)$ for which $G_{k,n,\ell}|^v$ is toric.
    }
    \label{table:gr_calculations}
\end{table}
\end{remark}


\subsection{Schubert and opposite Schubert varieties.}
We now focus on the Schubert and opposite Schubert varieties inside the flag variety. In each case, our goal is to give an explicit description of the permutations $v$ and $w$ for which $F_n|^v$ and $F_n|_w$ are monomial-free or zero. And for particular cases we note which of these ideals are principal.
{
First we recall that $F_n$ is a toric ideal by \cite[Corollary~4.13]{OllieFatemeh2}, in particular it is monomial free. We also recall following results from \cite{OllieFatemeh3} characterizing the permutations for which $F_n|_w$ is zero and contains no monomial respectively.}
\begin{theorem}[Schubert varieties, Theorem~4.1 and Theorem~5.5 in \cite{OllieFatemeh3}]
The ideal $F_n$ is monomial-free. Moreover, 
\begin{itemize}
    \item[{\rm (i)}] $F_n|_w=0$ if and only if $w\in Z_n$, where
\begin{itemize}
\item $Z_{n} = \{ s_{i_1} \dots s_{i_p} \in S_n : \lvert i_k - i_{\ell} \rvert \ge 2, \text{ for all } k, \ell\}.$
\end{itemize}
\item[{\rm (ii)}] 
$T_n=A_1 \cup A_2$, where 
\begin{itemize}
    \item $A_1 =\{w\in S_n:\ F_{n-1}|_{\ul w}=0,\  w_n=n-2\ \text{and}\ \{w_{n-2},w_{n-1}\}=\{n-1, n\}\}$,
    \item $A_2 = \{w \in S_n:\ F_{n-1}|_{\ul w}\ {\rm is\ monomial-free},\ \ul{w} \in S_{n-1}^> \text{ and if } w_s = n-1, w_t = n \text{ then } t \ge s-1\}$.
\end{itemize}
\end{itemize}
\end{theorem}

We now state and prove the analogous result for opposite Schubert varieties.

\begin{theorem}[Opposite Schubert varieties]\label{thm:toric_op_schu_diag}
Following our notation in Definition~\ref{def:notation}: 
\begin{itemize}\item[{\rm (i)}]  $v\in Z_n^{\text{op}}$ 
if and only if $w_0v\in Z_{n}$.
\item[{\rm (ii)}]
$T^{\text{op}}_{n}= A_1^{\text{op}} \cup A_2^{\text{op}}$, where
\begin{itemize}
    \item $A_1^{\text{op}} = \{v \in S_n:\ F_n|^{v}\neq 0,\ F_{n-1}|^{\underline{v}}=0,\ {\rm and}\ v \text{ has the ascending property}\}$
    \item $A_2^{\text{op}} = \{v \in S_n:\  F_{n-1}|^{\underline{v}}\ \text{ has no monomial and}\  v  \text{ has the ascending property}\}.$
\end{itemize}
\end{itemize}
\end{theorem}

\begin{proof}
To prove {\bf (i)} first note that as the initial term of each Pl\"ucker variable is the diagonal term in its corresponding submatrix, every binomial $P_IP_J-P_{I'}P_{J'}$ in $F_n$ leads to the binomial $P_{w_0I}P_{w_0J}-P_{w_0I'}P_{w_0J'}$ in $F_n$.
Now, note that if $F_n|^v$ is nonzero, then there is a binomial $P_IP_J-P_{I'}P_{J'}$ in $F_n$ such that $P_IP_J$ does not vanish in $F_n|^v$. 
Since $v \le I, J$, we have that $w_0I, w_0J \le w_0v$. Hence $P_{w_0I}P_{w_0J}$ does not vanish in $F_n|_{w_0v}$, and so $F_n|_{w_0v}$ is nonzero.

Conversely, if $w_0 v \not \in Z_n$ then $F_n|_{w_0v}$ is nonzero. So there is a binomial $P_I P_J - P_{I'} P_{J'}$ in $F_n$ such that $P_I P_J $ does not vanish in $F_n|_{w_0v}$.  
Since $I, J \le w_0 v$, we have that $w_0 (w_0 v) = v \le w_0I, w_0J$. Hence $P_{w_0I}P_{w_0J}$ does not vanish in $F_n|^v$, and so $F_n|^v$ is nonzero.
\medskip

(ii) First note that the classification of zero ideals in Theorem~\ref{thm:toric_op_schu_diag}(i) shows that every permutation $v$ in $A_1^{\text{op}}$ is either of the form  $v = (n-2, n, n-1, \dots)$ or 
   $v = (n-2, n-1, n, \dots)$.
We partition $A_1^{\text{op}}$ into \[
A_{1A}^{\text{op}} = \{v \in A_1^{\text{op}} : v_2 = n \} \text{ and }A_{1B}^{\text{op}} = \{ v \in A_1^{\text{op}} : v_3 = n \}.
\]
The proof of {\bf (ii)} follows from a series of lemmas which we prove later. By Lemma~\ref{lem:a1b_op_schu} and Lemma~\ref{lem:a1a_op_schu} we have that $A_1^{\text{op}} \subset T_{n}^{\text{op}}$. By Lemma~\ref{lem:a2_op_schu_diag} we have $A_2^{\text{op}} \subset T_{n}^{\text{op}}$. 

For the converse let $v = (v_1, \dots, v_n) \in T_{n}^{\text{op}}$. By Lemma~\ref{lem:toric_op_schu_diag_asc_prop} we have that $v$ has the ascending property. Then by Lemma~\ref{lem:n_op_schu_diag} we have that $\underline{v} \in T_{n-1}^{\text{op}}$ or $\underline{v} \in Z_{n-1}^{\text{op}}$. So $v \in A_2^{\text{op}}$ or $v \in A_1^{\text{op}}$ respectively by definition of $A_1^{\text{op}}$ and $A_2^{\text{op}}$. 
\end{proof}

We now prove the lemmas used in the proof of Theorem~\ref{thm:toric_op_schu_diag}(ii). Here we assume that $F_n^v$ is nonzero.

\begin{lemma}\label{lem:toric_op_schu_diag_asc_prop}
If $F_n|^v$ has no monomial, then $v$ has the ascending property.
\end{lemma}
\begin{proof}
Suppose that $v$ does not have the ascending property. Then we choose the minimum index $r$ such that $1\le r \le t-2$ and $v_r>v_{r+1}$. 
To obtain a contradiction, we consider the following two cases on $v$ in which we find a monomial $P_IP_J$ in $F_n|^v$ arising from a binomial $P_IP_J-P_{I'}P_{J'}$ in $F_n$.

\textbf{Case 1.} Let $v_{r+1} < v_1$. We choose \[ I=\{v_2,v_3, \ldots, v_r, v_{r+1}, n\},\    J=\{ v_{1}\},\] \[  I'=\{v_1,v_2, \ldots, v_r, n\}  \text{ and }  J'=\{v_{r+1}\}.\]

\textbf{Case 2.} Assume that there exists $1\leq k\leq r-1$ such that $v_k<v_{r+1}<v_{k+1}$.  We choose
\[
I=\{ v_1, \ldots, v_{k}, v_{k+2},\ldots v_r, v_{r+1}, n\}, \  
J=\{v_{1}, \ldots,  v_{k}, v_{k+1}\},
\]
\[
I'=\{ v_1, \ldots, v_{r}, n\} \text{ and } 
J'=\{ v_1, \ldots, v_{k}, v_{r+1} \}.
\]
In all the above cases, we have that $P_IP_J-P_{I'}P_{J'}$ is nonzero in $F_n$. Since $v_1<v_2<\cdots <v_r$, we have $v\le I, J$ and $P_IP_J$ does not vanish in $F_n|^v$.    
Also observe that $v\not \le J'$ and so $P_{J'}$ vanishes in $F_n|^v$. Hence $P_IP_J$ is a  monomial in $F_n|^v$, a contradiction.
Therefore we conclude that $v$ has the ascending property.
\end{proof}

\begin{lemma}\label{lem:a1b_op_schu}
$F_n|^v$ is a principal toric ideal for each $v \in A_{1B}^{\text{op}}$. In particular, $A_{1B}^{\text{op}} \subset T_{n}^{\text{op}}$.
\end{lemma}

\begin{proof}
Let $v \in A_{1B}^{\text{op}}$ then we have $v = (n-2, n-1, n, \dots)$. First we note that $F_n|^v$ is nonzero because it contains the binomial $P_{n-1} P_{n-2, n} - P_{n-2} P_{n-1, n}$.

Now we show that the above binomial generates $F_n|^v$. Let $P_I P_J - P_{I'} P_{J'}$ be a binomial in $F_n$ such that $P_I P_J$ does not vanish in $F_n|^v$. Assume $|I| \le |J|$, we show by contradiction that $|I| = 1$ and $|J| = 2$. Since $|I| = |J| = 1$ is impossible there are three cases, either $|I|, |J| \ge 3$, $|I| \le 2$ and $|J| \ge 3$ or $|I| = |J| = 2$.

\textbf{Case 1.} Let $|I|, |J| \ge 3$. Since $v_3 = n$ and $I, J \ge v$ therefore $n \in I$ and $n \in J$. It follows that
\[
P_{I \backslash \{ n\} } P_{J \backslash \{ n\} }  - P_{I' \backslash \{ n\} } P_{J' \backslash \{ n\} } 
\]
is a non-vanishing binomial $F_{n-1}|^{\underline{v}}$, a contradiction since we assumed $\underline{v} \in Z_{n-1}^{\text{op}}$.

\textbf{Case 2.} Let $|I| \le 2$ and $|J| \ge 3$. Since $J \ge v$, we have that $J = \{\dots, n-2, n-1, n \}$. Let us write $I = \{i_1, i_2 \}$ where $i_1 < i_2$ if $|I| = 2$ and write $I = \{i_1 \}$ if $|I| = 1$. Since $I \ge v$ we have $i_1 \ge n-2$ and $i_2 \ge n-1$. Observe that swapping any value in $I$ with some value in $J$, either results in no change to $I$ and $J$, or results in a repeated entry in $J$. Therefore $P_I P_J$ is not a term of any binomial in $F_n$, a contradiction. 

\textbf{Case 3.} Let $|I| = |J| = 2$. Since $v_3 = n$, it follows that $I, J \ge \underline{v}$ hence $P_IP_J$ does not vanish in $F_{n-1}|^{\underline v}$. So $F_{n-1}|^{\underline v}$ is nonzero, a contradiction.

So we have shown that $|I| = 1$ and $|J| =  2$. Let us write $I = \{ i_1\}$ and $J = \{j_1, j_2 \}$ where $j_1 < j_2$. Since $P_IP_J \neq P_{I'}P_{J'}$ it follows that $i_1, j_1$ and $j_2$ are distinct and $I' = \{j_1\}$ and $J' = \{i_1, j_2 \}$. Since $I, J \ge v$ we must have $\{i_1, j_1, j_2 \} = \{n-2, n-1, n \}$. So the binomial is $P_{n-1}P_{n-2, n} - P_{n-2} P_{n-1, n}$. In particular $F_n|^v$ is a principal toric ideal.
\end{proof}

\begin{lemma}\label{lem:a1a_op_schu}
$A_{1A}^{\text{op}} \subset T_{n}^{\text{op}}$.
\end{lemma}
\begin{proof}
Let $v \in A_{1A}^{\text{op}}$ then we have $v = (n-2, n, n-1, \dots)$. First we show that $F_n|^v$ is nonzero. Consider the binomial
\[
P_{n-1} P_{n-2, n} - P_{n-2} P_{n-1, n}.
\]
The above binomial does not vanish in $F_n|^v$ hence the ideal is nonzero. We now show that $F_n|^v$ is principal. Let $P_I P_J - P_{I'} P_{J'}$ be a binomial in $F_n$ such that $P_I P_J$ does not vanish in $F_n|^v$. Without loss of generality we assume $|I| \le |J|$. We will show by contradiction that $|I| = 1$ and $|J| = 2$.
Since $|I| = |J| = 1$ is impossible there are three cases, either $|I|, |J| \ge 3$, $|I| \le 2$ and $|J| \ge 3$ or $|I| = |J| = 2$.

\textbf{Case 1.} Let $|I|, |J| \ge 3$. Since $v_2 = n$ and $I, J \ge v$ therefore $n \in I$ and $n \in J$. It follows that
\[
P_{I \backslash \{ n\} } P_{J \backslash \{ n\} }  - P_{I' \backslash \{ n\} } P_{J' \backslash \{ n\} } 
\]
is a binomial which does not vanish in
$F_{n-1}|^{\underline{v}}$, a contradiction since we assumed $\underline{v} \in Z_{n-1}^{\text{op}}$.

\textbf{Case 2.} Let $|I| \le 2$ and $|J| \ge 3$. Since $J \ge v$, we have that $J = \{\dots, n-2, n-1, n \}$. Let us write $I = \{i_1, i_2 \}$ where $i_1 < i_2$ if $|I| = 2$ and write $I = \{i_1 \}$ if $|I| = 1$. Since $I \ge v$ we have $i_1 \ge n-2$ and $i_2 = n$. Observe that swapping any value in $I$ with any value in $J$ either results in no change to $I$ and $J$ or results in a repeated entry in $J$. Therefore $P_I P_J$ is not a term of any binomial in $F_n$, a contradiction. 

\textbf{Case 3.} Let $|I| = |J| = 2$. Let us write $I = \{ i_1, i_2 \}$ and $J = \{j_1, j_2 \}$ where $i_1 < i_2$ and $j_1 < j_2$. Since $I, J \ge v$ it follows that $i_2 =  j_2 = n$. So we must have $n \in I'$ and $n \in J'$ and the polynomial $P_I P_J - P_{I'} P_{J'} = 0$ is zero, a contradiction.

We have shown that $|I| = 1$ and $|J| =  2$. Let us write $I = \{ i_1\}$ and $J = \{j_1, j_2 \}$ where $j_1 < j_2$. We must have no repeated values among $i_1, j_1, j_2$ otherwise the binomial is trivial. Since $I, J \ge v$ we must have $\{i_1, j_1, j_2 \} = \{n-2, n-1, n \}$. It is now easy to see that $P_I P_J - P_{I'} P_{J'}$ is exactly of the form described above. Since we are in the diagonal case, we have that both left and right hand side of the binomial do not vanish in $F_n|^v$. So $F_n|^v$ is generated by $P_I P_J - P_{I'}P_{J'}$ and is a principal toric ideal. 
\end{proof}

\begin{lemma}\label{lem:n_op_schu_diag}
If $\underline{v} \in N_{n-1}^{\text{op}}$ then $v \in N_{n}^{\text{op}}$.
\end{lemma}
\begin{proof}
{Let $r$ be such that $v_r = n$.}
Let $P_I P_J - P_{I'} P_{J'}$ be in $F_{n-1}$
such that $P_I P_J$ does not vanish and $P_{I'} P_{J'}$ does vanish in $F_{n-1}|^{\underline{v}}$. We may assume, without loss of generality, $|I| = |I'|$ and $|J| = |J'|$. Define
\[
\tilde{I} = \left\{
\begin{tabular}{lr}
    $I$ & if $|I| < r$, \\
    $I \cup \{ n \}$ & if $|I| \ge r$. 
\end{tabular}
\right.
\]
Similarly, we define $\tilde{J}, \tilde{I'}$ and $ \tilde{J'}$. By construction, for each $X \in \{I, J, I', J' \}$ we have $X \ge \underline{v}$ if and only if $\tilde{X} \ge v$. In particular $P_{\tilde{I}} P_{\tilde{J}}$ does not vanish in $F_n|^v$ and $P_{\tilde{I'}} P_{\tilde{J'}}$ does vanish. 
Also by construction we have that $P_{\tilde{I}} P_{\tilde{J}} - P_{\tilde{I'}} P_{\tilde{J'}}$ is a binomial in $F_n$. So we have shown $F_n|^v$ contains the monomial $P_{\tilde{I}} P_{\tilde{J}}$.
\end{proof}

\begin{lemma}\label{lem:a2_op_schu_diag}
$A_2^{\text{op}} \subset T_{n}^{\text{op}}$.
\end{lemma}
\begin{proof}
Suppose that $v \in A_2^{\text{op}}$ so we have $\underline{v} \in T_{n-1}^{\text{op}}$. By assumption $v$ has the ascending property so we write $v_1 < v_2 < \dots < v_r = n$ for some $r \in [n]$. Let $P_I P_J - P_{I'} P_{J'}$ be in $F_n$ where $P_I P_J$ does not vanish in $F_n|^v$. We will show that $P_{I'} P_{J'}$ does not vanish in $F_n|^v$. Without loss of generality we assume $|I| = |I'|$, $|J| = |J'|$.

{Suppose that $n \in I, J$. We form the sets $\tilde{I}, \tilde{J}, \tilde{I'}, \tilde{J'}$ by removing $n$, that is $\tilde{I} = I \backslash \{ n \}$ and similarly for the other sets.
It follows that $P_{\tilde{I}} P_{\tilde{J}} - P_{\tilde{I'}} P_{\tilde{J'}}$ is either zero or a binomial
in $F_{n-1}$. By Lemma~\ref{lem:asc_prop},  $P_{\tilde{I}} P_{\tilde{J}}$ does not vanish in $F_{n-1}|^{\underline{v}}$. Since $F_{n-1}|^{\underline{v}}$ does not contain a monomial, $P_{\tilde{I'}} P_{\tilde{J'}}$ does not vanish in $F_{n-1}|^{\underline{v}}$. So by Lemma~\ref{lem:asc_prop}, $P_{I'} P_{J'}$ does not vanish in $F_n|^v$.
}

{
Suppose that $n \notin I$ and $n \in J$. By assumption $I \ge v$ so we must have $|I| < r$. Since $P_IP_J - P_{I'}P_{J'} \in F_n$, it is straightforward to show that $n \in J'$. Since $v$ and $\ul v$ agree on the first $|I|$ entries, we have $I \ge \ul v$. As above, we define $\tilde J = J \backslash \{n \}$ and $\tilde J' = J \backslash \{n\}$. By Lemma~\ref{lem:asc_prop} we have $\tilde J' \ge \ul v$. We also have that $P_IP_{\tilde J} - P_{I'}P_{\tilde J'}$ is either zero or a binomial relation in $F_{n-1}$. Since $\ul v \in T_{n-1}^{op}$, we have that $I', \tilde J' \ge \ul v$. Since $\ul v$ and $v$ agree on the first $|I'|$ entries we have $I' \ge v$. By Lemma~\ref{lem:asc_prop} we have $J' \ge v$. And so $P_{I'}P_{J'}$ does not vanish in $F_n|^v$. The other cases follow similarly.
}

Finally, we must show that $F_n|^v$ is nonzero. Since $F_{n-1}|^{\underline{v}}$ is nonzero, it contains a binomial $P_I P_J - P_{I'} P_{J'}$. By Lemma~\ref{lem:asc_prop} we have that $P_{I \cup \{ n\}} P_{J \cup \{ n\}} - P_{I' \cup \{ n\}}  P_{J' \cup \{ n\}} $ is a non-vanishing relation in $F_n|^v$. Hence $F_n|^v$ is nonzero.
\end{proof}


\section{Standard monomial theory}\label{sec:standard_monomial}

In this section we will study monomial bases of the ideals $G_{k,n,\ell}|_w^v$, $F_n|_w^v$ and $\init_{\bf w_\ell}(I(X_w^v))$ for Richardson varieties inside the Grassmannian and flag variety. We will show that if $G_{k,n,\ell}|_w^v$ is monomial-free then $G_{k,n,\ell}|_w^v = \init_{\bf w_\ell}(I(X_w^v))$ and $G_{k,n,\ell}|_w^v$ is a toric (prime binomial) ideal, assuming that $\init_{\bf w_\ell}(I(X_w^v))$ is quadratically generated. We will see that if $G_{k,n,\ell}|_w^v$ is monomial-free then $\init_{\bf w_\ell}(I(X_w^v))$ is, in fact, quadratically generated.
We prove this by showing that if $G_{k,n,\ell}|_w^v$ is monomial-free then it is the kernel of a monomial map. Similarly for the flag variety.

It will be important to consider generating sets of the ideals $G_{k,n,\ell}|_w^v$ and $F_n|_w^v$. We construct quadratic generating sets for these ideals as follows.

\begin{definition}
Let $G \subset \mathbb{K}[x_1, \dots, x_n]$ be a collection of homogeneous quadratic polynomials and $S \subseteq \{x_1, \dots, x_n \}$ be a collection of variables. We identify $S$ with its characteristic vector, i.e. $S_i = 1$ if $x_i \in S$ otherwise $S_i = 0$. For each $g \in G$ we write $g = \sum_{\alpha} c_\alpha x^\alpha$ and define
\[
\hat g = \sum_{S \cdot\alpha = 0} c_\alpha x^\alpha.
\]
We define $G_S = \{\hat g : g \in G \}$ to be the collection of all such polynomials.
\end{definition}

\begin{lemma}\label{lem:elim_ideal_gen_set}
Let $G \subseteq \mathbb{K}[x_1, \dots, x_n]$ be a set of homogeneous quadrics and $S \subseteq \{x_1, \dots, x_n \}$ a subset of variables. Then $\langle G_S \rangle = \langle G \cup S \rangle \cap \mathbb{K}[\{x_1, \dots, x_n\} \backslash S]$.
\end{lemma}

\begin{proof}
To show that $G_S \cup S$ and $G \cup S$ generate the same ideal, for each $g \in G$ we write $g = \sum_\alpha c_\alpha x^\alpha$, for some $c_\alpha \in \mathbb{K}$. We have that
\[
g - \hat g = \sum_{S\cdot \alpha \ge 1} c_\alpha x^\alpha.
\]
Each term appearing in the above sum is divisible by some variable in $S$, hence $\hat g \in \langle G \cup S \rangle$ and $g \in \langle G_S \cup S \rangle$.

For any polynomial $ f \in \langle G \cup S \rangle \cap \mathbb{K}[\{x_1, \dots, x_n\} \backslash S]$ we have
$
f = \sum_{g \in g} h_g g + \sum_{x_i \in S} h_i x_i,
$
for some $h_g, h_i \in \mathbb{K}[x_1, \dots, x_n]$. For each $h_g$ we define $\hat h_g$ similarly to $\hat g$ and rewrite this polynomial as
\[
f = \sum_{g \in G} \hat h_g \hat g + \left(\sum_{g \in G} (h_g g - \hat h_g \hat g) + \sum_{x_i \in S} h_i x_i
\right).
\]
Each monomial appearing in $\sum_{g \in G} \hat h_g \hat g$ is not divisible by any monomial in $S$. However each monomial appearing in the expressions
$\sum_{g \in G} (h_g g - \hat h_g \hat g)$ and $\sum_{x_i \in S} h_i x_i$ is divisible by some $x_i \in S$. Since $f \in \mathbb{K}[\{x_1, \dots, x_n\} \backslash S]$ it follows that the bracketed expression above is zero and so $f = \sum_{g \in G} \hat h_g \hat g \in \langle G_S \rangle$.
\end{proof}

\subsection{Grassmannians.}

We begin by defining a new monomial map. The kernel of this monomial map will coincide with $\init_{\bf w_\ell}(I(X_w^v))$ when $G_{k,n,\ell}|_w^v$ is monomial-free.

\begin{definition}[Restricted monomial map]
Fix $k \le n$ and let $v \le w$ be permutations in $S_n$. Let $R|_w^v = \mathbb K[P_I : I \subseteq [n], |I| = k,\  v \le I \le w]$ and $S = \mathbb K[x_{i,j} : i \in [n-1], j \in [n] ]$ be polynomial rings. We define the map $\phi_\ell|_w^v : R|_w^v \rightarrow S$ to be the restriction of the monomial map $\phi_\ell$ to the ring $R|_w^v$.
\end{definition}

\noindent\textbf{Notation.} Fix $k \le n$, $\ell \in \{0, \dots, n-1 \}$ natural numbers and $v \le w$ permutations. We use the following shorthand notation for ideals.
\begin{itemize}
  \item $J_1 := G_{k,n,\ell}|_w^v$, the matching field ideal with $P_I = 0$ if $I \in S_w^v$.  
  \item $J_2 := {\rm in}_{{\bf w}_\ell}(I(X_w^v))$, the initial ideal of the ideal of the Richardson variety $X_w^v$.
  \item $J_3 := \ker(\phi_\ell|_w^v)$, the kernel of the restricted monomial map. 
\end{itemize}

The matching field ideal $G_{k,n,\ell}$ is quadratically generated and is the kernel of a monomial map by Theorem~\ref{thm:JAlebra}. We will show that a quadratic generating set of $G_{k,n,\ell}$ naturally gives rise to a quadratic generating set of $J_1 = G_{k,n,\ell}|_w^v$.

\begin{lemma}\label{lem:J_1=J_3_binomial}
The ideals $J_1$ and $J_3$ coincide if and only if $J_1$ is monomial-free.
\end{lemma}

\begin{proof}
Note that $J_3$ is the kernel of a monomial map and therefore does not contain any monomials. So, if $J_1$ contains a monomial then it is not equal to $J_3$.

Suppose $J_1$ does not contain any monomials. Let $G$ be a quadratic generating set for $G_{k,n,\ell}$ and let $S = \{ P_I : I \in S_w^v\}$ be the collection of vanishing Pl\"ucker variables. By definition $J_1 = \langle G \cup S \rangle \cap \mathbb K[P_I : I \notin S_w^v]$. So by Lemma~\ref{lem:elim_ideal_gen_set} we have $J_1$ is generated by $G_S$. Since $J_1$ is monomial-free, we have that $G_S$ does not contain any monomials. By Theorem~\ref{thm:JAlebra}, the ideal $G_{k,n,\ell}$ is the kernel of a monomial map $\phi_\ell$ and by definition $J_3$ is the kernel of the restriction $\phi_\ell|_w^v$. Since all binomials $m_1 - m_2 \in G_S$ lie in $G_{k,n,\ell}$ and contain only the non-vanishing Pl\"ucker variables $P_J$ for $J \notin S_w^v$, therefore $m_1 - m_2 \in J_3$. And so we have $J_1 \subseteq J_3$. Also, for any polynomial $f \in J_3$ we have that $f \in G_{k,n,\ell}$. Since $f$ contains only the non-vanishing Pl\"ucker variables, therefore $f \in J_1$.
\end{proof}

\begin{lemma}\label{lem:J_1_subset_J_2}
$J_1 \subseteq J_2$.
\end{lemma}

\begin{proof}
Let $G$ be a quadratic binomial generating set for $G_{k,n,\ell}$ and $S = \{P_I : I \in S_w^v \}$. Let $\hat f \in G_S \subset J_1$ be any polynomial. By the definition of $G_S$, there exists $f \in G$ such that $\hat f$ is obtained from $f$ by setting some variables to zero. Recall $G_{k,n,\ell} = \textrm{in}_{{\bf w}_\ell}(G_{k,n})$, so there exists a polynomial $g \in G_{k,n}$ such that $f = \textrm{in}_{{\bf w}_\ell}(g)$. Since the leading term of $g$ is not set to zero in $I(X_w^v)$, it follows that $\hat f \in \textrm{in}_{{\bf w}_\ell}(I(X_w^v))$.
\end{proof}

\begin{theorem}\label{lem:std_monomials_diag_ell}
If $J_1$ does not contain any monomials then the number of standard monomials in degree two of $J_3$ and $G_{k,n,0}|_w^v$ are equal.
\end{theorem}

To prove this result we will show that there is a bijection between the semi-standard Young tableax and a collection of standard monomials for $J_3$ of degree two. We define the following map.

\begin{definition}\label{def:Gamma_ell_deg2}
Let $T$ be a semi-standard Young tableau with two columns and $k$ rows whose entries lie in $[n]$. For each $\ell \in \{1, \dots, n-1\}$ we define the map $\Gamma_\ell : T \mapsto T'$ where $T'$ is a tableau whose columns are ordered according to the matching field $B_\ell$. Suppose that the entries of the columns of $T$ are $I = \{i_1 < i_2 < \dots < i_k \}$ and $J = \{j_1 < j_2 < \dots < j_k\}$. Since $T$ is in semi-standard form, we assume that $i_s \le j_s$ for each $s \in [k]$. We define $T'$ as the tableau whose columns are $I'$ and $J'$ as sets and are ordered by the matching field $B_\ell$. The sets $I'$ and $J'$ are defined as follows.
\begin{itemize}
    \item If $i_1, i_2, j_1 \in \{1, \dots, \ell \}$, $j_2 \in \{\ell+1, \dots, n \}$ and $i_1 < j_1 < i_2$ then we define $I' = \{j_1 < i_2 < i_3 < \dots < i_k \}$ and $J' = \{i_1 < j_2 < j_3 < \dots < j_k\}$.
    \item If $i_1 \in \{1, \dots, \ell \}$, $i_2, j_1, j_2 \in \{\ell+1, \dots, n \}$ and $j_1 < i_2 < j_2$ then we define $I' = \{j_1 < i_2 < i_3 < \dots < i_k \}$ and $J' = \{i_1 < j_2 < j_3 < \dots < j_k\}$
    \item Otherwise we define $I' = I$ and $J' = J$.
\end{itemize}
\end{definition}

\begin{lemma}\label{lem:SSYT_Gamma_injective}
Let $T_1$ and $T_2$ be semi-standard Young tableaux. If $\Gamma_\ell(T_1)$ and $\Gamma_\ell(T_2)$ are row-wise equal then $T_1$ and $T_2$ are equal.
\end{lemma}

\begin{proof}
We begin by noting that all rows except possibly the first two rows are of a tableau are fixed by $\Gamma_\ell$. So it remains to show that if the first two rows of $\Gamma_\ell(T_1)$ and $\Gamma_\ell(T_2)$ are row-wise equal then so are the first two rows of $T_1$ and $T_2$. We also note that $\Gamma_\ell$ preserves the entries of a tableau, thought of as a multi-set. Let us assume by contradiction that $T_1$ and $T_2$ are not row-wise equal. By the above facts we may assume without loss of generality that
\[
T_1 = 
\begin{tabular}{|c|c|}
\hline
    $i_1$ & $j_1$  \\
\hline
    $i_2$ & $j_2$ \\
\hline
    \vdots & \vdots \\
\hline
\end{tabular} \, ,
\quad
T_2 = 
\begin{tabular}{|c|c|}
\hline
    $i_1$ & $i_2$  \\
\hline
    $j_1$ & $j_2$ \\
\hline
    \vdots & \vdots \\
\hline
\end{tabular}
\]
and $j_1 < i_2$. We proceed by taking cases on $s = |\{i_1, i_2, j_1, j_2 \} \cap \{1, \dots, \ell \}|$.

\textbf{Case 1.} Assume $s = 0$ or $4$. It follows that $\Gamma_\ell$ fixes $T_1$ and $T_2$. By row-wise equality of the second row of $\Gamma_\ell(T_1)$ and $\Gamma_\ell(T_2)$ we have that $j_1 = i_2$, a contradiction.

\textbf{Case 2.} Assume $s = 1$. It follows that $i_1 \in \{1, \dots, \ell \}$ and $i_2, j_1, j_2 \in \{\ell+1, \dots, n \}$. Since $j_1 < i_2$ we have
\[
\Gamma_\ell(T_1) = 
\begin{tabular}{|c|c|}
\hline
    $j_1$ & $j_2$  \\
\hline
    $i_2$ & $i_1$ \\
\hline
    \vdots & \vdots \\
\hline
\end{tabular} \, ,
\quad
\Gamma_\ell(T_2) = 
\begin{tabular}{|c|c|}
\hline
    $j_1$ & $i_2$  \\
\hline
    $i_1$ & $j_2$ \\
\hline
    \vdots & \vdots \\
\hline
\end{tabular}\, .
\]
By row-wise equality of the second row, we have that $j_2 = i_2$. However in the tableau $T_2$ we have that $i_2 < j_2$, a contradiction.

\textbf{Case 3.} Assume $s = 2$. Since $j_1 < i_2$, it follows that $i_1, j_1 \in \{1, \dots, \ell \}$ and $i_2, j_2 \in \{\ell+1, \dots, n \}$. And so we have
\[
\Gamma_\ell(T_1) = 
\begin{tabular}{|c|c|}
\hline
    $i_2$ & $j_2$  \\
\hline
    $i_1$ & $j_1$ \\
\hline
    \vdots & \vdots \\
\hline
\end{tabular} \, ,
\quad
\Gamma_\ell(T_2) = 
\begin{tabular}{|c|c|}
\hline
    $i_1$ & $i_2$  \\
\hline
    $j_1$ & $j_2$ \\
\hline
    \vdots & \vdots \\
\hline
\end{tabular}\, .
\]
By row-wise equality of second row the tableau we have that $i_1 = j_2$, a contradiction since $i_1 \in \{1, \dots, \ell \}$ and $j_2 \in \{\ell+1, \dots, n \}$.

\textbf{Case 4.} Assume that $s = 3$. It follows that $i_1, i_2, j_1 \in \{ 1, \dots, \ell\}$ and $j_2 \in \{\ell+1, \dots, n \}$. And so we have
\[
\Gamma_\ell(T_1) = 
\begin{tabular}{|c|c|}
\hline
    $j_1$ & $j_2$  \\
\hline
    $i_2$ & $i_1$ \\
\hline
    \vdots & \vdots \\
\hline
\end{tabular} \, ,
\quad
\Gamma_\ell(T_2) = 
\begin{tabular}{|c|c|}
\hline
    $i_1$ & $j_2$  \\
\hline
    $j_1$ & $i_2$ \\
\hline
    \vdots & \vdots \\
\hline
\end{tabular}\, .
\]
By row-wise equality of the second row, we have that $i_1 = j_1$. However in $T_2$, we have that $i_1 < j_1$, a contradiction.
\end{proof}

\begin{lemma}\label{lem:SSYT_Gamma_surj_any_tableau}
Let $T$ be any tableau whose columns are valid for the block diagonal matching field $B_\ell$. Then there exists a semi-standard Young tableau $T'$ such that $\Gamma_\ell(T')$ and $T$ are row-wise equal.
\end{lemma}

\begin{proof}
Let $T$ be the tableau with entries $\{i_1, i_2 < i_3 < \dots < i_k \}$ and $\{j_1, j_2 < j_3 < \dots < j_k \}$,
\[
T = 
\begin{tabular}{|c|c|}
\hline
    $i_1$ & $j_1$  \\
\hline
    $i_2$ & $j_2$ \\
\hline
    \vdots & \vdots \\
\hline
    $i_k$ & $j_k$ \\
\hline
\end{tabular}\, .
\]
Without loss of generality we may assume that $i_s \le j_s$ for all $s \ge 3$. We proceed by taking cases on $s = |\{i_1, i_2, j_1, j_2 \} \cap \{ 1, \dots, \ell\}|$. 

\textbf{Case 1.} Assume $s = 0$ or $4$. We have that $i_1 < i_2$ and $j_1 < j_2$. So we may order the entries in row to obtain $T'$. Note that in this case $\Gamma_\ell$ fixes $T'$.

\textbf{Case 2.} Assume $s = 1$. Without loss of generality we assume $j_2 \in \{1, \dots, \ell \}$.
\begin{itemize}
    \item If $j_1 > i_2$ then
    \[
    \Gamma_\ell \left( \,
    \begin{tabular}{|c|c|}
    \hline
        $j_2$ & $i_1$  \\
    \hline
        $i_2$ & $j_1$ \\
    \hline
        \vdots & \vdots \\
    \hline
    \end{tabular}\,
    \right)
    =
    \begin{tabular}{|c|c|}
    \hline
        $i_1$ & $j_1$  \\
    \hline
        $i_2$ & $j_2$ \\
    \hline
        \vdots & \vdots \\
    \hline
    \end{tabular}\, .
    \]
    \item If $j_1 \le i_2$ then
    \[
    \Gamma_\ell \left( \,
    \begin{tabular}{|c|c|}
    \hline
        $j_2$ & $i_1$  \\
    \hline
        $j_1$ & $i_2$ \\
    \hline
        \vdots & \vdots \\
    \hline
    \end{tabular}\,
    \right)
    =
    \begin{tabular}{|c|c|}
    \hline
        $j_1$ & $i_1$  \\
    \hline
        $j_2$ & $i_2$ \\
    \hline
        \vdots & \vdots \\
    \hline
    \end{tabular}\, .
    \]
    The tableau on the right is row-wise equal to $T$.
\end{itemize}

\textbf{Case 3.} Assume $s = 2$. 
\begin{itemize}
    \item If $i_1, i_2 \in \{1, \dots, \ell \}$ then $\Gamma_\ell$ fixes each column of $T$, which is a semi-standard Young tableau.
    \item If $i_2, j_2 \in \{1, \dots, \ell \}$ then without loss of generality assume $i_2 \le j_2$ and $i_1 \le j_1$. We have
    \[
    \Gamma_\ell \left( \,
    \begin{tabular}{|c|c|}
    \hline
        $i_2$ & $j_2$  \\
    \hline
        $i_1$ & $j_1$ \\
    \hline
        \vdots & \vdots \\
    \hline
    \end{tabular}\,
    \right)
    =
    \begin{tabular}{|c|c|}
    \hline
        $i_1$ & $j_1$  \\
    \hline
        $i_2$ & $j_2$ \\
    \hline
        \vdots & \vdots \\
    \hline
    \end{tabular}\, .
    \]
\end{itemize}

\textbf{Case 4.} Assume $s = 3$. Without loss of generality we may assume $j_1 \in \{\ell+1, \dots, n \}$.
\begin{itemize}
    \item If $j_2 < i_1$ then
    \[
    \Gamma_\ell \left( \,
    \begin{tabular}{|c|c|}
    \hline
        $j_2$ & $i_1$  \\
    \hline
        $i_2$ & $j_1$ \\
    \hline
        \vdots & \vdots \\
    \hline
    \end{tabular}\,
    \right)
    =
    \begin{tabular}{|c|c|}
    \hline
        $i_1$ & $j_1$  \\
    \hline
        $i_2$ & $j_2$ \\
    \hline
        \vdots & \vdots \\
    \hline
    \end{tabular}\, .
    \]
    Note that in this case we have $j_2 < i_1 < i_2$ and so the tableau on the left is a semi-standard Young tableau.
    \item If $j_2 \ge i_1$ then 
    \[
    \Gamma_\ell \left( \,
    \begin{tabular}{|c|c|}
    \hline
        $i_1$ & $j_2$  \\
    \hline
        $i_2$ & $j_1$ \\
    \hline
        \vdots & \vdots \\
    \hline
    \end{tabular}\,
    \right)
    =
    \begin{tabular}{|c|c|}
    \hline
        $i_1$ & $j_1$  \\
    \hline
        $i_2$ & $j_2$ \\
    \hline
        \vdots & \vdots \\
    \hline
    \end{tabular}\, .
    \]
\end{itemize}
\end{proof}

\begin{lemma}\label{lem:basis_rich_corresp}
Let $v \le w$ be permutations. A semi-standard Young tableau $T$ vanishes in $G_{k,n,0}|_w^v$ if and only if $\Gamma_\ell(T)$ vanishes in $G_{k,n,\ell}|_w^v$.
\end{lemma}

\begin{proof}
Let $I, J$ be the columns of $T$ and $I', J'$ be the columns of $\Gamma_\ell(T)$. The result follows from the fact that $\{\min(I), \min(J)\} = \{ \min(I'), \min(J')\}$ and similarly for the second smallest of elements of $I, J, I'$ and $J'$.
\end{proof}

By the results of Kreiman and Lakshmibai, see \cite{kreiman2002richardson}, the semi-standard Young tableau for $X_w^v$ are the tableaux such that each column $I$ satisfies $v \le I \le w$.

\begin{lemma}\label{lem:SSYT_Gamma_surj_restrict}
If $J_1$ is monomial-free then the set
\[
\textrm{Im}(\Gamma_\ell)|_w^v = \{\Gamma_\ell(T) : T \textrm{ a two column semi-standard Young tableau for } X_w^v \}
\]
is a monomial basis for $J_3$ in degree two. 
\end{lemma}

\begin{proof}
We prove the contrapositive, i.e. if $\textrm{Im}(\Gamma_\ell)|_w^v$ is not a monomial basis for $J_3$ then $J_1$ contains a monomial. Let $T$ be a matching field tableau for $B_\ell$ representing a monomial in $G_{k,n,\ell}|_w^v$ which does not lie in the span of $\textrm{Im}(\Gamma_\ell)|_w^v$. Since $\textrm{Im}(\Gamma_\ell)$ is a basis for $G_{k, n, \ell}$, it follows that $T$ is row-wise equal to $\Gamma_\ell(T')$ for some semi-standard Young tableau $T'$ which vanishes in $G_{k,n,\ell}|_w^v$. We write $I, J$ for the columns of $T$ and $I', J'$ for the columns of $\Gamma_\ell(T')$. Since
$T$ and $\Gamma_\ell(T')$ are row-wise equal we may assume that all their entries below the second row are in semi-standard form. So we write
\[
T = 
\begin{tabular}{|c|c|}
\multicolumn{1}{c}{$I$} & \multicolumn{1}{c}{$J$} \\
\hline
    $i_1$ & $j_1$  \\
\hline
    $i_2$ & $j_2$ \\
\hline
    \vdots & \vdots \\
\hline
\end{tabular}\, , \quad
\Gamma_\ell(T') = 
\begin{tabular}{|c|c|}
\multicolumn{1}{c}{$I'$} & \multicolumn{1}{c}{$J'$} \\
\hline
    $i_1$ & $j_1$  \\
\hline
    $j_2$ & $i_2$ \\
\hline
    \vdots & \vdots \\
\hline
\end{tabular}\, .
\]
Throughout the proof we write $v =\{v_1 < \dots < v_k \}$ and $w = \{w_1 < \dots < w_k \}$ for the Grassmannian permutations. We now take cases on $s = |\{i_1, i_2, j_1, j_2 \} \cap \{1, \dots, \ell \}|$.

\textbf{Case 1.} Assume $s = 0$ or $4$. It follows that $\Gamma_\ell(T')$ is a semi-standard Young tableaux and so $\Gamma_\ell(T')$ does not vanish, a contradiction.

\textbf{Case 2.} Assume $s = 1$. Without loss of generality assume that $j_2 \in \{1, \dots, \ell \}$ and note that in this case we may possibly have that $I'$ and $J'$ are swapped in $\Gamma_\ell(T')$. Since $T$ does not vanish we have $v \le I, J \le w$. So by ordering the entries of $I, J$ in increasing order and comparing them with $v$ and $w$, we have
\[
v_1 \le  \{i_1, j_2 \} \le w_1, \quad
v_2 \le \{ i_2, j_1 \} \le w_2.
\]
Since $\Gamma_\ell(T')$ vanishes we must have that either $I'$ or $J'$ vanishes. Let us take cases.

\textbf{Case 2.1} Assume $I' = \{ i_1, j_2, \dots\}$ vanishes. We have    
\[
v_1 \le j_2 \le w_1, \quad i_1 \le w_1 < w_2
\]
and so $I' \le w$. Since $I'$ vanishes, we must have $I' \not\ge v$ and so $i_1 < v_2$. We have the following
\begin{itemize}
    \item $v_1 \in \{1, \dots, \ell \}$ because $v_1 \le j_2$,
    \item $v_2 \in \{\ell+2, \dots, n \}$ because $v_2 > i_1 \in \{\ell+1, \dots, n \}$.
\end{itemize}
By Theorems~\ref{thm:Rich} and \ref{main:zero Gr} we have that $G_{k,n,\ell}|_w^v$ contains a monomial.

\textbf{Case 2.2} Assume $J' = \{ j_1, i_2, \dots \}$ vanishes. We have
\[
v_1 < v_2 \le j_1, \quad
v_2 \le i_2 \le w_2.
\]
Therefore $J' \ge v$. Since $J'$ vanishes we have $J \not\le w$ and so must have $j_1 > w_1$. We have the following
\begin{itemize}
    \item $w_i \in \{\ell+1, \dots, n \}$ for all $i \ge 2$ because $w_2 \ge i_2 \in \{\ell+1, \dots, n \}$,
    \item $w_2 \neq w_1 + 1$ because $w_1 < j_1 < i_2 \le w_2$,
    \item $w_1 \le n-k$ because $w_1 < j_1 = \min(J') \ge n-k+1$,
    \item $w_1 \ge \ell+1$ because $w_1 \ge i_1 \in \{ \ell+1, \dots, n\}$.
\end{itemize}
And so by Theorem~\ref{thm:Rich}
and Lemma~\ref{lem:G_knlw_classification} we have that $G_{k,n,\ell}|_w^v$ contains a monomial.

\textbf{Case 3.} Assume $s = 2$.
If $i_1, i_2 \in \{1, \dots, \ell \}$ then $\Gamma_\ell(T')$ is not a valid tableau with respect to the matching field $B_\ell$. It follows that $i_2, j_2, \in \{1, \dots, \ell\}$. However it easily follows that $\Gamma_\ell(T')$ does not vanish in $G_{k, n, \ell}|_w^v$, a contradiction.

\textbf{Case 4.} Assume $s = 3$. Without loss of generality assume that $j_1 \in \{\ell+1, \dots, n \}$. Note that in this case we may possibly have that $I'$ and $J'$ are swapped in $\Gamma_\ell(T')$. Since $T$ does not vanish we have $v \le I, J \le w$. So by ordering the entries of $I, J$ in increasing order and comparing them with $v$ and $w$, we have
\[
v_1 \le \{i_1, j_2 \} \le w_1, \quad
v_2 \le \{i_2, j_1\} \le w_2.
\]
Since $\Gamma_\ell(T')$ vanishes we must have that either $I'$ or $J'$ vanishes. We proceed by taking cases.

\textbf{Case 4.1} Assume that $I' = \{i_1, j_2, \dots \}$ vanishes. We have
\[
v_1 \le i_1 \le w_1, \quad 
j_2 \le w_1 < w_2
\]
and so $I' \le w$. Since $I'$ vanishes we must have $I' \not\ge v$ and we deduce that $j_2 < v_2$. We have the following
\begin{itemize}
    \item $v_1 \in \{1, \dots, \ell \}$ because $v_1 \le i_1 \in \{1, \dots, \ell \}$,
    \item $v_2 > v_1 + 1$ because $v_1 \le i_1 < j_2 < v_2$,
    \item $v_2 \neq \ell+1$ because $v_2 \le i_2 \in \{ 1, \dots, \ell\}$.
\end{itemize}
By Theorems~\ref{thm:Rich} and \ref{main:zero Gr} we have that $G_{k,n,\ell}|_w^v$ contains a monomial.

\textbf{Case 4.2} Assume that $J' = \{j_1, i_2, \dots \}$ vanishes. We have 
\[
v_1 < v_2 \le i_2, \quad 
v_2 \le j_1 \le w_2
\]
and so $J' \ge v$. Since $J'$ vanishes we must have $J' \not\le w$ and we deduce that $i_2 > w_1$. We have the following
\begin{itemize}
    \item $w_i \in \{\ell+1, \dots, n \}$ for all $i \ge 2$ because $w_2 \ge j_1 \in \{\ell+1, \dots, n \}$,
    \item $w_2 \neq w_1 + 1$ because $w_1 < i_2 < j_1 \le w_2$,
    \item $w_1 \le n-k$ because $w_1 < i_2 = \min(J') \ge n-k+1$
    \item $w_1 \neq \ell$ because $w_1 < i_2 \in \{1, \dots, \ell \}$,
    \item $w_1 \ge 2$ because, by column $I'$, we have $i_1 < j_2 \le w_1$.
\end{itemize}
And so by Theorem~\ref{thm:Rich}
and Lemma~\ref{lem:G_knlw_classification} we have that $G_{k,n,\ell}|_w^v$ contains a monomial.
\end{proof}

\begin{proof}[Proof of Theorem~\ref{lem:std_monomials_diag_ell}]
A collection of standard monomials for $G_{k,n,0}|_w^v$ in degree two is given by semi-standard Young tableaux $T$ with two columns such that each column $I$ satisfies $v \le I \le w$. By Lemma~\ref{lem:basis_rich_corresp} we have that $\Gamma_\ell$ is a map from such semi-standard Young tableaux to matching field tableau for $B_\ell$ which do not vanish in $G_{k,n,\ell}|_w^v$. By Lemmas~\ref{lem:SSYT_Gamma_injective} and \ref{lem:SSYT_Gamma_surj_restrict} we have that this map is a bijection.
\end{proof}

Recall from Theorem~\ref{thm:JAlebra} that the initial ideals $\init_{\wb_\ell}(G_{k,n})$ are quadratically generated.

\begin{conjecture}\label{conj:J_2_quad_gen}
If $J_1$ is monomial-free then $J_2$ is quadratically generated.
\end{conjecture}

\begin{remark}
We have calculated the initial ideals $J_2 = \init_{\wb_\ell}(I(X_w^v))$ for all Richardson varieties of Grassmannians $G_{k,n}$ where $n \in \{ 4,5,6,7\}$ and $k \in \{2, \dots n-2 \}$ using the software \texttt{Macaulay2}. We have observed that all such initial ideals are quadratically generated if $J_1 = G_{k,n,\ell}|_w^v$ is monomial-free. 
\end{remark}

\begin{theorem}\label{thm:toric_degen_std_monomial_ell_0_gr}
Let $\ell = 0$. If $J_1$ is monomial-free then $J_1$, $J_2$ and $J_3$ coincide. In particular $\textrm{in}_{{\bf w}_\ell}(I(X_w^v))$ is a toric ideal.
\end{theorem}

\begin{proof}
Let $R= \mathbb K[P_I : I \subseteq [n], |I| = k, v \le I \le w]$ be the polynomial ring containing $J_1, J_2$ and $J_3$. Suppose $J_1$ is monomial-free. By Lemma~\ref{lem:J_1=J_3_binomial} we have that $J_1 = J_3$. 
By Lemma~\ref{lem:J_1_subset_J_2}, we have $J_1 \subseteq J_2$. Let $M$ be a collection of linearly independent monomials in $R / J_2$, if $M$ is linearly dependent in $R / J_1$ then we have $\sum_{m \in M} c_m m \in J_1 \subseteq J_2$ for some $c_m \in \mathbb K$. And so $M$ is linearly dependent in $J_2$, a contradiction. Hence for all $d \ge 1$, any collection of standard monomials for $J_2$ of degree $d$ is linearly independent in $R / J_1 = R / J_3$. Since $J_2 = \textrm{in}_{{\bf w}_\ell}(I(X_w^v))$ is an initial ideal of a homogeneous ideal, the number of standard monomials of degree $d$ coincides with the number of standard monomials of degree $d$ of $I(X_w^v)$. Recall that the semi-standard Young tableaux such that each column $I$ satisfies $v \le I \le w$ with $d$-columns are in bijection with a collection of standard monomials of $I(X_w^v)$ of degree $d$.

Consider the case $\ell = 0$. Two monomials are equal in $R/J_3$ if and only if their corresponding tableaux are row-wise equal. Therefore, the semi-standard Young tableaux are in bijection with standard monomials for $J_3$. And so we have $J_1 = J_2 = J_3$.
\end{proof}

\begin{corollary}[Corollary of Conjecture~\ref{conj:J_2_quad_gen}]\label{cor:toric_degen_std_monomial_gr}
Suppose $\ell \in \{1, \dots, n-1 \}$. If $J_1$ is monomial-free then $J_1, J_2$ and $J_3$ coincide. In particular $\init_{\bf w_\ell}(I(X_w^v))$ is a toric ideal.
\end{corollary}

\begin{proof}
Suppose that $J_1$ is monomial-free. By Lemmas~\ref{lem:J_1=J_3_binomial} and \ref{lem:J_1_subset_J_2}, we have $J_1 = J_3 \subseteq J_2$. And so any collection of standard monomials of degree $d$ for $J_2$ is linearly independent in $R/J_3$. By Theorem~\ref{lem:basis_rich_corresp} we have that $J_2$ and $J_3$ have the same number of standard monomials in degree $2$. Since $J_3 \subseteq J_2$ it follows that $\textrm{Im}(\Gamma_\ell)$ is a collection of standard monomials for $J_2$ and $J_3$. Suppose that Conjecture~\ref{conj:J_2_quad_gen} holds, then we have that $J_2$ is generated in degree two. Since $J_1 = J_3$ is generated in degree two, it follows that $J_1$, $J_2$ and $J_3$ coincide.
\end{proof}

\subsection{Flag varieties.}
Fix $n$ a natural number and $v \le w$ permutations. Let $R = \mathbb K[P_I : I \subseteq [n], |I| \in \{1, \dots, n-1\}, I \notin S_w^v ]$ and $S = \mathbb K[x_{i,j} : i \in \{1, \dots, n-1 \}, j \in \{1, \dots, n \}]$ be polynomial rings. We define the restricted monomial map $\phi_n|_w^v : R \rightarrow S$ for the flag variety analogously to the Grassmannian case. We use the following shorthand notation for ideals of $R$.
\begin{itemize}
  \item $J_1 := F_n|_w^v$, the restricted matching field ideal.  
  \item $J_2 := {\rm in}_{{\bf w}_0}(I(X_w^v))$, the initial ideal of the ideal of the Richardson variety.
  \item $J_3 := \ker(\phi|_w^v)$, the kernel of the restricted monomial map. 
\end{itemize}

\begin{remark}\label{rmk:ssyt_mon_basis_rich}
By \cite[Theorem V.14]{kim2015richardson}, there is collection of semi-standard Young tableaux $T$ that form a monomial basis for the Richardson variety $X_w^v$ which are called \textit{standard} tableaux. The conditions for a semi-standard Young tableau to be standard implies that each column $I$ of $T$ satisfies $v \le I \le w$. Working directly from the combinatorial conditions, one can show that if $(v, w) \in T_n^R \cup Z_n^R$ then a tableau $T$ is standard if and only if $v \le I \le w$ for each column $I$ of $T$.
\end{remark}

By Theorem~\ref{thm:Pure}, the matching field ideal $F_n$ is quadratically generated and is the kernel of the monomial map $\phi_0$ in \eqref{eqn:monomialmapflags}. Hence, by a similar argument to the Grassmannian case we have the following results.

\begin{lemma}\label{lem:J_1=J_3_binomial_flag}
The ideals $J_1$ and $J_3$ coincide if and only if $J_1$ is monomial-free.
\end{lemma}

\begin{proof}
Note that $J_3$ is the kernel of a monomial map and therefore does not contain any monomials. So, if $J_1$ contains a monomial then it is not equal to $J_3$.

Suppose $J_1$ does not contain any monomials. Let $G$ be a quadratic generating set for $F_n$ and let $S = \{ P_I : I \in S_w^v\}$ be the collection of vanishing Pl\"ucker variables. By definition $J_1 = \langle G \cup S \rangle \cap \mathbb K[P_I : I \notin S_w^v]$. So by Lemma~\ref{lem:elim_ideal_gen_set} we have $J_1$ is generated by $G_S$. Since $J_1$ is monomial-free, we have that $G_S$ does not contain any monomials. By Theorem~\ref{thm:Pure}, the ideal $F_n$ is the kernel of a monomial map $\phi_0$ and by definition $J_3$ is the kernel of the restriction $\phi_n|_w^v$. Since all binomials $m_1 - m_2 \in G_S$ lie in $F_n$ and contain only the non-vanishing Pl\"ucker variables $P_J$ for $J \notin S_w^v$, therefore $m_1 - m_2 \in J_3$. And so we have $J_1 \subseteq J_3$. Also, for any polynomial $f \in J_3$ we have that $f \in F_n$. Since $f$ contains only the non-vanishing Pl\"ucker variables, therefore $f \in J_1$.
\end{proof}

\begin{lemma}\label{lem:J_1_subset_J_2_flag}
$J_1 \subseteq J_2$.
\end{lemma}

\begin{proof}
Let $G$ be a quadratic binomial generating set for $F_n$ and $S = \{P_I : I \in S_w^v \}$. Let $\hat f \in G_S \subset J_1$ be any polynomial. By the definition of $G_S$, there exists $f \in G$ such that $\hat f$ is obtained from $f$ by setting some variables to zero. Recall $F_n = \textrm{in}_{{\bf w}_0}(I_n)$, so there exists a polynomial $g \in I_n$ such that $f = \textrm{in}_{{\bf w}_0}(g)$. Since the leading term of $g$ is not set to zero in $I(X_w^v)$, it follows that $\hat f \in \textrm{in}_{{\bf w}_0}(I(X_w^v))$.
\end{proof}

We write $SSYT_d(v,w)$ for the collection of semi-standard Young tableau $T$ with columns $I_1, \dots, I_d$ satisfying $v \le I_j \le w$ for each $j \in [d]$. We note that, in contrast to the Grassmannian, the monomials associated to the tableaux $SSYT_d(v,w)$ do not constitute a monomial basis for the Richardson variety $X_w^v$ inside the flag variety.

\begin{theorem}\label{thm:toric_degen}
If $J_1$ is monomial-free then $J_1$, $J_2$ and $J_3$ coincide. In particular $\textrm{in}_{{\bf w}_0}(I(X_w^v))$ is a toric ideal. 
\end{theorem}

\begin{proof}
Since $J_1$ is monomial-free, we have $J_1 = J_3$ by Lemma~\ref{lem:J_1=J_3_binomial_flag}. By Lemma~\ref{lem:J_1_subset_J_2_flag} we have $J_1 \subseteq J_2$. Let $M$ be a collection of linearly independent monomials in $R / J_2$, if $M$ is linearly dependent in $R / J_1$ then we have $\sum_{m \in M} c_m m \in J_1 \subseteq J_2$ for some $c_m \in \mathbb K$. And so $M$ is linearly dependent in $J_2$, a contradiction. Hence for all $d \ge 1$, any collection of standard monomials for $J_2$ of degree $d$ is linearly independent in $R / J_1 = R / J_3$. Since we are working with the diagonal matching field, the monomials associated to $SSYT_d(v,w)$ form a monomial basis for $J_1$ in degree $d$. By Remark~\ref{rmk:ssyt_mon_basis_rich}, the monomials associated to $SSYT_d(v,w)$ form a monomial basis for the Richardson variety $X_w^v$. Since $J_2 = \textrm{in}_{{\bf w}_0}(I(X_w^v))$ is an initial ideal, any monomial basis for $J_2$ in degree $d$ has size $|SSYT_d(v,w)|$. Therefore $J_1$ and $J_2$ have monomial bases of the same size in each degree. So $J_1 = J_2$.
\end{proof}


\subsection{Applications and further computations.} \label{sec:kim}

In this section we provide some further applications of our results and compare them to the results in \cite{kim2015richardson}. In particular we give explicit calculations for $\Flag_3$ and $\Flag_4$.

Let us begin by calculating the pairs of permutations $(v,w)$ for which $F_n|_w^v$ is toric, non-toric and zero using a different weight matrix. For the weight matrix we use the same convention as Kim, so for $n = 3$ and $n = 4$ we use the weight matrices
\[
M_4 = 
\begin{bmatrix}
6   &3  &1  &0 \\
3   &1  &0  &0 \\
1   &0  &0  &0 \\
\end{bmatrix} \ \text{and} \ 
M_3 = 
\begin{bmatrix}
3   &1  &0  \\
1   &0  &0  \\
\end{bmatrix}.
\]
The leading term of any determinant $\phi_n(P_I)$ is the antidiagonal term. In \cite{kim2015richardson}, the pairs of permutations $(v,w)$ for which $F_n|_w^v$ is toric is parametrized by so-called \emph{pipe dreams}. More precisely, each pair of reduced pipe dreams associated to $(v,w)$ gives rise to a face of the Gelfand-Tsetlin polytope. The toric varieties associated to these faces give rise to the (possibly irreducible) components of $X^v_w$ under the degeneration. Note that description in terms of pipe dreams cannot differentiate between toric and zero ideals and in some cases cannot determine whether the ideal is toric, non-toric or zero.

We calculate the ideals $F_n|_w^v$ for $\Flag_3$ and obtain the following results for pairs of permutation $(v,w)$ where $v < w$. 
\begin{center}
    \resizebox{0.6\textwidth}{!}{
    \begin{tabular}{|c|c|c|}
        \hline
        Toric               & Non-toric         &Zero \\
        \hline
       ((1,2,3),(3,1,2))    &((1,2,3),(2,3,1))  &((1,2,3),(1,3,2)) \\
       ((1,2,3),(3,2,1))    &((1,3,2),(2,3,1))  &((1,2,3),(2,1,3)) \\ 
       ((2,1,3),(3,1,2))    &((1,3,2),(3,1,2))  &((2,3,1),(3,2,1)) \\
       ((2,1,3),(3,2,1))    &((1,3,2),(3,2,1))  &((3,1,2),(3,2,1)) \\
                            &((2,1,3),(2,3,1))  &                  \\
        \hline
    \end{tabular}
    }
\end{center}
In addition to the results of Kim, we note that we are also able to characterize those pairs of permutations for which the ideal of the corresponding Richardson variety is zero under the degeneration. In these calculations, similar to the diagonal case in Theorem~\ref{thm:toric_fl_rich_diag}, the pairs $(v,w)$ for which $F_n|_w^v$ is zero satisfy: $w \in W_v$. 

\medskip

We perform the same calculation for $\Flag_4$ and obtain the above complete list of toric and zero pairs of permutations. The symbol $*$ appears in the table beside pairs of permutations for which the description by pipe dreams does not determine whether the corresponding ideal is toric, non-toric or zero. 

\begin{remark}
We note that our results determine pairs of permutations $(v,w)$ for which the corresponding Richardson variety degenerates to a toric variety. As seen in the above examples, it is not always possible to use pipe dreams, see \cite[Corollary V.26]{kim2015richardson}, to determine whether the Gelfand-Tsetlin degeneration gives rise to a toric degeneration of Richardson varieties. And so our results may be viewed as a strengthening of the previous results. 
\end{remark}
\begin{center}
    \resizebox{0.8\textwidth}{!}{
    \begin{tabular}{|c|c|c|}
        \hline
        \multicolumn{2}{|c|}{Toric}                           &Zero \\
        \hline
        ((1, 2, 3, 4), (1, 4, 2, 3)) \textcolor{white}{$*$}   &((2, 3, 1, 4), (4, 3, 1, 2)) \textcolor{white}{$*$}   &((1, 2, 3, 4), (1, 2, 4, 3)) \textcolor{white}{$*$} \\
        ((1, 2, 3, 4), (1, 4, 3, 2)) \textcolor{white}{$*$}   &((2, 3, 1, 4), (4, 3, 2, 1)) \textcolor{white}{$*$}   &((1, 2, 3, 4), (1, 3, 2, 4)) $*$ \\
        ((1, 2, 3, 4), (3, 1, 2, 4)) \textcolor{white}{$*$}   &((2, 3, 4, 1), (4, 2, 3, 1)) $*$                      &((1, 2, 3, 4), (2, 1, 3, 4)) \textcolor{white}{$*$} \\
        ((1, 2, 3, 4), (3, 2, 1, 4)) \textcolor{white}{$*$}   &((2, 3, 4, 1), (4, 3, 2, 1)) \textcolor{white}{$*$}   &((1, 2, 3, 4), (2, 1, 4, 3)) \textcolor{white}{$*$} \\
        ((1, 2, 3, 4), (4, 1, 2, 3)) \textcolor{white}{$*$}   &((3, 1, 2, 4), (4, 1, 2, 3)) \textcolor{white}{$*$}   &((1, 2, 4, 3), (2, 1, 4, 3)) $*$ \\
        ((1, 2, 3, 4), (4, 1, 3, 2)) \textcolor{white}{$*$}   &((3, 1, 2, 4), (4, 1, 3, 2)) \textcolor{white}{$*$}   &((1, 3, 4, 2), (1, 4, 3, 2)) $*$ \\
        ((1, 2, 3, 4), (4, 2, 1, 3)) \textcolor{white}{$*$}   &((3, 1, 2, 4), (4, 2, 1, 3)) \textcolor{white}{$*$}   &((1, 4, 2, 3), (1, 4, 3, 2)) $*$ \\
        ((1, 2, 3, 4), (4, 3, 1, 2)) \textcolor{white}{$*$}   &((3, 1, 2, 4), (4, 3, 1, 2)) \textcolor{white}{$*$}   &((2, 1, 3, 4), (2, 1, 4, 3)) \textcolor{white}{$*$} \\
        ((1, 2, 3, 4), (4, 3, 2, 1)) \textcolor{white}{$*$}   &((3, 1, 2, 4), (4, 3, 2, 1)) \textcolor{white}{$*$}   &((2, 3, 1, 4), (3, 2, 1, 4)) \textcolor{white}{$*$} \\
        ((1, 3, 2, 4), (1, 4, 2, 3)) $*$                      &((3, 1, 4, 2), (4, 1, 3, 2)) \textcolor{white}{$*$}   &((2, 3, 4, 1), (2, 4, 3, 1)) $*$\\
        ((1, 3, 2, 4), (1, 4, 3, 2)) $*$                      &((3, 2, 1, 4), (4, 2, 1, 3)) \textcolor{white}{$*$}   &((2, 3, 4, 1), (3, 2, 4, 1)) $*$\\
        ((2, 1, 3, 4), (3, 1, 2, 4)) \textcolor{white}{$*$}   &((3, 2, 1, 4), (4, 3, 1, 2)) \textcolor{white}{$*$}   &((3, 1, 2, 4), (3, 2, 1, 4)) \textcolor{white}{$*$} \\
        ((2, 1, 3, 4), (3, 2, 1, 4)) \textcolor{white}{$*$}   &((3, 2, 1, 4), (4, 3, 2, 1)) \textcolor{white}{$*$}   &((3, 4, 1, 2), (3, 4, 2, 1)) $*$ \\
        ((2, 1, 3, 4), (4, 1, 2, 3)) \textcolor{white}{$*$}   &((3, 2, 4, 1), (4, 2, 3, 1)) $*$                      &((3, 4, 1, 2), (4, 3, 1, 2)) \textcolor{white}{$*$} \\
        ((2, 1, 3, 4), (4, 1, 3, 2)) \textcolor{white}{$*$}   &((3, 2, 4, 1), (4, 3, 2, 1)) \textcolor{white}{$*$}   &((3, 4, 2, 1), (4, 3, 2, 1)) \textcolor{white}{$*$} \\
        ((2, 1, 3, 4), (4, 2, 1, 3)) \textcolor{white}{$*$}   &((4, 1, 2, 3), (4, 3, 1, 2)) \textcolor{white}{$*$}   &((4, 1, 2, 3), (4, 1, 3, 2)) \textcolor{white}{$*$} \\
        ((2, 1, 3, 4), (4, 3, 1, 2)) \textcolor{white}{$*$}   &((4, 1, 2, 3), (4, 3, 2, 1)) \textcolor{white}{$*$}   &((4, 1, 2, 3), (4, 2, 1, 3)) \textcolor{white}{$*$} \\
        ((2, 1, 3, 4), (4, 3, 2, 1)) \textcolor{white}{$*$}   &((4, 2, 1, 3), (4, 3, 1, 2)) \textcolor{white}{$*$}   &((4, 2, 3, 1), (4, 3, 2, 1)) $*$ \\
        ((2, 3, 1, 4), (2, 4, 1, 3)) \textcolor{white}{$*$}   &((4, 2, 1, 3), (4, 3, 2, 1)) \textcolor{white}{$*$}   &((4, 3, 1, 2), (4, 3, 2, 1)) \textcolor{white}{$*$} \\
        ((2, 3, 1, 4), (4, 2, 1, 3)) \textcolor{white}{$*$}   & &\\
        \hline
    \end{tabular}
    }
\end{center}

\subsection{Further connections.} 
\label{subsec:Connections}

We proceed by comparing our results to previous results in the literature. We highlight possible connections to other areas and future research opportunities.

\begin{remark}
In \cite{bossinger2018following}, the authors study the degeneration of Schubert varieties inside the full flag variety. Their construction is built upon on the flat degeneration of the flag variety given by Feigin \cite{Feigin2012} by restriction to its Schubert varieties. They give a number of sufficient conditions on the permutation $w \in S_n$ such that restriction of the degeneration to the Schubert variety $X(w)$ is reducible. Similarly to our methods, this is done by showing that the corresponding initial ideals contain monomials arising from certain Pl\"ucker relations.
\end{remark}

\begin{remark}
Here we mention possible connections of our results to representation theory. The Kazhdan-Lusztig polynomials play a key role in the representation theory of algebraic groups. The coefficients of these polynomials are closely related to the rank of intersection cohomology groups for Richardson varieties, see \cite[Theorem~4.3]{proudfoot2018algebraic}. In the case of projective toric varieties, the rank of the intersection cohomology groups can computed from the combinatorial data of the associated polytope $P$, namely in terms of the so-called $h$-functions of $P$, see \cite[Theorem~3.1]{stanley1987generalized} and \cite{fieseler1991rational}. And so we leave the study of the polytopes associated to toric degenerations of Richardson varieties to future work however we will note some particularly relevant results regarding matching field polytopes. 
In \cite{olliefatemehakihiro2020}, the authors show that block diagonal matching field polytopes, which coincide with the toric polytope for toric degenerations of the Grassmannian via matching fields, are related to each other by sequences of combinatorial mutations. Therefore one may naturally try to prove that the toric polytopes of toric degenerations of Richardson varieties are combinatorial mutations of faces of the Gelfand-Tsetlin polytope.
In such a case, by the semi-continuity arguments, we obtain information about Kazhdan-Lusztig polynomials from the toric degenerations of Richardson varieties.
\end{remark}

\begin{remark}
Here we mention some application of our results and methods to further families of varieties. In our paper we use standard monomial theory for Schubert and Richardson varieties inside Grassmannians and flag varieties, i.e. the monomial basis given by semi-standard Young tableaux which is compatible with these subvarieties. The desingularizations of Schubert varieties are known as Bott-Samelson varieties and a standard monomial theory for these varieties is developed in \cite{lakshmibai2002bottsamelson}. In \cite{balan2013standard}, the author develops standard monomial theory for desingularized Richardson varieties indexed by combinatorial objects called $w_0$-standard tableaux. Similarly one can try to develop analogues for matching fields for these desingularized varieties and study their standard monomials with maps analogous to $\Gamma_\ell$, see Definition~\ref{def:Gamma_ell_deg2}.
\end{remark}

\medskip

\noindent{\large\bf References}

\bibliographystyle{alpha} 
\bibliography{JACO-Trop1.bib}

\bigskip
\noindent
\footnotesize {\bf Authors' addresses:}

\medskip

\noindent Narasimha Chary Bonala\\ Ruhr-Universit\"at Bochum,
Fakult\"at f\"ur Mathematik, D-44780 Bochum, Germany
\\
\noindent  E-mail address: {\tt  narasimha.bonala@rub.de}
\medskip

\noindent Oliver Clarke\\ University of Bristol, School of Mathematics,
BS8 1TW, Bristol, UK
\\
\noindent  E-mail address: {\tt oliver.clarke@bristol.ac.uk}

\medskip

\noindent Fatemeh Mohammadi \\
Department of Mathematics: Algebra and Geometry, Ghent University, 9000 Ghent, Belgium \\
Department of Mathematics and Statistics, 
UiT – The Arctic University of Norway, 9037 Troms\o, Norway
\\ E-mail address: {\tt fatemeh.mohammadi@ugent.be}

\end{document}